\documentclass[11pt]{article}
\usepackage[utf8]{inputenc}

\usepackage{epsfig,epsf,fancybox}
\usepackage{amsmath}
\usepackage{mathrsfs,bbm}
\usepackage{amssymb}
\usepackage{graphicx}
\usepackage{color}
\usepackage{multirow}
\usepackage{paralist}
\usepackage{verbatim}
\usepackage{galois}
\usepackage{algorithm}
\usepackage{algorithmic}
\usepackage{boxedminipage}
\usepackage{booktabs}
\usepackage{accents}
\usepackage{stmaryrd}
\usepackage{subfig}
\usepackage{appendix}
\usepackage{graphicx}
\usepackage{epstopdf}
\usepackage{appendix}
\usepackage{amsthm}
\usepackage{cases}  
\usepackage[top=1in,bottom=1.2in,left=1in,right=1in,xetex]{geometry}
\usepackage{extarrows}

\usepackage{titlesec}
\usepackage{titletoc}

\usepackage[unicode=true,pdfusetitle,
bookmarks=true,bookmarksnumbered=true,bookmarksopen=true,bookmarksopenlevel=3,
breaklinks=false,pdfborder={0 0 1},backref=false,colorlinks=false]
{hyperref}

\numberwithin{equation}{section}

\newtheorem{theorem}{Theorem}[section]

\newtheorem{lemma}[theorem]{Lemma}

\newtheorem{condition}[theorem]{Condition}
\newtheorem{remark}[theorem]{Remark}

\newtheorem{example}{Example}[section]

\newcommand{\f}{\mathscr{F}}
\newcommand{\lr}{\mathcal{L}}
\newcommand{\sr}{\mathcal{S}}
\newcommand{\e}{\mathbb{E}}
\newcommand{\br}{\mathbb{R}}
\newcommand{\pr}{\mathcal{P}}
\newcommand{\dd}{\partial}

\newcommand{\brn}{{\mathbb{R}^n}}

\newcommand{\brd}{{\mathbb{R}^d}}
\newcommand{\de}{\Delta}

\newcommand{\hv}{\widehat{v}}

\newcommand{\tx}{\widetilde{X}}
\newcommand{\ty}{\widetilde{Y}}

\newcommand{\dr}{\mathcal{D}}
\newcommand{\bd}{\mathbb{D}}

\newcommand{\argmin}{\mathop{\rm argmin}}

\allowdisplaybreaks[2]

\title{A Class of Degenerate Mean Field Games, Associated FBSDEs and Master Equations}

\usepackage{authblk}

\author[a]{Alain Bensoussan\footnote{E-mail: axb046100@utdallas.edu}}
\author[b]{Ziyu Huang\footnote{E-mail: zyhuang19@fudan.edu.cn}}
\author[c]{Shanjian Tang\footnote{E-mail: sjtang@fudan.edu.cn}}
\author[b]{Sheung Chi Phillip Yam\footnote{E-mail: scpyam@sta.cuhk.edu.hk}}

\affil[a]{\small \it International Center for Decision and Risk Analysis, Naveen Jindal School of Management, University of Texas at Dallas}
\affil[b]{\small \it Department of Statistics and Data Science, The Chinese University of Hong Kong}
\affil[c]{\small \it Department of Finance and Control Sciences, School of Mathematical Sciences, Fudan University}

\date{}

\begin{document}

\maketitle

\begin{abstract}
In this paper, we study a class of degenerate mean field games (MFGs) with state-distribution dependent and unbounded functional diffusion coefficients. With a probabilistic method, we  study the well-posedness of the  forward-backward stochastic differential equations (FBSDEs) associated with the MFG and arising from the maximum principle,  and  estimate the corresponding Jacobian and Hessian flows. We further establish the classical regularity of the value functional $V$; in particular, we show that when the cost function is $C^3$ in the spatial and control variables and $C^2$ in the distribution argument, then the value functional is $C^1$ in time and $C^2$ in the spatial and distribution variables. As a consequence,  the value functional $V$ is the unique classical solution of the degenerate MFG master equation. \\



\end{abstract}

%\tableofcontents

%\newpage

\section{Introduction}\label{sec:intro}

Recently popular in mathematical sciences, mean field games (MFGs) were first proposed by Lasry and Lions in a series of papers \cite{CP,JM1,JM2,JM3} and also independently by Huang, Caines and Malhamé \cite{HM2,HM1,HM}. In the field of study, the controlled individual dynamical system is affected not only by the dynamical state and the control, but also by the equilibrium probability distribution of the state of the overall population. There are numerous works in various settings in this area. For analytical methods to second order mean field games, we refer to Bensoussan-Frehse-Yam \cite{AB_book}, Cardaliaguet-Cirant-Porretta \cite{CP1}, Cardaliaguet-Delarue-Lasry-Lions \cite{CDLL}, Cardaliaguet-Seeger-Souganidis \cite{Cardaliaguet-Seeger-Souganidis}, Gangbo-M\'esz\'aros-Mou-Zhang \cite{GW}, Gomes-Pimentel-Voskanyan \cite{GDA}, Graber-M\'{e}sz\'{a}ros \cite{GM}, Huang-Tang \cite{HZ} and Porretta \cite{PA}. For probabilistic approaches to mean field games and master equation, we refer to Ahuja-Ren-Yang \cite{SA1}, Bensoussan-Tai-Wong-Yam \cite{AB9'}, Bensoussan-Wong-Yam-Yuan \cite{AB10'}, Buckdahn-Li-Peng- Rainer \cite{BR}, Carmona-Delarue \cite{book_mfg}, Chassagneux-Crisan-Delarue \cite{CJF} and Huang-Tang \cite{HZ1}. For mean field type control problem and their master equations, we refer to Bensoussan-Huang-Yam \cite{AB6,AB8}, Bensoussan-Huang-Tang-Yam \cite{AB10}, Bensoussan-Tai-Yam \cite{AB5}, Bensoussan-Yam \cite{AB}, Carmona-Delarue \cite{CR} and  Ricciardi \cite{RM}.

%probabilistic approach
In this paper, we adopt a probabilistic approach and stochastic control method to study mean field game with a state-distribution dependent diffusion which can  be degenerate, and give the unique classical solution of the	 associated master equation. Let $(\Omega,\f,\{\f_t,0\le t\le T\},\mathbb{P})$ be a complete filtered probability space (with the filtration being  augmented by all the $\mathbb{P}$-null sets)  on which an $n$-dimensional Brownian motion $\{B_t,{0\le t\le T}\}$ is defined and is $\f_t$-adapted. We denote by $\mathcal{P}_{2}(\brn)$ the space of all probability measures of finite second-order  moments on $\brn$,  equipped with the 2-Wasserstein metric $W_2$; see Subsection~\ref{sec:notation} for details. For $(t,\mu)\in[0,T]\times\pr_2(\brn)$, we choose a random vector $\xi\in L_{\f_t}^2$ independent of the Brownian motion $\{B_s-B_t, t\le s\le T\}$ such that $\lr(\xi)=\mu$, and consider the following mean field game: 
\begin{equation}\label{intro_1}
	\left\{
	\begin{aligned}
		&v_{t\xi}(\cdot)\in\argmin_{v(\cdot)\in\mathcal{L}_{\mathscr{F}}^2(t,T)}J_{t\xi}\left(v(\cdot);m(s),t\le s\le T\right)\\
		&\qquad\qquad\qquad\qquad :=\e\left[\int_t^T f\left(s,X^v(s),m(s),v(s)\right)ds+g\left(X^v(T),m(T)\right)\right],\\
		&X^v(s)=\xi+\int_t^sb\left(r,X^v(r),m(r),v(r)\right)dr+\int_t^s\sigma\left(r,X^v(r),m(r),v(r)\right)dB(r),\\
		&m(s):=\mathcal{L}(X^{v_{t\xi}}(s)),\quad  s\in[t,T],
	\end{aligned}
	\right.
\end{equation}
for some regular enough functional coefficients:
\begin{align*}
	&b:[0,T]\times\brn\times\mathcal{P}_{2}(\brn)\times \brd\to \brn,\quad \sigma:[0,T]\times \brn\times\mathcal{P}_{2}(\brn)\times \brd\to \br^{n\times n},\\
	&f:[0,T]\times\brn\times\mathcal{P}_{2}(\brn)\times \brd\to \br,\quad g:\brn\times\mathcal{P}_{2}(\brn)\to \br.
\end{align*}
Under assumptions on coefficients $(b,\sigma,f,g)$ (see Section~\ref{sec:MP}), MFG \eqref{intro_1}  is expected to have a unique smooth enough feedback control $v_{t\xi} \in \mathcal{L}_{\mathscr{F}}^2(t,T)$.  Denote by  $Y_{t\xi}\in\sr_\f^2(t,T)$ the corresponding equilibrium state process for the mean field game, and by  $\left\{m^{t,\mu}(s):=\lr(Y_{t\xi}(s)),t\le s\le T\right\}$ the distribution flow of the equilibrium state process. Given the distribution flow $m^{t,\mu}$ and  $x\in\brn$, we then consider the following optimal stochastic control problem:
\begin{equation}\label{intro_1'}
	\left\{
	\begin{aligned}
		&v_{tx\mu}(\cdot)\in\argmin_{v(\cdot)\in\mathcal{L}_{\mathscr{F}}^2(t,T)}J_{tx}\left(v(\cdot);m^{t,\mu}\right)\\
		&\ \qquad\qquad\qquad\qquad :=\e\left[\int_t^T f\left(s,X^v(s),m^{t,\mu}(s),v(s)\right)ds+g\left(X^v(T),m^{t,\mu}(T)\right)\right],\\
		&X^v(s)=x+\int_t^sb\left(r,X^v(r),m^{t,\mu}(r),v(r)\right)dr+\int_t^s\sigma\left(r,X^v(r),m^{t,\mu}(r),v(r)\right)dB(r).
	\end{aligned}
	\right.
\end{equation}
Note that it is a standard stochastic control problem,  with the distribution of the equilibrium state of the MFG \eqref{intro_1} being autonomous and serving as a time-changing parameter. Therefore, it  is distinguished from  a McKean-Vlasov optimal stochastic control problem, %since its state process depends on the distribution of the equilibrium state of the MFG \eqref{intro_1}, 
where the distribution flow is specified by  the controlled state process of Problem \eqref{intro_1'} and is thus simultaneously controlled rather than being autonomous. Problem \eqref{intro_1'}  is expected to have a unique optimal control under assumptions (A1), (A2) and (A3) (see Section~\ref{sec:MP}). Since Problem \eqref{intro_1'} depends on $\xi$ only through its law $\mu$, it is reasonable to denote the optimal control by $v_{tx\mu}\in\mathcal{L}_{\mathscr{F}}^2(t,T)$. We shall use a probabilistic approach to study the regularity of the value functional of  Problem \eqref{intro_1'}:
\begin{equation}\label{intro_4}
	\begin{split}
		V(t,x,\mu):&=\inf_{v(\cdot)\in\lr_{\f}^2(t,T)}J_{tx}\left(v(\cdot);m^{t,\mu}\right)\\
		&=J_{tx}\left(v_{tx\mu}(\cdot);m^{t,\mu}\right),\quad (t,x,\mu)\in[0,T]\times\brn\times\pr_2(\brn),
	\end{split}	
\end{equation}
where $t$ is the initial time, $\mu$ is the initial condition of MFG \eqref{intro_1} and serves as an augmented infinite-dimensional state of  Problem \eqref{intro_1'}, and $x$ is the spatial  initial condition for Problem \eqref{intro_1'}. One of our main results asserts that, when the diffusion $\sigma$ does not depend on the control variable, the function ${V}$ is the unique classical solution of the following mean-field game master equation: 
\begin{equation}\label{intro_5}
	\left\{
	\begin{aligned}
		&\dd_t V(t,x,\mu)+H\left(t,x,\mu,D_x V(t,x,\mu),\frac{1}{2}D_x^2 V(t,x,\mu)\sigma(t,x,\mu)\right)\\
		&+\int_\brn \left\{ \left(D_p H \left(t,y,\mu,D_x V(t,y,\mu),\frac{1}{2}D_x^2 V(t,y,\mu)\sigma(t,y,\mu)\right)\right)^\top D_y\frac{dV}{d\nu}(t,x,\mu)({y})\right.\\
		&\quad\qquad\left. +\frac{1}{2}\text{Tr}\left[\left(\sigma\sigma^\top\right) (t,{y},\mu)D_y^2\frac{dV}{d\nu}(t,x,\mu)({y})\right]\right\}d\mu(y)=0,\quad t\in[0,T),\\ 
		&V(T,x,\mu)=g(x,\mu),\quad (x,\mu)\in \brn\times\pr_2(\brn),
	\end{aligned}	
	\right.
\end{equation}
where the Hamiltonian $H:[0,T]\times\brn\times\pr_2(\brn)\times\brn\times\br^{n\times n}\to\br$ is defined as
\begin{align}
		&H (s,x,m,p,q):=\inf_{v\in\brd} L\left(s,x,m,v,p,q\right),\label{bh'}\\
		&L(s,x,m,v,p,q):=p^\top  b(s,x,m,v)+\sum_{j=1}^n \left(q^j\right)^\top \sigma^j(s,x,m,v)+f(s,x,m,v), \label{H'}
\end{align}
and $\frac{dV}{d\nu}$ is the linear functional derivative of $V$ in the measure variable (see Section~\ref{sec:notation} below) . With $v^{t,\mu}(s,x):=V\left(s,x,m^{t,\mu}(s)\right)$ for $(s,x)\in[t,T]\times\brn$,  the pair $\left(m^{t,\mu},v^{t,\mu}\right)$ solves the system of HJB-FP equations associated to the classical mean field game with initial data $(x,\mu)$ at time $t$; see Remark~\ref{rk:FP-HJB}, and also  Section 2.2 of \cite{AB_JMPA} and Section 3.3 of \cite{HZ2} for more discussions. In fact, the mean field game master equation \eqref{intro_5} is a decoupling field of the mean field game HJB-FP system, and our probabilistic approach gives a classical solution of the master equation \eqref{intro_5}. The proof of the well-posedness of Equation \eqref{intro_5} relies heavily on the regularity  in $(x,\mu)$ of the value functional $V$. When the cost function is $C^3$ both in the spatial and control variables,  the value functional is shown to be also $C^2$ in its spatial variables; when the derivatives $D_y^2 \frac{d}{d\nu}D_x f,\ D_y^2 \frac{d}{d\nu}D_v f,$ and  $\ D_y^2 \frac{d}{d\nu}D_x g$ are bounded and continuous, then the value functional is linearly functionally differentiable in $\mu$ with the corresponding derivative $D_y^2\frac{dV}{d\nu}$ being continuous. The proof of these regularity requires that  $Y_{t\xi}$ has the G\^ateaux derivative  in $\xi$ and $(Y_{tx\mu},v_{tx\mu})$ has  the usual derivative in $(x,\mu)$. We first study the regularity solutions  to two systems of forward-backward stochastic differential equations (FBSDEs) corresponding to MFG \eqref{intro_1} and Problem \eqref{intro_1'}, which arise from the maximum principle \cite{AB_JMPA,AB8}. Our approach is significantly different from the recently well-received PDE approach to mean field games via master equation under Lasry-Lions monotonicity condition or displacement monotonicity condition in the existing literature. Our probabilistic approach is more aligned with the stochastic control method, and allows us to cope with the case where $b$ is linear in $(x,v)$ and nonlinearly dependent on $m$, while $\sigma$ can be degenerate so that it is linear in $x$ and nonlinearly dependent on $m$. We also require less regularity in the cost functions than those demanded in the existing literature. To the best of our knowledge, the degenerate mean field games are barely touched in the contemporary research, and most  studies require  the diffusion term to be constant, such as most of preceding works. When the diffusion coefficient is degenerate, the difficulty in the mathematical analysis is widely believed to be escalated to another higher level. Nevertheless, Chassagneux et al. \cite{CJF} study the case when $\sigma$ is bounded and lists out some interesting examples, but the linear volatility function is not included. Cardaliaguet et al. \cite{Cardaliaguet-Seeger-Souganidis} prove the existence of the weak solutions to the HJB-FP system corresponding to MFG when the noise coefficient is control-independent and the the initial distribution has a bounded density. Differently, in our work, we give the strong solution of the MFG \eqref{intro_1} and the classical solution of the master equation \eqref{intro_5} via a stochastic control approach. In particular, when the cost functions are $C^1$ and the diffusion coefficient $\sigma$ is state-distribution-control dependent, MFG \eqref{intro_1} can still be warranted with a unique solution. And our running cost function $f$ can be non-separable with a quadratic growth.

%continuation method for FBSDEs
For the well-posedness of the FBSDEs corresponding to MFG \eqref{intro_1} and Problem \eqref{intro_1'} and the regularity of their solutions, we use the continuation method for general FBSDEs defined on Hilbert spaces under a monotonicity condition, and then apply this result to the solvability of the associated FBSDEs \eqref{intro_2}, \eqref{lem:MP2}, \eqref{FB:dr} and \eqref{FB:dr'}. In a particular, the monotonicity condition here is satisfied when $f$ is strongly convex in $v$. This kind of continuation method was first proposed by Hu and Peng \cite{YH2}, and then extended by Peng and Wu \cite{YH2}. After that, it was utilized in FBSDEs for mean field type control problem by Carmona and Delarue \cite{CR}, and then used by Ahuja et al. \cite{SA1}, and Huang and Tang \cite{HZ1} for mean field games with common noises, and Bensoussan et al. \cite{AB10} for mean field type control problems. In our work, we apply the continuation method not only to the solvability of the FBSDEs corresponding to MFG \eqref{intro_1} and Problem \eqref{intro_1'}, but also to that of its Jacobian and Hessian flows together with the well-posedness of the linear functional derivatives of its solution, which goes beyond the framework of \cite{SA1,CR,HZ1}. The linear functional derivatives are also characterized as a system of FBSDEs, whose monotonicity conditions are derived from taking the derivatives of the usual optimality conditions. 

%convex condition
In this work, we establish both global solvability when $f$ is strongly convex and local solvability of all aforementioned FBSDEs for a general $f$. More precisely, the local solvability does not require $f$ to be convex in $x$, and we can assume the coefficients $(b,\sigma)$ to be linear in $(x,v)$ and nonlinearly dependent on $m$. To ensure the global solvability, we require $f$ to be convex in $(x,v)$ and $(b,\sigma)$ to be independent of $m$. Nevertheless, we allow $f$ to be non-separable in $v$ and $m$. This goes beyond the framework of \cite{CJF}. We also refer to Bensoussan et al. \cite{AB5,AB9'} for the relation between the convexity assumption and the displacement monotonicity condition and Lasry-Lions monotonicity condition extensively used in the literature \cite{SA1,CP,CDLL,book_mfg,CJF,HZ1}. 

The rest of the paper is organized as follows. In Section~\ref{sec:MP}, we give the well-posedness of FBSDEs corresponding to MFG \eqref{intro_1} and Problem \eqref{intro_1'} arising from the maximum principle. In Section~\ref{sec:distribution}, we compute the G\^ateaux derivatives with respect to the initial condition $\xi$ of the solution processes of FBSDEs associated to MFG \eqref{intro_1}. In Section~\ref{sec:state}, we compute the G\^ateaux derivatives in the initial $(x,\mu)$ of the solution of FBSDEs associated  to Problem \eqref{intro_1'}. In Section~\ref{sec:V}, we give the regularity of the value functional $V$. In Section~\ref{sec:master}, we eventually establish that $V$ is the unique classical solution of the master equation \eqref{intro_5}. %In Section~\ref{sec:LQ}, we address more explicitly  the linear-quadratic case. 
The statements in Sections~\ref{sec:MP}, \ref{sec:distribution}, \ref{sec:state}, \ref{sec:V} and \ref{sec:master} are proved in Appendices~\ref{pf:MP}, \ref{pf:distribution}, \ref{pf:state}, \ref{pf:V} and \ref{pf_thm1}, respectively.

\subsection{Notations}\label{sec:notation}

For any $X\in L^2(\Omega,\f,\mathbb{P};\brn)$, we denote by $\lr(X)$ its law. For every $t\in[0,T]$, we  denote by $L^2_{\f_t}$ the set of all $\f_t$-measurable square-integrable $\brn$-valued random vectors, and denote by $\lr^2_{\f}(0,T)$  the set of all $\f_t$-progressively-measurable $\brn$-valued processes $\alpha=\{\alpha_t,\ 0\le t\le T\}$ such that $\e\left[\int_0^T |\alpha_t|^2dt\right]<+\infty$. We  denote by $\mathcal{S}^2_{\f}(0,T)$ the set of all $\f_t$-progressively-measurable $\brn$-valued processes $\beta=\{\beta_t,\ 0\le t\le T\}$ such that $\e\left[\sup_{0\le t\le T} |\beta_t|^2\right]<+\infty$. We denote by $\mathcal{P}_{2}(\brn)$ the space of all probability measures of finite second-order moments on $\brn$, 
%i.e. $\int_{\brn}|x|^{2}dm(x)<+\infty$, 
equipped with the 2-Wasserstein metric: $W_2\left(m,m'\right):=\inf_{\pi\in\Gamma\left(m,m'\right)}\sqrt{\int_{\brn\times\brn}\left|x-x'\right|^2\pi\left(dx,dx'\right)}$, where $\Gamma\left(m,m'\right)$ is the set of joint probability measures with respective marginals $m$ and $m'$. We denote by $\delta_0$ the distribution of the random variable $\xi$ such that $\mathbb{P}(\xi=\mathbf{0})=1$. 
%Recall that a $\{m_{k},\ k\geq 1\}$ converges to $m$ in $\mathcal{P}_{2}(\brn)$ if and only if it converges in the sense of the weak convergence and $W_2(m_k,\delta_0)\to W_2(m,\delta_0)$ as $k\to+\infty$; 
See \cite{AL,VC} for more results on Wasserstein metric space.

%For the derivative in $m\in\pr_2(\brn)$, we use the concept of the linear functional derivative; see Carmona and Delarue \cite{book_mfg}. 
The linear functional derivative of a functional $k(\cdot):\mathcal{P}_{2}(\brn)\to\br$ at $m\in \mathcal{P}_{2}(\brn)$ is another functional $\mathcal{P}_{2}(\brn)\times \brn\ni(m,y)\mapsto\dfrac{dk}{d\nu}(m)(y)$,  being continuous under the product topology and satisfying $\int_{\brn}\Big|\dfrac{dk}{d\nu}(m)(y)\Big|^{2}dm(y)\leq c(m)$ for some positive constant $c(m)$ which is bounded on any bounded subsets of $\pr_2(\brn)$, such that 
$$
\lim_{\epsilon\to0}\dfrac{k((1-\epsilon)m+\epsilon m')-k(m)}{\epsilon}=\int_\brn\dfrac{dk}{d\nu}(m)(y)\left(dm'(y)-dm(y)\right), \quad \forall m'\in\mathcal{P}_{2}(\brn);
$$ we may refer to \cite{AB,book_mfg} for more details on linear functional derivatives. Particularly, the linear functional derivatives in $\pr_2(\brn)$ are connected to the G\^ateaux derivatives in $L^2(\Omega,\f,\mathbb{P};\brn)$ in the following way. For a linearly functional differentiable functional $k:\pr_2(\brn)\to\br$ such that the derivative $D_y\frac{d k}{d\nu}(\mu)(y)$ is continuous in $(\mu,y)$ and $D_y\frac{d k}{d\nu}(\mu)(y)\le c(\mu)(1+|y|)$ for $(\mu,y)\in\pr_2(\brn)\times\brn$,  the functional $K(X):=k(\lr(X)), \  X\in L^2(\Omega,\f,\mathbb{P};\brn)$ has the following G\^ateaux derivative
\begin{align}\label{lem01_1}
	D_X K(X)(\omega)=D_y\frac{d k}{d\nu}(\lr(X))(X(\omega)). 
\end{align}
moreover, when $X$ is the identity function $I$, \eqref{lem01_1} is identical to the $L$-derivative $\dd_m k(m)(x)$ of \cite{book_mfg}. 
%Similarly, for a differentiable functional ${k}:\brn\times\pr_2(\brn)\to\br$, we define ${K}:L^2(\Omega,\f,\mathbb{P};\brn)\to L^2(\Omega,\f,\mathbb{P};\brn)$ as
%\begin{align*}
	%{K}(X):={k}(X,\lr(X)), \quad X\in L^2(\Omega,\f,\mathbb{P};\brn),
%\end{align*}
%then, the G\^ateaux derivative of ${K}$ is given by
%\begin{align*}
	%D_X {K}(X)(Y)(\omega)=\left(D_x{k}(X,\lr(X)) \right)^\top  Y+\widetilde{\e}\left[\left(D_y\frac{d k}{d\nu}(X,\lr(X))\left(\tx\right) \right)^\top  \ty\right],
%\end{align*}
%where $(\tx,\ty)$ is an independent copy of $(X,Y)$; from now on, 

In this paper, for any random variable $\xi$, we use $\widetilde{\xi}$ to stand for its independent copy, and use $\widetilde{\e}[\widetilde{\xi}]$ for its expectation.  For convenience, we write $f|_{a}^b:=f(b)-f(a)$ for the difference of a  functional $f$ between two points $b$ and $a$. For any matrix $Q\in\br^{n\times n}$ and vector $x\in\brn$, we use the notation $Qx^{\otimes 2}:=x^\top Qx$.

\section{Forward-backward systems}\label{sec:MP}

In this section, we write down the maximum principle and associated FBSDEs for MFG \eqref{intro_1} and Problem \eqref{intro_1'}, respectively. We need the following assumptions. 

\textbf{(A1)} The functions $(b,\sigma)$ are linear in $(x,v)$. That is,
\begin{align*}
	&b(s,x,m,v)=b_0(s,m)+b_1(s)x+b_2(s)v,\\ 
	& \sigma^j(s,x,m,v)=\sigma^j_0(s,m)+\sigma^j_1(s)x+\sigma^j_2(s)v,\quad 1\le j\le n.
\end{align*}
Here, the maps $b_0,\sigma^j_0:[0,T]\times\pr_2(\brn)\to\brn$ are $L$-Lipschitz continuous in $m\in\pr_2(\brn)$; and the maps $b_1,\sigma^j_1:[0,T]\to\br^{n\times n}$ and $b_2,\sigma^j_2:[0,T]\to\br^{n\times d}$ are bounded  by $L$. Moreover, the derivatives $\left(D \frac{db_0}{d\nu}, D \frac{d\sigma^j_0}{d\nu}\right)$ exist, and they are continuous in all their arguments and are bounded by $L$.

\textbf{(A2)} The functions $f$ and $g$ have a  quadratic growth, and  satisfy for $(s,x,m,v)\in [0,T]\times\brn\times \pr_2(\brn)\times \brd$,
\begin{align*}
	|f(s,x,m,v)|\le L \left(1+|x|^2+W_2^2(m,\delta_0)+|v|^2\right),\quad |g(x,m)|\le L \left(1+|x|^2+W_2^2(m,\delta_0)\right).
\end{align*}
All the derivatives $D_x^2 f, \ D_vD_x f, \ D_v^2 f, \ D_y \frac{d}{d\nu}D_xf, \ D_y \frac{d}{d\nu}D_v f, \ D_x^2 g,\ D_y \frac{d}{d\nu}D_xg$ exist, and they are continuous in all their arguments, and they also satisfy for $m,m'\in \pr_2(\brn)$ and $y,y'\in\brn$,
\begin{align*}
	\left| \left(D_x^2,D_vD_x,D_v^2\right) f (s,x,m,v) \right|+\left| D_x^2 g(x,m) \right|\le\ & L,\\
	\left| D_y\frac{d f}{d \nu} (s,x,m',v)(y')-D_y\frac{d f}{d \nu} (s,x,m,v)(y)\right|&+\left| D_y\frac{d g}{d \nu} (x,m')(y')-D_y\frac{d g}{d \nu} (x,m)(y)\right|\\
	\le\ & L\left(W_2(m,m')+|y'-y|\right).
\end{align*}
Moreover, there exist nonnegative constants $L_x,L_v,L_g\le L$, such that
\begin{align*}
	&\left| D_y \frac{d}{d\nu} D_x f(s,x,m,v)(y)\right|\le L_x, \quad \left| D_y \frac{d}{d\nu} D_v f(s,x,m,v)(y)  \right|\le L_v,\quad \left| D_y \frac{d}{d\nu} D_x g(x,m)(y)\right|\le L_g.
\end{align*}

\textbf{(A3)} There exists $\lambda_v> 0$ such that for any $s\in[0,T]$ and  $(x,m,v,v')\in\brn\times\pr_2(\brn)\times\brd\times\brd$,
\begin{align*}
	&f(s,x,m,v')-f(s,x,m,v)\geq \left(D_v f (s,x,m,v)\right)^\top  (v'-v)+\lambda_v |v'-v|^2.
\end{align*}

\textbf{(A3')} The functions $(b_0,\sigma_0)$ are independent of $m$, and there exist $\lambda_v>0,\ \lambda_x\geq\frac{L_v^2}{8\lambda_v}+\frac{L_x}{2}$ and  $\lambda_g\geq\frac{L_g}{2} $ such that, for any $(s,m)\in[0,T]\times\pr_2(\brn)$, $x,x'\in\brn$ and $v,v'\in\brd$,
\begin{align*}
	&f\left(s,x',m,v'\right)-f(s,x,m,v)\geq \left(D_x f (s,x,m,v)\right)^\top  \left(x'-x\right)+\lambda_x \left|x'-x\right|^2\\
	&\qquad\qquad\qquad\qquad\qquad\qquad\qquad +\left(D_v f (s,x,m,v)\right)^\top  \left(v'-v\right)+\lambda_v \left|v'-v\right|^2,\\
	&g\left(x',m\right)-g(x,m)\geq \left(D_x g (x,m)\right)^\top  \left(x'-x\right)+\lambda_g \left|x'-x\right|^2.
\end{align*}

\begin{remark}
	For the convexity conditions in (A3), see \eqref{convex}-\eqref{convex'} for their roles. The condition $\lambda_x\geq\frac{L_v^2}{8\lambda_v}+\frac{L_x}{2}$ in (A3') is a kind of small mean field effect condition; see \eqref{small_mf_eff_3}, \eqref{lem3_6} and \eqref{small_mf_eff_1} for instance for its roles. And we refer to \cite{AB5,AB9'} for a detailed discussion on the relation among this small mean field effect assumption, the displacement monotonicity condition, and Lasry-Lions monotonicity condition. 
\end{remark}

For the Lagrangian $L:[0,T]\times \brn\times\pr_2(\brn)\times\brd\times\brn\times\br^{n\times n}\to\br$ defined in \eqref{H'}, we define the minimizing  $[0,T]\times \brn\times\pr_2(\brn)\times \brn\times\br^{n\times n}\ni(s,x,m,p,q)\mapsto \widehat{v}(s,x,m,p,q)\in\brd$ as
\begin{align}\label{hv'}
	D_v L \left(s,x,m,\widehat{v}(s,x,m,p,q),p,q\right)=0,
\end{align}
then, the Hamiltonian defined in \eqref{bh'} satisfies as
\begin{equation*}
	\begin{split}
		&H (s,x,m,p,q):=L\left(s,x,m,\widehat{v}(s,x,m,p,q),p,q\right).
	\end{split}
\end{equation*}
From Assumptions (A1) and (A3), we know that the minimizing map $\hv$ is well-defined and satisfies
\begin{align}\label{hv}
	[b_2(s)]^\top p+\sum_{j=1}^n\left[\sigma^j_2(s)\right]^\top q^j+D_vf\left(s,x,m,\hv(s,x,m,p,q)\right)=0.
\end{align}
From the convexity of $f$ in $v$ and the Lipchitz-continuity of $D_v f$ in $x$ and $m$, recalling in Assumption (A2) that $L_v\le L$, the map $\widehat{v}$ is $\frac{L}{2\lambda_v}$-Lipschitz continuous in $(x,m,p,q)\in \brn\times\pr_2(\brn)\times\brn\times \br^{n\times n}$; the proof of this claim is similar as \cite[Lemma 5.10]{AB6}, which is omitted here.

\subsection{Well-posedness of MFG \eqref{intro_1}}

We begin by giving a sufficient condition for MFG \eqref{intro_1}, whose proof is given in Appendix~\ref{pf_lem_MP1}. 

\begin{lemma}\label{lem:MP1}
	Under Assumptions (A1)-(A3), suppose that the following FBSDEs 
	\begin{equation}\label{intro_2}
		\left\{	
		\begin{aligned}
			&Y_{t\xi}(s)=\xi+\int_t^s D_pH\left(r,Y_{t\xi}(r),\lr(Y_{t\xi}(r)),p_{t\xi}(r),q_{t\xi}(r)\right)dr\\
			&\ \qquad\qquad+\int_t^s D_qH(r,Y_{t\xi}(r),\lr(Y_{t\xi}(r)),p_{t\xi}(r),q_{t\xi}(r))dB(r),\\
			&p_{t\xi}(s)=D_x g(Y_{t\xi}(T),\lr(Y_{t\xi}(T)))+\int_s^T D_x H(r,Y_{t\xi}(r),\lr(Y_{t\xi}(r)),p_{t\xi}(r),q_{t\xi}(r))dr\\
			&\ \qquad\qquad -\int_s^T q_{t\xi}(r)dB(r),\quad s\in[t,T],
		\end{aligned}
		\right.
	\end{equation}
	has a unique solution $(Y_{t\xi},p_{t\xi},q_{t\xi})\in \sr^2_\f(t,T)\times\lr^2_{\f}(t,T)\times\left(\lr^2_{\f}(t,T)\right)^n$. Then, there exists a constant $C(L,T)$ depending only on $(L,T)$, such that when $\lambda_v\geq C(L,T)$, MFG \eqref{intro_1} has a unique solution
	\begin{align}\label{vxi}
		v_{t\xi}(s):=\hv\left(s,Y_{t\xi}(s),\lr(Y_{t\xi}(s)),p_{t\xi}(s),q_{t\xi}(s)\right),\quad s\in[t,T].
	\end{align}
    Furthermore, if Assumption (A3') is satisfied, for any $\lambda_v> 0$, the same assertion remains valid.  
\end{lemma}

\begin{remark}
In Lemma~\ref{lem:MP1}, the processes $\left(p_{t\xi}, q_{t\xi}\right)$ are the first-order adjoint processes for MFG \eqref{intro_1}. The reason why we only involve the first-order adjoint processes is that our control domain $\brd$ is convex. Consequently, when deriving the necessary condition, we can apply the convex variational method (i.e., setting $v^\epsilon:=v+\epsilon (v'-v)$) instead of the spike variation method (i.e., $v^\epsilon(s):=v(s) \mathbbm{1}_{[t,T]\text{\textbackslash} E_\epsilon}(s)+ v'(s)\mathbbm{1}_{E_\epsilon}(s)$ where $E_\epsilon\subset [t,T]$ is a measurable set and the measure of $E_\epsilon$ is $\epsilon$). Therefore, we do not need to involve the second-order adjoint processes, which are required in non-convexity settings (see  \cite[Page 116, (3.9)]{MR1696772} for instance). Indeed, for the maximum principle of stochastic control problems, when the control domain is convex and all coefficients are $C^1$ in the control variable, the diffusion term can depend on the control, and the optimality condition does not require the second-order adjoint processes; more discussions can be found in \cite[Page 120, Case 2]{MR1696772}. We also refer to \cite[Page 138, Lemma 5.1 and Theorem 5.2]{MR1696772}, and the paragraph before Section 6 in \cite[Page 141]{MR1696772}, for the sufficient conditions of optimality for the stochastic control problems under convexity settings. For maximum principle in mean field theory, also see \cite[Theorem 8]{SA1} for maximum principle for MFG and \cite[Section 4]{CR} for maximum principle for mean field type control problems. 
\end{remark}

To give the well-posedness of FBSDEs \eqref{intro_2}, we first give the well-posedness of the following generic FBSDEs defined on Hilbert space of $L^2(\Omega,\f,\mathbb{P};\brn)$: for initial $(t,\xi)\in[0,T]\times  L_{\f_t}^2$,
\begin{equation}\label{FB:11}
	\left\{
	\begin{aligned}
		&Y(s)=\xi+\int_t^s \mathbf{B} (r,Y(r),p(r),q(r))dr+\int_t^s \mathbf{A}(r,Y(r),p(r),q(r))dB(r),\\
		&p(s)=\mathbf{G}(Y(T))-\int_s^T\mathbf{F}(r,Y(r),p(r),q(r))dr-\int_s^T q(r)dB(r),\quad s\in[t,T],
	\end{aligned}
	\right.
\end{equation}
where $\mathbf{B},\mathbf{F}:[0,T]\times (L^2\times L^2\times(L^2)^n)(\Omega,\f,\mathbb{P};\brn)\to L^2(\Omega,\f,\mathbb{P};\brn)$, $\mathbf{A}: [0,T]\times (L^2\times L^2\times(L^2)^n)(\Omega,\f,\mathbb{P};\brn)\to (L^2(\Omega,\f,\mathbb{P};\brn))^n$ and $\mathbf{G}:L^2(\Omega,\f,\mathbb{P};\brn)\to L^2(\Omega,\f,\mathbb{P};\brn)$. For the well-posedness of the generic FBSDEs \eqref{FB:11}, we normally adopt the following condition: 

\begin{condition}\label{Condition_generic}
	(i) There exists a map $\beta: [0,T]\times \left(L^2\times L^2\times(L^2)^n\right) (\Omega,\f,\mathbb{P};\brn)\ni(s,X,p,q)\mapsto\beta(s,X,p,q)\in L^2(\Omega,\f,\mathbb{P};\brd)$ and constants $\Lambda>0$, $\alpha\geq 0$, such that for any $X,X',p,p'\in L^2(\Omega,\f,\mathbb{P};\brn)$ and $q,q'\in \left(L^2(\Omega,\f,\mathbb{P};\brn)\right)^n$,
	\begin{align*}
		&\e\bigg[\left(\mathbf{F}(s,X',p',q')-\mathbf{F}(s,X,p,q)\right)^\top  (X'-X)+\left(\mathbf{B}(s,X',p',q')-\mathbf{B}(s,X,p,q)\right)^\top  (p'-p) \notag \\
		&\quad +\sum_{j=1}^n \left(\mathbf{A}^j(s,X',p',q')-\mathbf{A}^j(s,X,p,q)\right)^\top  \left({q'}^{j}-q^j\right)\bigg] \notag \\
		\le\ & \e\left[ -\Lambda \left|\beta(s,X',p',q')-\beta(s,X,p,q)\right|^2+\alpha\left(|X'-X|^2+|p'-p|^2+|q'-q|^2\right)\right]. %\label{monotonicity} 
	\end{align*}
	(ii) There exists a constant $K>0$, such that,
	\begin{align*}
		&\e\left[|\mathbf{B}(s,X',p',q')-\mathbf{B}(s,X,p,q)|^2+ |\mathbf{A}(s,X',p',q')-\mathbf{A}(s,X,p,q)|^2\right]  \notag\\
		\le\ &   K  \e \left[|X'-X|^2+| \beta(s,X',p',q')-\beta(s,X,p,q)|^2\right],\\
		&\e\left[|\mathbf{F}(s,X',p',q')-\mathbf{F}(s,X,p,q)|^2+|\mathbf{G}(X')-\mathbf{G}(X)|^2\right] \notag\\
		\le\ &   K \e \left[|X'-X|^2+|p'-p|^2+|q'-q|^2+| \beta(X',s;p',q')-\beta(X,s;p,q)|^2 \right]. 
	\end{align*}
    (iii) The coefficient function $\mathbf{G}$ satisfies 
    \begin{align*}
    	&\e\left[\left(\mathbf{G}(X')-\mathbf{G}(X)\right)^\top  (X'-X)\right]\geq 0,\quad X,X'\in L^2(\Omega,\f,\mathbb{P};\brn).
    \end{align*} 
\end{condition}

\begin{remark}
	In different situations, Condition \ref{Condition_generic} is verified under our Assumptions (A1)-(A3) all together with (A3'); see Lemmas~\ref{lem:2} and \ref{lem:3} for instance. 
\end{remark}

We now give the following well-posedness of FBSDEs \eqref{FB:11}.

\begin{lemma}\label{lem:1}
    Under Condition \ref{Condition_generic} ((i) and (ii)), there exists a constant $c(\alpha, K ,T)>0$ depending only on $(\alpha, K ,T)$, such that when $\Lambda\geq c(\alpha, K ,T)$, there is a unique adapted solution $(Y_{t\xi},p_{t\xi},q_{t\xi})$ of the FBSDEs \eqref{FB:11}. And for $\xi,\xi'\in L_{\f_t}^2$, we have     
    \small
    \begin{align}
    	&\e\Bigg[\sup_{t\le s\le T}\left|\left(Y_{t\xi}(s),p_{t\xi}(s)\right)^\top\right|^2+\int_t^T |q_{t\xi}(s)|^2ds\Bigg] \le C(\alpha, K ,T,\Lambda)\left(1+\e\left[|\xi|^2\right]\right), \label{lem1_1}\\
    	&\e\Bigg[\sup_{t\le s\le T}\left|\left(Y_{t\xi'}(s)-Y_{t\xi}(s),p_{t\xi'}(s)-p_{t\xi}(s)\right)^\top\right|^2 +  \int_t^T \left|q_{t\xi'}(s)-q_{t\xi}(s)\right|^2 ds\Bigg] \le C(\alpha, K ,T,\Lambda) \e\left[|\xi'- \xi|^2\right],\label{lem1_2}
    \end{align}
    \normalsize
    where $C(\alpha, K ,T,\Lambda)$ is a constant depending only on $(\alpha, K ,T,\Lambda)$. Furthermore, if the parameter $\alpha=0$ in Condition \ref{Condition_generic} (i) and Condition \ref{Condition_generic} (iii) is satisfied, then for any $\Lambda> 0$, FBSDEs \eqref{FB:11} has a unique adapted solution satisfying \eqref{lem1_1} and \eqref{lem1_2}.
\end{lemma}

This well-posedness result can be proven with the method of continuation in coefficients of \cite{YH2}, similar to the proof of \cite[Theorem 2.3]{SP},  \cite[Theorem 1]{SA1} and \cite[Lemma 4.1]{AB10}, which is omitted here. 

\begin{remark}
	Condition \ref{Condition_generic} (i) is actually the 'usual monotonicity condition' for FBSDEs in the literature. This monotonicity condition of the fully coupled FBSDEs in Euclidean spaces can be dated back to Hu-Peng \cite{YH2} and Peng-Wu \cite{SP}. It is also used as a weak monotonicity condition (displacement monotonicity condition) in Ahuja et al. \cite{SA1} for FBSDEs in Hilbert spaces, which can be applied for the FBSDEs arising from mean field games with common noise. Our Condition \ref{Condition_generic} (i) makes an attempt to further extend their results to a more general situation. For example, it is feasible to establish the unique existence of respective solutions of FBSDEs \eqref{intro_2} and \eqref{lem:MP2} (and those of their Jacobian and Hessian flows) under Assumptions (A1)-(A3) and (A3') by choosing a suitable map $\beta$ via \eqref{def:BAFG} and \eqref{chose_coefficients}, respectively. 
\end{remark}

Using Lemma~\ref{lem:1}, we give the well-posedness of FBSDEs \eqref{intro_2}, and also the $L^2$-boundedness and continuity of its solution with respect to the initial condition $\xi$. 

\begin{lemma}\label{lem:2}
	Under Assumptions (A1), (A2) and (A3), there exists a constant $c(L,T)>0$ depending only on $(L,T)$, such that when $\lambda_v\geq c(L,T)$, there is a unique adapted solution $(Y_{t\xi},p_{t\xi},q_{t\xi})$ of the FBSDEs \eqref{intro_2}. For $\xi,\xi'\in L_{\f_t}^2$, we then have	
	\small
	\begin{align}
		&\e\left[\sup_{t\le s\le T}\left|\left(Y_{t\xi}(s), p_{t\xi}(s)\right)^\top\right|^2+\int_t^T |q_{t\xi}(s)|^2ds \right]\le C(L,T,\lambda_v)\left(1+\e\left[|\xi|^2\right]\right), \label{lem2_1}\\
		&\e\left[\sup_{t\le s\le T}\left|\left(Y_{t\xi'}(s)-Y_{t\xi}(s), p_{t\xi'}(s)-p_{t\xi}(s)\right)^\top\right|^2 + \int_t^T \left|q_{t\xi'}(s)-q_{t\xi}(s)\right|^2 ds\right] \le C(L,T,\lambda_v)\e\left[|\xi'- \xi|^2\right]. \label{lem2_2}
	\end{align}	
	\normalsize
	Furthermore, if Assumption (A3') is satisfied, FBSDEs \eqref{intro_2} has a unique solution satisfying \eqref{lem2_1} and \eqref{lem2_2} for any $\lambda_v> 0$.	
\end{lemma}

Now,  FBSDEs \eqref{intro_2} can be viewed  as a subcase of FBSDEs \eqref{FB:11} in Lemma~\ref{lem:1} by setting
\begin{equation}\label{def:BAFG}
	\begin{split}
		&\mathbf{B}(s,X,p,q)(\omega):=D_p H(s,X(\omega),\lr(X),p(\omega),q(\omega)),\\
		&\mathbf{A}(s,X,p,q)(\omega):=D_q H(s,X(\omega),\lr(X),p(\omega),q(\omega)),\\
		&\mathbf{F}(s,X,p,q)(\omega):=-D_x H(s,X(\omega),\lr(X),p(\omega),q(\omega)),\\
		&\mathbf{G}(X)(\omega):=D_x g(X(\omega),\lr(X)),\quad X,p\in L^2\left(\Omega,\f,\mathbb{P};\brn\right),\quad q\in \left(L^2\left(\Omega,\f,\mathbb{P};\brn\right)\right)^n.
	\end{split}
\end{equation}
If we choose the map $\beta(s,X,p,q)(\omega):=\hv(s,X(\omega),\lr(X),p(\omega),q(\omega))$ and $\Lambda:=\lambda_v$, then, Condition \ref{Condition_generic} is fulfilled; the proof of Lemma~\ref{lem:2} is given in Appendix~\ref{pf_lem2}.

\subsection{Well-posedness of Problem \eqref{intro_1'}}

As long as $m$ is given, Problem \eqref{intro_1'} is certainly a classical stochastic control problem, meanwhile $Y_{t\xi}$ is a known process.  The following sufficient condition is a well-known result, so the proof is omitted.

\begin{lemma}\label{lem:MP2}
	Under Assumptions (A1), (A2) and (A3), suppose that the following FBSDEs 
	\begin{equation}\label{intro_3}
		\left\{	
		\begin{aligned}
			&Y_{tx\mu}(s)=x+\int_t^s D_pH\left(r,Y_{tx\mu}(r),\lr(Y_{t\xi}(r)),p_{tx\mu}(r),q_{tx\mu}(r)\right)dr\\
			&\ \ \; \qquad\qquad+\int_t^s D_qH\left(r,Y_{tx\mu}(r),\lr(Y_{t\xi}(r)),p_{tx\mu}(r),q_{tx\mu}(r)\right)dB(r),\\
			&p_{tx\mu}(s)=D_x g\left(Y_{tx\mu}(T),\lr(Y_{t\xi}(T))\right)+\int_s^T D_x H\left(r,Y_{tx\mu}(r),\lr(Y_{t\xi}(r)),p_{tx\mu}(r),q_{tx\mu}(r)\right)dr\\
			&\ \ \;\qquad\qquad -\int_s^T q_{tx\mu}(r)dB(r),\quad s\in[t,T],
		\end{aligned}
		\right.
	\end{equation}
	has a solution $(Y_{tx\mu},p_{tx\mu},q_{tx\mu})\in \sr^2_\f(t,T)\times\lr^2_{\f}(t,T)\times\left(\lr^2_{\f}(t,T)\right)^n$. Then, there exists a constant $c(L,T)>0$ depending only on $(L,T)$, such that when $\lambda_v\geq c(L,T)$, Problem \eqref{intro_1'} has a unique optimal control
    \begin{align}\label{vxmu}
    	v_{tx\mu}(s):=\hv\left(s,Y_{tx\mu}(s),\lr(Y_{t\xi}(s)),p_{tx\mu}(s),q_{tx\mu}(s)\right),\quad s\in[t,T].
    \end{align}
    Furthermore, if Assumption (A3') is satisfied, for any $\lambda_v> 0$, the same assertion holds.  
\end{lemma}

In FBSDEs \eqref{intro_3}, the system depends on $\xi$ only through its law $\mu$, so it is reasonable to use the subscript $\mu$ instead of $\xi$. We now give the well-posedness of FBSDEs \eqref{intro_3}, whose proof is given in Appendix~\ref{pf_lem5}. 

\begin{lemma}\label{lem:5}
	Under Assumptions (A1), (A2) and (A3), there exists a constant $c(L,T)>0$ depending only on $(L,T)$, such that when $\lambda_v\geq c(L,T)$, there is a unique adapted solution $(Y_{tx\mu},p_{tx\mu},q_{tx\mu})$ of FBSDEs \eqref{intro_3}. For $x,x'\in\brn$ and $\mu,\mu'\in\pr_2(\brn)$, we then have
	\begin{align}
		&\e\bigg[\sup_{t\le s\le T}\left|\left(Y_{tx\mu}(s), p_{tx\mu}(s)\right)^\top\right|^2+ \int_t^T |q_{tx\mu}(s)|^2ds\bigg]\le C(L,T,\lambda_v)\left(1+|x|^2+W^2_2(\mu,\delta_0)\right), \label{lem5_1}\\
		&\e\bigg[\sup_{t\le s\le T}\left|\left(Y_{tx'\mu'}(s)-Y_{tx\mu}(s), p_{tx'\mu'}(s)-p_{tx\mu}(s)\right)^\top\right|^2 + \int_t^T \left|q_{tx'\mu'} (s) -q_{tx\mu}(s)\right|^2 ds\bigg] \notag \\
		\le\ & C(L,T,\lambda_v)\left(|x'-x|^2+W_2^2\left(\mu,\mu'\right)\right).\label{lem5_2}
	\end{align}
    Furthermore, if Assumption (A3') is satisfied, FBSDEs \eqref{intro_3} has a unique solution satisfying \eqref{lem5_1} and \eqref{lem5_2} for any $\lambda_v> 0$.	
\end{lemma}

From Lemmas \ref{lem:MP2} and \ref{lem:5}, we know that the value functional $V$ defined in \eqref{intro_4} satisfies
\begin{equation}\label{intro_4'}
	\begin{split}
		V(t,x,\mu)=\e\left[\int_t^T f\left(s,Y_{tx\mu}(s),\lr(Y_{t\xi}(s)),v_{tx\mu}(s)\right)ds+g\left(Y_{tx\mu}(T),\lr(Y_{t\xi}(T))\right)\right].
	\end{split}	
\end{equation}
In the following sections, we shall study the G\^ateaux differentiability of $\left(Y_{t\xi}(s),p_{t\xi}(s),q_{t\xi}(s)\right)$ in $\xi$, and that of  $\left(Y_{tx\mu}(s),p_{tx\mu}(s),q_{tx\mu}(s)\right)$ in $(t,x,\mu)$, with which we can investigate the classical regularity of the value functional $V$.

\section{Derivatives of $\left(Y_{t\xi}(s),p_{t\xi}(s),q_{t\xi}(s)\right)$ in $\xi$}\label{sec:distribution}

In this section, we consider the G\^ateaux differentiability of $\left(Y_{t\xi}(s),p_{t\xi}(s),q_{t\xi}(s)\right)$ in the initial $\xi\in L_{\f_t}^2$. For the sake of convenience, in the rest of this article, we denote by $\theta=(x,m,v)\in \brn\times\pr_2(\brn)\times\brd$ and $\Theta=(x,m,p,q)\in \brn\times\pr_2(\brn)\times\brn\times\br^{n\times n}$. We denote by $\theta_{t\xi}(s)$ the processes $(Y_{t\xi}(s),\lr(Y_{t\xi}(s)),v_{t\xi}(s))$ and $\Theta_{t\xi}(s)$ the processes $(Y_{t\xi}(s),\lr(Y_{t\xi}(s)),p_{t\xi}(s),q_{t\xi}(s))$, where $v_{t\xi}(s)$ is defined in \eqref{vxi}. In view of Assumption (A2) and (A3) and the Young's inequality with a weight of $\frac{\lambda_v}{L}$, for any $x,y\in\brn$, $v,u\in\brd$ and $(s,m)\in[0,T]\times\pr_2(\brn)$,
\begin{equation}\label{convex}
	\begin{split}
		\left[\left(
		\begin{array}{cc}
			D_v^2 f & D_vD_x f\\
			D_xD_v f & D_x^2 f
		\end{array}
		\right)(s,x,m,v)\right] \left(
		\begin{array}{cc}
			u\\
			y
		\end{array}
		\right)^{\otimes 2}\geq \ & 2\lambda_v |u|^2-2L|x||u|-L|y|^2\\
		\geq\ & \lambda_v |u|^2-\left(L+\frac{L^2}{\lambda_v}\right) |y|^2.
	\end{split}
\end{equation}
Furthermore, if Assumption (A3') is satisfied, then,
\begin{equation}\label{convex'}
	\begin{split}
		\left[\left(
		\begin{array}{cc}
			D_v^2 f & D_vD_x f\\
			D_xD_v f & D_x^2 f
		\end{array}
		\right)(s,x,m,v) \right] \left(
		\begin{array}{cc}
			u\\
			y
		\end{array}
		\right)^{\otimes 2} & \geq  2\lambda_v |u|^2+2\lambda_x |y|^2,\\
		y^\top  D_x^2 g(x,m)y &\geq 2\lambda_g |y|^2.
	\end{split}
\end{equation}
From the linearity of $b$ and $\sigma$ in $v$ in Assumption (A1), the differentiability of $D_vf$ in $(x,m)$ in Assumption (A2), the strong convexity of $f$ in $v$, and the application of the implicit function theorem, we know that the map $\widehat{v}$ is continuously differentiable in $(x,p,q)$ and linearly functionally differentiable in $m$ (also see \cite[Page 29]{AB9'}). By taking the derivative of Equation \eqref{hv} with respect to $x$, $m$, $p$ and $q$, respectively, we know that
\begin{equation}\label{optimal_condition}
	\begin{split}
		0=\ &\left[D_xD_vf (s,x,m,\hv(s,\Theta))\right]^\top+D_v^2f(s,x,m,\hv(s,\Theta))\left[D_x\hv (s,\Theta)\right]^\top;\\
		0=\ &\left[D_y\frac{d}{d\nu}D_vf (s,x,m,\hv(s,\Theta))(y)\right]^\top+D_v^2f(s,x,m,\hv(s,\Theta))\left[D_y\frac{d\hv}{d\nu} (s,\Theta)(y)\right]^\top;\\
		0=\ &\left[b _2(s)\right]^\top+D_v^2f(s,x,m,\hv(s,\Theta))\left[D_p\hv (s,\Theta)\right]^\top;\\
		0=\ &\left[\sigma^j_2 (s)\right]^\top+D_v^2f(s,x,m,\hv(s,\Theta))\left[D_{q^j}\hv (s,\Theta)\right]^\top,\quad 1\le j\le n.
	\end{split}
\end{equation}
For any $\eta\in L_{\f_t}^2$, consider the following FBSDEs:
\small
\begin{align}
		\dr_\eta Y_{t\xi}(s)=\ & \eta+\int_t^s \bigg\{\widetilde{\e}\left[ \left[D \frac{db_0}{d\nu} (r,\lr(Y_{t\xi}(r)))\left(\widetilde{Y_{t\xi}}(r)\right) \right]^\top \!\! \widetilde{\dr_\eta Y_{t\xi}}(r)\right]+b_1(r)\dr_\eta Y_{t\xi}(r)+b_2(r)\dr_\eta v_{t\xi}(r)\bigg\}dr \notag \\
		&+\int_t^s \bigg\{\widetilde{\e}\left[\left[D \frac{d\sigma_0}{d\nu} (r,\lr(Y_{t\xi}(r)))\left(\widetilde{Y_{t\xi}}(r)\right)\right]^\top \!\! \widetilde{\dr_\eta Y_{t\xi}}(r)\right] +\sigma_1(r)\dr_\eta Y_{t\xi}(r)+\sigma_2(r)\dr_\eta v_{t\xi}(r)\bigg\}dB(r), \notag \\
		\dr_\eta p_{t\xi}(s)=\ & \left[D_x^2 g (Y_{t\xi}(T),\lr(Y_{t\xi}(T)))\right]^\top  \dr_\eta Y_{t\xi}(T) \notag \\
		&+\widetilde{\e}\left[\left[D_y \frac{d}{d\nu}D_x g (Y_{t\xi}(T),\lr(Y_{t\xi}(T)))\left(\widetilde{Y_{t\xi}}(T)\right) \right]^\top  \widetilde{\dr_\eta Y_{t\xi}}(T)\right] \notag \\
		&+\int_s^T\bigg\{\left[b _1(r)\right]^\top \dr_\eta p_{t\xi}(r)+\sum_{j=1}^n\left[\sigma_1^j (r)\right]^\top \dr_\eta q^j_{t\xi}(r)+\left[D_x^2 f  (r,\theta_{t\xi}(r)) \right]^\top \dr_\eta Y_{t\xi}(r) \notag \\
		&\qquad\qquad +\widetilde{\e}\left[\left[D_y \frac{d}{d\nu}D_x f  (r,\theta_{t\xi}(r))\left(\widetilde{Y_{t\xi}}(r)\right) \right]^\top \widetilde{\dr_\eta Y_{t\xi}}(r)\right]  +\left[D_vD_x f (r,\theta_{t\xi}(r)) \right]^\top  \dr_\eta v_{t\xi}(r)\bigg\}dr \notag \\
		&-\int_s^T\dr_\eta q_{t\xi}(r)dB(r), \quad s\in[t,T], \label{FB:dr}
\end{align}
\normalsize
where
\begin{align*}
	\dr_\eta v_{t\xi}(s):=\ & \left[D_x\hv  (s,\Theta_{t\xi}(s))\right]^\top \dr_\eta Y_{t\xi}(s) +\widetilde{\e}\left[\left[D_y\frac{d\hv}{d\nu} (s,\Theta_{t\xi}(s))\left(\widetilde{Y_{t\xi}}(s)\right) \right]^\top  \widetilde{\dr_\eta Y_{t\xi}}(s) \right]\\
	&+\left[D_p\hv  (s,\Theta_{t\xi}(s))\right]^\top\dr_\eta p_{t\xi}(s)+\sum_{j=1}^n \left[D_{q^j}\hv (s,\Theta_{t\xi}(s))\right]^\top \dr_\eta q^j_{t\xi}(s),\quad s\in[t,T],
\end{align*}
and $\left(\widetilde{Y_{t\xi}}(s),\widetilde{\dr_\eta Y_{t\xi}}(s)\right)$ is an independent copy of $\left(Y_{t\xi}(s),\dr_\eta Y_{t\xi}(s)\right)$. By using Lemma~\ref{lem:1}, it comes with the following solvability of FBSDEs \eqref{FB:dr} and the boundedness of the solution.

\begin{lemma}\label{lem:3}
	Under Assumptions (A1), (A2) and (A3), there exists a constant $c(L,T)$ depending only on $(L,T)$, such that when $\lambda_v\geq c(L,T)$, there is a unique adapted solution $(\dr_\eta Y_{t\xi},\dr_\eta p_{t\xi},\dr_\eta q_{t\xi})$ of FBSDEs \eqref{FB:dr}, such that
	\begin{align}\label{lem3_1}
		\e\left[\sup_{t\le s\le T}\left|\left(\dr_\eta Y_{t\xi}(s),\dr_\eta p_{t\xi}(s)\right)^\top\right|^2+ \int_t^T\left|\dr_\eta q_{t\xi}(s)\right|^2 ds\right]\le C(L,T,\lambda_v)\e\left[|\eta|^2\right].
	\end{align}
    Furthermore, if Assumption (A3') is satisfied, FBSDEs \eqref{FB:dr} has a unique solution satisfying \eqref{lem3_1} for any $\lambda_v> 0$.
\end{lemma}

Note the fact that $\Theta_{t\xi}$ is already known, and \eqref{FB:dr} is a system of FBSDEs for $(\dr_\eta Y_{t\xi},\dr_\eta p_{t\xi},\dr_\eta q_{t\xi})$. Actually,  FBSDEs \eqref{FB:dr} can be viewed  as a subcase of FBSDEs \eqref{FB:11} in Lemma~\ref{lem:1} by setting: for $(s,X,p,q)\in [0,T]\times L^2(\Omega,\f,\mathbb{P};\brn)\times L^2(\Omega,\f,\mathbb{P};\brn)\times(L^2(\Omega,\f,\mathbb{P};\brn))^n$, 
\begin{align}
	&\mathbf{B}(s,X,p,q):=\widetilde{\e}\left[\left[D \frac{db_0}{d\nu} (s,\lr(Y_{t\xi}(s)))\left(\widetilde{Y_{t\xi}}(s)\right) \right]^\top \tx\right]+b_1(s)X+b_2(s)\mathbf{V}(s,X,p,q), \notag\\
	&\mathbf{A}(s,X,p,q):=\widetilde{\e}\left[\left[D \frac{d\sigma_0}{d\nu} (\lr(Y_{t\xi}(s)),s)\left(\widetilde{Y_{t\xi}}(s)\right) \right]^\top \tx\right]+\sigma_1(s)X+\sigma_2(s)\mathbf{V}(s,X,p,q), \notag \\
	&\mathbf{F}(s,X,p,q):=-[b _1(s)]^\top p-\sum_{j=1}^n\left[\sigma_1^j (s)\right]^\top q^j-\left[D_x^2 f (s,\theta_{t\xi}(s)) \right]^\top X \notag\\
	&\ \ \qquad\qquad\qquad  -\widetilde{\e}\left[\left[D_y \frac{d}{d\nu}D_x f  (s,\theta_{t\xi}(s))\left(\widetilde{Y_{t\xi}}(s)\right) \right]^\top \tx \right]-\left[D_vD_x f  (s,\theta_{t\xi}(s))\right]^\top \mathbf{V}(s,X,p,q), \notag\\
	&\mathbf{G}(X):=\left[D_x^2 g  (Y_{t\xi}(T),\lr(Y_{t\xi}(T))) \right]^\top X+\widetilde{\e}\left[\left[D_y \frac{d}{d\nu}D_x g  (Y_{t\xi}(T),\lr(Y_{t\xi}(T)))(\widetilde{Y_{t\xi}}(T)) \right]^\top \tx\right], \label{chose_coefficients}
\end{align}  
\normalsize  
where
\begin{align*}
	\mathbf{V}(s,X,p,q):=&\left[D_x\hv (s,\Theta_{t\xi}(s)) \right]^\top X +\widetilde{\e}\left[\left[D_y\frac{d\hv}{d\nu} (s,\Theta_{t\xi}(s))\left(\ty_{t\xi}(s)\right) \right]^\top  \tx\right]\\
	&+\left[D_p\hv (s,\Theta_{t\xi}(s)) \right]^\top p +\sum_{j=1}^n \left[D_{q^j}\hv  (s,\Theta_{t\xi}(s)) \right]^\top q^j,
\end{align*}
\normalsize
and $(\widetilde{Y_{t\xi}}(s),\tx)$ is an independent copy of $(Y_{t\xi}(s),X)$. By choosing the map $\beta(s,X,p,q):=\mathbf{V}(s,X,p,q)$, then, the corresponding Condition \ref{Condition_generic} holds. The full proof is given in Appendix~\ref{pf_lem3}. We next establish that the components of $(\dr_\eta Y_{t\xi}(s),\dr_\eta p_{t\xi}(s),\dr_\eta q_{t\xi}(s))$ are the G\^ateaux derivatives of $(Y_{t\xi}(s),p_{t\xi}(s),q_{t\xi}(s))$ with respect to the initial $\xi$ along the direction $\eta$.

\begin{theorem}\label{lem:4}
	Under Assumptions (A1), (A2) and (A3), there exists a constant $c(L,T)>0$ depending only on $(L,T)$, such that when $\lambda\geq c(L,T)$, for any $\eta\in L_{\f_t}^2$, we have
	\begin{equation}\label{lem4_1}
		\begin{split}
			\lim_{\epsilon\to0}\e\Bigg[& \sup_{t\le s\le T}\left|\frac{ Y_{ t\xi^\epsilon}(s)-Y_{t\xi}(s)}{\epsilon}-\dr_\eta Y_{t\xi}(s)\right|^2+ \sup_{t\le s\le T}\left|\frac{ p_{ t\xi^\epsilon}(s)-p_{t\xi}(s)}{\epsilon}-\dr_\eta p_{t\xi}(s)\right|^2\\
			&+ \int_t^T \left|\frac{ q_{ t\xi^\epsilon}(s)-q_{t\xi}(s)}{\epsilon}-\dr_\eta q_{t\xi}(s)\right|^2 ds\Bigg]=0,
		\end{split}
	\end{equation}
	where $\xi^\epsilon:=\xi+\epsilon\eta$ for $\epsilon\in(0,1)$. That is, the components of $(\dr_\eta Y_{t\xi}(s),\dr_\eta p_{t\xi}(s),\dr_\eta q_{t\xi}(s))$ defined in \eqref{FB:dr} are the G\^ateaux derivatives of $(Y_{t\xi}(s),p_{t\xi}(s),q_{t\xi}(s))$ in $\xi$ along the direction $\eta$. Moreover, the G\^ateaux derivatives satisfy the boundedness condition \eqref{lem3_1}, and they are linear in $\eta$ and continuous in $\xi$. Furthermore, if Assumption (A3') is satisfied, for any $\lambda_v> 0$, the same assertion still holds.  	
\end{theorem}

Theorem \ref{lem:4} is proven in Appendix~\ref{pf_lem4}; we also refer to \cite{AB5} for further discussion on their very nature as Fr\'echet derivatives.

\section{Derivatives of $(Y_{tx\mu}(s),p_{tx\mu}(s),q_{tx\mu}(s))$ in $(x,\mu)$}\label{sec:state}

In this section, we study the regularity of $(Y_{tx\mu}(s),p_{tx\mu}(s),q_{tx\mu}(s))$ in $(x,\mu)$, defined in \eqref{intro_3}, with respect to initial data $(x,\mu)\in\brn\times\pr_2(\brn)$.  For the sake of convenience, in the rest of this paper, we denote by $\theta_{tx\mu}$ the processes $(Y_{tx\mu},\lr(Y_{t\xi}),v_{tx\mu})$ and $\Theta_{tx\mu}$ the processes $(Y_{tx\mu},\lr(Y_{t\xi}),p_{tx\mu},q_{tx\mu})$, where $\lr(\xi)=\mu$, $Y_{t\xi}$ is defined in FBSDEs \eqref{intro_2} and $v_{tx\mu}$ is defined in \eqref{vxmu}. Since the processes $(Y_{tx\mu},\lr(Y_{t\xi}),v_{tx\mu},p_{tx\mu},q_{tx\mu})$ depend on $\xi$ only through its law $\mu$, it is natural to use the subscript $\mu$ in $\theta_{tx\mu}$ and $\Theta_{tx\mu}$.

\subsection{Jacobian flow of first-order derivatives}\label{subsec:1-order}

We first consider the differentiability of $(Y_{tx\mu}(s),p_{tx\mu}(s),q_{tx\mu}(s))$ in the initial state $x\in\brn$. Consider the following FBSDEs: 
\begin{align}
		D_xY_{tx\mu}(s)=\ & I+\int_t^s \left[b_1(r)D_xY_{tx\mu}(r)+b_2(r)D_x v_{tx\mu}(r)\right]dr \notag \\
		&\;\; +\int_t^s \left[\sigma_1(r)D_xY_{tx\mu}(r)+\sigma_2(r)D_x v_{tx\mu}(r)\right]dB(r),\notag \\
		D_xp_{tx\mu}(s)=\ & \left[D_x^2 g \left(Y_{tx\mu}(T),\lr(Y_{t\xi}(T))\right) \right]^\top D_x Y_{tx\mu}(T) \notag \\
		& +\int_s^T \bigg[b _1(r)^\top D_xp_{tx\mu}(r)+\sum_{j=1}^n \left(\sigma_1^j (r)\right)^\top D_x q^j_{tx\mu}(r) \notag \\
		&\quad\qquad+\left[D_x^2 f (r,\theta_{tx\mu}(r)) \right]^\top D_x Y_{tx\mu}(r) + \left[D_vD_x f (r,\theta_{tx\mu}(r))\right]^\top D_x v_{tx\mu}(r)\bigg]dr \notag \\
		&  -\int_s^T D_x q_{tx\mu}(r)dB(r),\quad s\in[t,T], \label{FB:x}
\end{align}
where $I$ is the identity matrix and 
\begin{align*}
	D_x v_{tx\mu}(s):=&\left[D_x\hv (s,\Theta_{tx\mu}(s))\right]^\top D_x Y_{tx\mu}(s)+\left[D_p\hv (s,\Theta_{tx\mu}(s)) \right]^\top D_x p_{tx\mu}(s)\\
	&+\sum_{j=1}^n \left[D_{q^j}\hv (s,\Theta_{tx\mu}(s)) \right]^\top D_x q^j_{tx\mu}(s),\quad s\in[t,T].
\end{align*}
Since the FBSDEs \eqref{FB:x} depend on $\xi$ only through its law $\mu$, it is natural to use the subscript $\mu$ in $(D_x Y_{tx\mu},D_x v_{tx\mu},D_x p_{tx\mu},D_x q_{tx\mu})$. The following result gives the solvability of FBSDEs \eqref{FB:x} and shows that the components of $(D_xY_{tx\mu}(s), D_xp_{tx\mu}(s),D_xq_{tx\mu}(s))$ are  the respective G\^ateaux derivatives of $(Y_{tx\mu}(s),p_{tx\mu}(s),q_{tx\mu}(s))$ in $x\in\brn$. Its proof is similar as that of Lemma~\ref{lem:3} and Theorem~\ref{lem:4}, which is omitted here.
\begin{theorem}\label{prop:2}
	Under Assumptions (A1), (A2) and (A3), there exists a constant $c(L,T)$ depending only on $(L,T)$, such that when $\lambda_v\geq c(L,T)$, there is a unique solution $(D_x Y_{tx\mu},D_x p_{tx\mu},D_x q_{tx\mu})$ of FBSDEs \eqref{FB:x}, and the components of $(D_x Y_{tx\mu}(s),D_x p_{tx\mu}(s),D_x q_{tx\mu}(s))$ are the G\^ateaux derivatives of $(Y_{tx\mu}(s),p_{tx\mu}(s),q_{tx\mu}(s))$ in $x$. Moreover, the G\^ateaux derivatives satisfy the following boundedness property:
	\begin{equation}\label{thm2_1}
		\begin{split}
			&\e\bigg[\sup_{t\le s\le T}|D_x Y_{tx\mu}(s)|^l\bigg]^{\frac{1}{l}}+\e\bigg[\sup_{t\le s\le T}|D_x p_{tx\mu}(s)|^l\bigg]^{\frac{1}{l}}+\left[\int_t^T\e\left[|D_x q_{tx\mu}(s)|^l\right]^{\frac{2}{l}} ds\right]^{\frac{1}{2}}\\
			\le\ & C(L,T,\lambda_v),\quad l=2,4,
		\end{split}
	\end{equation}
	and are continuous in $(x,\mu)$. Furthermore, if Assumption (A3') is satisfied, for any $\lambda_v> 0$, the same assertion holds.  	
\end{theorem}

Recall that $(Y_{t\xi},p_{t\xi},q_{t\xi})$ is the solution of FBSDEs \eqref{intro_2}. We consider the following FBSDEs: for $s\in[t,T]$, 
\small
\begin{align}
		\bd_\eta Y_{t\xi}(s)=\ & \int_t^s \bigg\{ \widetilde{\e}\left[\left[D \frac{db_0}{d\nu} (r,\lr(Y_{t\xi}(r)))\left(\widetilde{Y_{t\xi}}(r)\right) \right]^\top  \left(\widetilde{\bd_\eta Y_{t\xi}}(r)+\widetilde{D_xY_{tx\mu}}(r)\left|_{x=\widetilde{\xi}}\ \right. \widetilde{\eta}\right)\right] \notag \\
		&\ \qquad +b_1(r)\bd_\eta Y_{t\xi}(r)+b_2(r)\bd_\eta v_{t\xi}(r)\bigg\}dr \notag \\
		& +\int_t^s \bigg\{ \widetilde{\e}\left[\left[D \frac{d\sigma_0}{d\nu}  (r,\lr(Y_{t\xi}(r)))\left(\widetilde{Y_{t\xi}}(r)\right)\right]^\top  \left(\widetilde{\bd_\eta Y_{t\xi}}(r)+\widetilde{D_xY_{tx\mu}}(r)\left|_{x=\widetilde{\xi}}\ \right. \widetilde{\eta}\right)\right] \notag \\
		&\ \qquad +\sigma_1(r)\bd_\eta Y_{t\xi}(r)+\sigma_2(r)\bd_\eta v_{t\xi}(r)\bigg\} dB(r), \notag \\
		\bd_\eta p_{t\xi}(s)=\ & \left[D_x^2 g (Y_{t\xi}(T),\lr(Y_{t\xi}(T)))\right]^\top \bd_\eta Y_{t\xi}(T) \notag \\
		&+\widetilde{\e}\left[\left[\left(D_y \frac{d}{d\nu}D_x g\right) (Y_{t\xi}(T),\lr(Y_{t\xi}(T)))\left(\widetilde{Y_{t\xi}}(T)\right) \right]^\top \left(\widetilde{\bd_\eta Y_{t\xi}}(T)+\widetilde{D_xY_{tx\mu}}(T)\left|_{x=\widetilde{\xi}}\ \right. \widetilde{\eta}\right)\right] \notag \\
		& +\int_s^T\bigg\{ b _1(r)^\top \bd_\eta p_{t\xi}(r)+\sum_{j=1}^n\left({\sigma_1^j} (r)\right)^\top \bd_\eta q^j_{t\xi}(r)+\left[D_x^2 f  (r,\theta_{t\xi}(r))\right]^\top  \bd_\eta Y_{t\xi}(r) \notag \\
		&\ \quad\qquad +\widetilde{\e}\left[\left[\left(D_y \frac{d}{d\nu}D_x f\right)  (r,\theta_{t\xi}(r))\left(\widetilde{Y_{t\xi}}(r)\right)\right]^\top  \left(\widetilde{\bd_\eta Y_{t\xi}}(r)+\widetilde{D_xY_{tx\mu}}(r)\left|_{x=\widetilde{\xi}}\ \right. \widetilde{\eta}\right)\right] \notag \\
		&\ \quad\qquad +\left[D_vD_x f (r,\theta_{t\xi}(r)) \right]^\top  \bd_\eta v_{t\xi}(r)\bigg\}dr-\int_s^T\bd_\eta q_{t\xi}(r)dB(r), \label{FB:dr'}
\end{align}
\normalsize
where
\begin{align*}
	\bd_\eta v_{t\xi}(s):=&\left[ D_x\hv (s,\Theta_{t\xi}(s))\right]^\top  \bd_\eta Y_{t\xi}(s)\\
	&+\widetilde{\e}\left[\left[D\frac{d\hv}{d\nu} (s,\Theta_{t\xi}(s))\left(\widetilde{Y_{t\xi}}(s)\right)\right]^\top  \left(\widetilde{\bd_\eta Y_{t\xi}}(s)+\widetilde{D_xY_{tx\mu}}(s)\left|_{x=\widetilde{\xi}}\ \right. \widetilde{\eta}\right)\right]\\
	&+\left[D_p\hv (s,\Theta_{t\xi}(s))\right]^\top \bd_\eta p_{t\xi}(s)+\sum_{j=1}^n \left[ D_{q^j}\hv (s,\Theta_{t\xi}(s))\right]^\top \bd_\eta q^j_{t\xi}(s),\quad s\in[t,T],
\end{align*}
and $\left(\widetilde{\xi},\widetilde{\eta},\widetilde{Y_{t\xi}}(s),\widetilde{D_x Y_{tx\mu}}(s),\widetilde{\bd_\eta Y_{t\xi}}(s)\right)$ is an independent copy of $(\xi,\eta,Y_{t\xi}(s),D_x Y_{tx\mu}(s),\bd_\eta Y_{t\xi}(s))$. We now build a connection among $(D_x Y_{tx\mu}(s),D_x p_{tx\mu}(s),D_x q_{tx\mu}(s))$, $(\bd_\eta Y_{t\xi}(s),\bd_\eta p_{t\xi}(s),\bd_\eta q_{t\xi}(s))$ and $(\dr_\eta Y_{t\xi}(s),\dr_\eta p_{t\xi}(s),\dr_\eta q_{t\xi}(s))$ (defined in FBSDEs \eqref{FB:dr}). 

\begin{lemma}\label{prop:4}
	Under Assumptions (A1), (A2) and (A3), there exists a constant $c(L,T)$ depending only on $(L,T)$, such that when $\lambda_v\geq c(L,T)$, there is a unique adapted solution $(\bd_\eta Y_{t\xi},\bd_\eta p_{t\xi},\bd_\eta q_{t\xi})$ of FBSDEs \eqref{FB:dr'}, and 
	\begin{equation}\label{prop4_1}
		\begin{split}
			&\dr_\eta Y_{t\xi}(s)=D_xY_{tx\mu}(s)|_{x=\xi}\ \eta+\bd_\eta Y_{t\xi}(s),\\
			&\dr_\eta p_{t\xi}(s)=D_xp_{tx\mu}(s)|_{x=\xi}\ \eta+\bd_\eta p_{t\xi}(s),\\
			&\dr_\eta q_{t\xi}(s)=D_xq_{tx\mu}(s)|_{x=\xi}\ \eta+\bd_\eta q_{t\xi}(s),\quad s\in[t,T].
		\end{split}
	\end{equation}
    Furthermore, if Assumption (A3') is satisfied, for any $\lambda_v> 0$, the same assertion holds.  	
\end{lemma}

\begin{remark}
	To consider the G\^ateaux differentiability of $\Theta_{tx\mu}$ with respect to $\xi\in L_{\f_t}^2$ along the direction $\eta$, where $\lr(\xi)=\mu$ and $\xi$ and $\eta$ are both independent of the Brownian motion $B$, we shall characterize the G\^ateaux derivatives as the solution of a system of FBSDEs, by using the processes $(\bd_\eta Y_{t\xi},\bd_\eta p_{t\xi},\bd_\eta q_{t\xi})$. The advantage of using $(\bd_\eta Y_{t\xi},\bd_\eta p_{t\xi},\bd_\eta q_{t\xi})$ instead of $(\dr_\eta Y_{t\xi}(s),\dr_\eta p_{t\xi}(s),\dr_\eta q_{t\xi}(s))$ is that the initial $\bd_\eta Y_{t\xi}(t)=0$, which will be more convenient when considering the linear functional differentiability of $\Theta_{tx\mu}(s)$ in $\mu\in\pr_2(\brn)$. 
\end{remark}

Lemma~\ref{prop:4} is proven in Appendix~\ref{pf_prop4}. For any $\eta\in L_{\f_t}^2$, consider the following FBSDEs for $(\dr_\eta Y_{tx\xi},\dr_\eta p_{tx\xi},\dr_\eta q_{tx\xi})$:
\small
\begin{align}
		&\dr_\eta Y_{tx\xi}(s)=\int_t^s \bigg\{\widetilde{\e}\left[\left[D \frac{db_0}{d\nu} (r,\lr(Y_{t\xi}(r)))\left(\widetilde{Y_{t\xi}}(r)\right) \right]^\top  \left(\widetilde{D_x Y_{tx\mu}}(r)\left|_{x=\widetilde{\xi}}\ \right. \widetilde{\eta}+\widetilde{\bd_\eta Y_{t\xi}}(r)\right)\right] \notag \\
		&\qquad\qquad\qquad\qquad +b_1(s)\dr_\eta Y_{tx\xi}(s)+b_2(s)\dr_\eta v_{tx\xi}(s)\bigg\}dr \notag \\
		&\quad\qquad\qquad+\int_t^s \bigg\{\widetilde{\e}\left[\left[D \frac{d\sigma_0}{d\nu}(r,\lr(Y_{t\xi}(r)))\left(\widetilde{Y_{t\xi}}(r)\right) \right]^\top \left(\widetilde{D_x Y_{tx\mu}}(r)\left|_{x=\widetilde{\xi}}\ \right. \widetilde{\eta}+\widetilde{\bd_\eta Y_{t\xi}}(r)\right)\right] \notag \\
		&\qquad\qquad\qquad\qquad +\sigma_1(r)\dr_\eta Y_{tx\xi}(r)+\sigma_2(r)\dr_\eta v_{tx\xi}(r)\bigg\}dB(r), \notag \\
		&\dr_\eta p_{tx\xi}(s)=\left[D_x^2 g (Y_{tx\mu}(T),\lr(Y_{t\xi}(T)))\right]^\top \dr_\eta Y_{tx\xi}(T) \notag \\
		&\quad\qquad\qquad +\widetilde{\e}\left[\left[\left(D_y \frac{d}{d\nu}D_x g\right) (Y_{tx\mu}(T),\lr(Y_{t\xi}(T)))\left(\widetilde{Y_{t\xi}}(T)\right)\right]^\top \left(\widetilde{D_x Y_{tx\mu}}(T)\left|_{x=\widetilde{\xi}}\ \right. \widetilde{\eta}+\widetilde{\bd_\eta Y_{t\xi}}(T)\right)\right] \notag \\
		&\quad\qquad\qquad +\int_s^T \bigg\{b_1(r)^\top \dr_\eta p_{tx\xi}(r)+\sum_{j=1}^n \left(\sigma_1^j (r)\right)^\top \dr_\eta q^j_{tx\xi}(r) +\left[D_x^2 f (r,\theta_{tx\mu}(r))\right]^\top \dr_\eta Y_{tx\xi}(r) \notag \\
		&\ \qquad\qquad\qquad\qquad +\widetilde{\e}\left[\left[\left(D_y \frac{d}{d\nu}D_x f\right) (r,\theta_{tx\mu}(r))\left(\widetilde{Y_{t\xi}}(r)\right)\right]^\top \left(\widetilde{D_x Y_{tx\mu}}(r)\left|_{x=\widetilde{\xi}}\ \right. \widetilde{\eta}+\widetilde{\bd_\eta Y_{t\xi}}(r)\right)\right] \notag \\
		&\ \qquad\qquad\qquad\qquad +\left[D_vD_x f (r,\theta_{tx\mu}(r))\right]^\top \dr_\eta v_{tx\xi}(r)\bigg\}dr  -\int_s^T \dr_\eta q_{tx\xi}(r)dB(r),\quad s\in[t,T], \label{FB:mu'} 
\end{align}
\normalsize
where
\begin{align*}
	\dr_\eta v_{tx\xi}(s):=&\left[D_x\hv (s,\Theta_{tx\mu}(s)) \right]^\top \dr_\eta Y_{tx\xi}(s)\\
	&+\widetilde{\e}\left[\left[D_y\frac{d\hv}{d\nu} (s,\Theta_{tx\mu}(s))\left(\widetilde{Y_{t\xi}}(s)\right)\right]^\top\left(\widetilde{D_x Y_{tx\mu}}(s)\left|_{x=\widetilde{\xi}}\ \right. \widetilde{\eta}+\widetilde{\bd_\eta Y_{t\xi}}(s)\right) \right]\\
	&+\left[D_p\hv (s,\Theta_{tx\mu}(s))\right]^\top\dr_\eta p_{tx\xi}(s)+\sum_{j=1}^n \left[D_{q^j}\hv  (s,\Theta_{tx\mu}(s))\right]^\top\dr_\eta q^j_{tx\xi}(s),\quad s\in[t,T],
\end{align*}
and $\left(\widetilde{\xi},\widetilde{\eta},\widetilde{Y_{t\xi}}(s),\widetilde{D_x Y_{tx\mu}}(s),\widetilde{\bd_\eta Y_{t\xi}}(s)\right)$ is an independent copy of $(\xi,\eta,Y_{t\xi}(s),D_x Y_{tx\mu}(s),\bd_\eta Y_{t\xi}(s))$. We then have the following statement, whose proof is given in Appendix~\ref{pf_prop3}. 

\begin{lemma}\label{prop:3}
	Under Assumptions (A1), (A2) and (A3), there exists a constant $c(L,T)>0$ depending only on $(L,T)$, such that when $\lambda_v\geq c(L,T)$, for any $\eta\in L_{\f_t}^2$, there is a unique adapted solution $\left(\dr_\eta Y_{tx\xi},\dr_\eta p_{tx\xi},\dr_\eta q_{tx\xi}\right)$ of FBSDEs \eqref{FB:mu'}, and $\dr_\eta Y_{tx\xi}(s)$, $\dr_\eta p_{tx\xi}(s)$ and $\dr_\eta q_{tx\xi}(s))$ are the respective G\^ateaux derivatives of $(Y_{tx\mu}(s),p_{tx\mu}(s),q_{tx\mu}(s))$ with respect to the lifting $\xi\sim\mu$ along the direction $\eta$. Moreover, the G\^ateaux derivatives satisfy the following boundedness property	
	\begin{align}\label{prop3_1}
		\e\bigg[\sup_{t\le s\le T}|\dr_\eta Y_{tx\xi}(s)|^2+\sup_{t\le s\le T}|\dr_\eta p_{tx\xi}(s)|^2+ \int_t^T|\dr_\eta q_{tx\xi}(s)|^2 ds\bigg]\le C(L,T,\lambda_v)\e\left[|\eta|^2\right],
	\end{align}
	and they are linear in $\eta$ and continuous in $(x,\xi)$. Furthermore, if Assumption (A3') is satisfied, for any $\lambda_v> 0$, the same assertion holds.  		
\end{lemma}

To give the differentiability of $(Y_{tx\mu}(s),p_{tx\mu}(s),q_{tx\mu}(s))$ with respect to the initial distribution $\mu\in\pr_2(\brn)$, we consider the following FBSDEs:
\small
\begin{align}
		\bd Y_{t\xi}(s,z)=\ & \int_t^s \bigg\{\widetilde{\e}\bigg[\left[D \frac{db_0}{d\nu} (r,\lr(Y_{t\xi}(r)))\left(\widetilde{Y_{tz\mu}}(r)\right)\right]^\top \widetilde{D_zY_{tz\mu}}(r) \notag\\
		&\qquad\qquad +\left[D \frac{db_0}{d\nu} (r,\lr(Y_{t\xi}(r)))\left(\widetilde{Y_{t\xi}}(r)\right)\right]^\top \widetilde{\bd Y_{t\xi}}(r,z)\bigg] \notag\\
		&\quad\qquad +b_1(r)\bd Y_{t\xi}(r,z)+b_2(r)\bd v_{t\xi}(r,z)\bigg\}dr \notag\\
		& +\int_t^s \bigg\{\widetilde{\e}\bigg[\left[D \frac{d\sigma_0}{d\nu} (r,\lr(Y_{t\xi}(r)))\left(\widetilde{Y_{tz\mu}}(r)\right)\right]^\top \widetilde{D_zY_{tz\mu}}(r) \notag\\
		&\quad\qquad\qquad +\left[D \frac{d\sigma_0}{d\nu} (r,\lr(Y_{t\xi}(r)))\left(\widetilde{Y_{t\xi}}(r)\right) \right]^\top \widetilde{\bd Y_{t\xi}}(r,z)\bigg] \notag\\
		&\qquad\qquad +\sigma_1(r)\bd Y_{t\xi}(r,z)+\sigma_2(r)\bd v_{t\xi}(r,z)\bigg\}dB(r), \notag\\
		\bd p_{t\xi}(s,z)=\ & \left[D_x^2 g (Y_{t\xi}(T),\lr(Y_{t\xi}(T)))\right]^\top \bd Y_{t\xi}(T,z) \notag\\
		&+\widetilde{\e}\bigg[\left[\left(D_y \frac{d}{d\nu}D_x g\right) (Y_{t\xi}(T),\lr(Y_{t\xi}(T)))\left(\widetilde{Y_{tz\mu}}(T)\right)\right]^\top \widetilde{D_zY_{tz\mu}}(T) \notag\\
		&\qquad +\left[\left(D_y \frac{d}{d\nu}D_x g\right) (Y_{t\xi}(T),\lr(Y_{t\xi}(T)))\left(\widetilde{Y_{t\xi}}(T)\right)\right]^\top\widetilde{\bd Y_{t\xi}}(T,z)\bigg] \notag\\
		&+\int_s^T \bigg\{b_1 (r)^\top \bd p_{t\xi}(r,z)+\sum_{j=1}^n\left(\sigma_1^j (r)\right)^\top \bd q^j_{t\xi}(r,z) \notag\\
		&\qquad\qquad +\left[D_x^2 f  (r,\theta_{t\xi}(r))\right]^\top \bd Y_{t\xi}(r,z) +\left[D_vD_x f (r,\theta_{t\xi}(r))\right]^\top \bd v_{t\xi}(r,z)\notag \\
		&\qquad\qquad +\widetilde{\e}\bigg[\left[\left(D_y \frac{d}{d\nu}D_x f\right)  (r, \theta_{t\xi}(r))\left(\widetilde{Y_{tz\mu}}(r)\right)\right]^\top \widetilde{D_zY_{tz\mu}}(r) \notag \\
		&\qquad\qquad\qquad +\left[\left(D_y \frac{d}{d\nu}D_x f\right) (r, \theta_{t\xi}(r))\left(\widetilde{Y_{t\xi}}(r)\right)\right]^\top \widetilde{\bd Y_{t\xi}}(r,z)\bigg]\bigg\}dr \notag \\
		&-\int_s^T \bd q_{t\xi}(r,z)dB(r),\quad (s,z)\in [t,T]\times\brn, \label{FB:xi_y}
\end{align}
\normalsize
where
\small
\begin{align}
		\bd v_{t\xi}(s,z):=\ & \widetilde{\e}\bigg\{\left[D_y\frac{d\hv}{d\nu} (s,\Theta_{t\xi}(s))\left(\widetilde{Y_{tz\mu}}(s)\right)\right]^\top\!\! \widetilde{D_z Y_{tz\mu}}(s) \notag \\
		&\quad +\left[D_y\frac{d\hv}{d\nu}(s,\Theta_{t\xi}(s)) \left(\widetilde{Y_{t\xi}}(s)\right)\right]^\top\!\! \widetilde{\bd Y_{t\xi}}(s,z)\bigg\}+\left[D_x\hv (s,\Theta_{t\xi}(s))\right]^\top \bd Y_{t\xi}(s,z)  \notag \\
		&+\left[D_p\hv (s,\Theta_{t\xi}(s))\right]^\top \bd p_{t\xi}(s,z)+\sum_{j=1}^n \left[D_{q^j}\hv (s,\Theta_{t\xi}(s))\right]^\top \bd q^j_{t\xi}(s,z). \label{FB:xi_y_v}
\end{align}
\normalsize
Here, we denote $D_z Y_{tz\mu}=D_x Y_{tx\mu}\big|_{x=z}$, and $\left(\widetilde{Y_{tz\mu}}(s),\widetilde{Y_{t\xi}}(s),\widetilde{D_z Y_{tz\mu}}(s),\widetilde{\bd Y_{t\xi}}(s,z)\right)$ is an independent copy of $(Y_{tz\mu}(s),Y_{t\xi}(s),D_z Y_{tz\mu}(s),\bd Y_{t\xi}(s,z))$. We also consider the following FBSDEs for derivative proxies:
\small
\begin{align}
		\dr Y_{tx\mu}(s,z)=\ & \int_t^s \bigg\{\widetilde{\e}\bigg[\left[D \frac{db_0}{d\nu} (r,\lr(Y_{t\xi}(r)))\left(\widetilde{Y_{tz\mu}}(r)\right) \right]^\top \widetilde{D_z Y_{tz\mu}}(r) \notag \\
		&\qquad\qquad +\left[D \frac{db_0}{d\nu}  (r,\lr(Y_{t\xi}(r)))\left(\widetilde{Y_{t\xi}}(r)\right)\right]^\top \widetilde{\bd Y_{t\xi}}(r,z)\bigg] \notag \\
		&\quad\qquad +b_1(r)\dr Y_{tx\mu}(r,z)+b_2(r)\dr v_{tx\mu}(r,z)\bigg\}dr \notag \\
		& +\int_t^s \bigg\{\widetilde{\e}\bigg[\left[D \frac{d\sigma_0}{d\nu} (r,\lr(Y_{t\xi}(r)))\left(\widetilde{Y_{tz\mu}}(r)\right) \right]^\top\widetilde{D_z Y_{tz\mu}}(r) \notag \\
		&\quad\qquad\qquad +\left[D \frac{d\sigma_0}{d\nu} (r,\lr(Y_{t\xi}(r)))\left(\widetilde{Y_{t\xi}}(r)\right)\right]^\top\widetilde{\bd Y_{t\xi}}(r,z)\bigg] \notag \\
		&\qquad\qquad +\sigma_1(r)\dr Y_{tx\mu}(r,z)+\sigma_2(r)\dr v_{tx\mu}(r,z)\bigg\}dB(r), \notag \\
		\dr p_{tx\mu}(s,z)=\ &\left[D_x^2 g (Y_{tx\mu}(T),\lr(Y_{t\xi}(T)))\right]^\top\dr Y_{tx\mu}(T,z) \notag \\
		&+\widetilde{\e}\bigg[\left[\left(D_y \frac{d}{d\nu}D_x g\right) (Y_{tx\mu}(T),\lr(Y_{t\xi}(T)))\left(\widetilde{Y_{tz\mu}}(T)\right)\right]^\top\widetilde{D_z Y_{tz\mu}}(T) \notag \\
		&\qquad +\left[\left(D_y \frac{d}{d\nu}D_x g\right) (Y_{tx\mu}(T),\lr(Y_{t\xi}(T)))\left(\widetilde{Y_{t\xi}}(T)\right)\right]^\top\widetilde{\bd Y_{t\xi}}(T,z)\bigg] \notag \\
		&+\int_s^T \bigg\{b_1 (r)^\top \dr p_{tx\mu}(r,z)+\sum_{j=1}^n\left(\sigma_1^j (r)\right)^\top \dr q^j_{tx\mu}(r,z) \notag \\
		&\qquad\qquad +\left[D_x^2 f (r,\theta_{tx\mu}(r)) \right]^\top\dr Y_{tx\mu}(r,z)+\left[D_vD_x f (r,\theta_{tx\mu}(r))\right]^\top\dr v_{tx\mu}(r,z) \notag \\
		&\qquad\qquad +\widetilde{\e}\bigg[\left[\left(D_y \frac{d}{d\nu}D_x f\right) (r,\theta_{tx\mu}(r))\left(\widetilde{Y_{tz\mu}}(r)\right)\right]^\top \widetilde{D_z Y_{tz\mu}}(r)  \notag \\
		&\qquad\qquad\qquad +\left[\left(D_y \frac{d}{d\nu}D_x f\right) (r,\theta_{tx\mu}(r))\left(\widetilde{Y_{t\xi}}(r)\right)\right]^\top \widetilde{\bd Y_{t\xi}}(r,z)\bigg]\bigg\}dr \notag \\
		&-\int_s^T \dr q_{tx\mu}(r,z)dB(r),\quad (s,z)\in [t,T]\times\brn, \label{FB:mu_y}
\end{align}
\normalsize
where
\small
\begin{align*}
	\dr v_{tx\mu}(s,z):=&\left[D_x\hv(s,\Theta_{tx\mu}(s))\right]^\top\dr Y_{tx\mu}(s,z)\\
	&+\widetilde{\e}\bigg\{\left[D_y\frac{d\hv}{d\nu} \!\! (s,\Theta_{tx\mu}(s))\left(\widetilde{Y_{tz\mu}}(s)\right)\right]^\top \widetilde{D_z Y_{tz\mu}}(s)\!+\!\left[D_y\frac{d\hv}{d\nu}\!\!  (s,\Theta_{tx\mu}(s))\left(\widetilde{Y_{t\xi}}(s)\right)\right]^\top\widetilde{\bd Y_{t\xi}}(s,z) \bigg\}\\
	&+\left[D_p\hv (s,\Theta_{tx\mu}(s))\right]^\top\dr p_{tx\mu}(s,z)+\sum_{j=1}^n \left[D_{q^j}\hv (s,\Theta_{tx\mu}(s))\right]^\top \dr q^j_{tx\mu}(s,z).
\end{align*}
\normalsize
Since the FBSDEs \eqref{FB:mu_y} depend on $\xi$ only through its law $\mu$, it is legitimate to use the subscript $\mu$ in $(\dr Y_{tx\mu},\dr v_{tx\mu},\dr p_{tx\mu},\dr q_{tx\mu})$. Finally, the following result shows the well-posedness of FBSDEs \eqref{FB:xi_y}  and \eqref{FB:mu_y}, and also about the linear functional differentiability of $(Y_{tx\mu},v_{tx\mu},p_{tx\mu}, q_{tx\mu})$ in $\mu$. The proof is given in Appendix~\ref{pf_prop5}. 

\begin{theorem}\label{prop:5}
	Under Assumptions (A1), (A2) and (A3), there exists a constant $c(L,T)>0$ depending only on $(L,T)$, such that when $\lambda_v\geq c(L,T)$, there is a unique adapted solution
	\begin{align*}
		(\bd Y_{t\xi}(s,z),\bd p_{t\xi}(s,z),\bd q_{t\xi}(s,z)),\quad (s,z)\in[t,T]\times\brn
	\end{align*}
	for FBSDEs \eqref{FB:xi_y} satisfying the boundedness property:
	\begin{align}\label{prop5_01}
		\e\bigg[\sup_{t\le s\le T}\left|\left(\bd Y_{t\xi}(s,z),\bd p_{t\xi}(s,z)\right)^\top\right|^2+\int_t^T\left|\bd q_{t\xi}(s,z)\right|^2 ds\bigg]\le C(L,T,\lambda_v),
	\end{align}
    and these component processes are also continuous in $(\xi,y)$. Besides, there is a unique adapted solution 
    \begin{align*}
    	(\dr Y_{tx\mu}(s,z),\dr p_{tx\mu}(s,z),\dr q_{tx\mu}(s,z)),\quad (s,z)\in[t,T]\times\brn
    \end{align*}
    for the system of FBSDEs \eqref{FB:mu_y} satisfying the boundedness property:
    \begin{align}\label{prop5_02}
    	\e\bigg[\sup_{t\le s\le T}\left|\left(\dr Y_{tx\mu}(s,z),\dr p_{tx\mu}(s,z)\right)^\top\right|^2+\int_t^T\left|\dr q_{tx\mu}(s,z)\right|^2 ds\bigg]\le C(L,T,\lambda_v),
    \end{align}
	and the component processes are continuous in $(x,\mu,y)$. Moreover, for any $\xi,\eta\in L_{\f_t}^2$ independent of the Brownian motion $B$ such that $\lr(\xi)=\mu$, we have
	\begin{equation}\label{prop5_03}
		\begin{split}
			&\bd_\eta Y_{t\xi}(s)=\widehat{\e}\left[\bd Y_{t\xi}\left(s,\widehat{\xi}\right) \widehat{\eta}\right],\quad \bd_\eta v_{t\xi}(s)=\widehat{\e}\left[\bd v_{t\xi}\left(s,\widehat{\xi}\right) \widehat{\eta}\right],\\
			&\bd_\eta p_{t\xi}(s)=\widehat{\e}\left[\bd p_{t\xi}\left(s,\widehat{\xi}\right) \widehat{\eta}\right],\quad \bd_\eta q_{t\xi}(s)=\widehat{\e}\left[\bd q_{t\xi}\left(s,\widehat{\xi}\right) \widehat{\eta}\right],\quad s\in[t,T],
		\end{split}
	\end{equation}
	and then,
	\begin{equation}\label{prop5_04}
		\begin{split}
			&\dr_\eta Y_{tx\xi}(s)=\widehat{\e}\left[\dr Y_{tx\mu}\left(s,\widehat{\xi}\right) \widehat{\eta}\right],\quad \dr_\eta v_{tx\xi}(s)=\widehat{\e}\left[\dr v_{tx\mu}\left(s,\widehat{\xi}\right) \widehat{\eta}\right],\\
			&\dr_\eta p_{tx\xi}(s)=\widehat{\e}\left[\dr p_{tx\mu}\left(s,\widehat{\xi}\right) \widehat{\eta}\right],\quad \dr_\eta q_{tx\xi}(s)=\widehat{\e}\left[\dr q_{tx\mu}\left(s,\widehat{\xi}\right) \widehat{\eta}\right],\quad s\in[t,T],
		\end{split}
	\end{equation}
	where $(\widehat{\xi},\widehat{\eta})$ is an independent copy of $(\xi,\eta)$. That is, the mapping 	
	\begin{align*}
		\mu \mapsto (Y_{tx\mu}(s),v_{tx\mu}(s),p_{tx\mu}(s),q_{tx\mu}(s))
	\end{align*}
	is linearly functionally differentiable, with the respective derivatives being
	\begin{align*}
		&D_y\frac{d}{d\nu} Y_{tx\mu}(s)(y)=\dr Y_{tx\mu}(s,y),\quad D_y\frac{d}{d\nu} v_{tx\mu}(s)(y)=\dr v_{tx\mu}(s,y),\\
		&D_y\frac{d}{d\nu} p_{tx\mu}(s)(y)=\dr p_{tx\mu}(s,y),\quad D_y\frac{d}{d\nu} q_{tx\mu}(s)(y)=\dr q_{tx\mu}(s,y).
	\end{align*}
    Furthermore, if Assumption (A3') is satisfied, for any $\lambda_v> 0$, the same assertion holds.  		
\end{theorem}

\subsection{Hessian flows of second-order derivatives}

In order to ensure the well-posedness of the second-order derivatives of $(Y_{tx\mu}(s),p_{tx\mu}(s),q_{tx\mu}(s))$ in $(x,\mu)$, we need the following regularity-enhanced version of Assumptions (A1) and (A2):

\textbf{(A1')} The coefficients $b$ and $\sigma$ satisfy (A1). The derivatives $(D_y^2 \frac{db_0}{d\nu}, \ D_y^2 \frac{d\sigma_0}{d\nu})$ exist, and they are continuous in all their arguments and are globally bounded by $L$ in norm.

\textbf{(A2')} The functionals $f$ and $g$ satisfy (A2). The following derivatives exist, and they are continuous in all their arguments and are globally bounded by $L$ in norm:
\begin{align*}
	D_x^3 f,\ D_x^2 D_v f, \ D_x D_v^2 f,\ D_v^3f, D_y^2 \frac{d}{d\nu}D_x f,\ D_y^2 \frac{d}{d\nu}D_v f,\ D_x^3 g,\ D_y^2 \frac{d}{d\nu}D_x g.
\end{align*}
Here, for instance, the notation $D_x^3f$ stands for the tensor $\left(D_{x_i}D_{x_j}D_{x_k}f,\ 1\le i,j,k\le n \right)$.

From Assumption (A2') and \eqref{hv}, we know that the map $\hv(\cdot)$ is twice continuously differentiable in $(x,p,q)$, and $D_y\frac{d\hv}{d\nu}$ is continuously differentiable in $y$. By differentiating Conditions \eqref{optimal_condition} in $(x,p,q)$, we have
\begin{equation}\label{optimal_condition'}
	\begin{split}
		&\begin{pmatrix}
			1 & D_x \hv(s,\Theta)\\
			0 & D_p\hv(s,\Theta)\\
			0 & D_q\hv(s,\Theta)
		\end{pmatrix}\left[
		\begin{pmatrix}
			D_x^2 D_v f &  D_x D_v D_x f  \\
			D_v D_x D_v f & D_v^3 f 
		\end{pmatrix}\left(s,x,m,\hv(s,\Theta)\right)\right]
		\begin{pmatrix}
			1 & D_x \hv(s,\Theta)\\
			0 & D_p\hv(s,\Theta)\\
			0 & D_q\hv(s,\Theta)
		\end{pmatrix}^\top  \\
	    =\ & -D_v^2f(s,x,m,\hv(s,\Theta )) \begin{pmatrix}
	    	D_x^2\hv &  D_xD_p\hv & D_xD_q\hv \ \\
	    	D_p D_x \hv & D_p^2 \hv & D_pD_q \hv\ \\
	    	D_q D_x \hv & D_qD_p \hv & D_p^2 \hv\ \\
	    \end{pmatrix} (s,\Theta).
	\end{split}
\end{equation}
By differentiating Conditions \eqref{optimal_condition} in $y$, we have
\begin{equation}\label{optimal_condition''}
	\begin{split}
		&\left(D_y^2\frac{d}{d\nu}D_vf\right) (s,x,m,\hv(s,\Theta))(y)+D_v^2f(s,x,m,\hv(s,\Theta)) D_y^2 \frac{d\hv}{d\nu} (s,\Theta)(y)=0.
	\end{split}
\end{equation}
We first give the second-order derivatives of $(Y_{tx\mu}(s),p_{tx\mu}(s),q_{tx\mu}(s))$ in $x$.  From FBSDEs \eqref{FB:x}, Theorem~\ref{prop:2}, and Conditions \eqref{optimal_condition} and \eqref{optimal_condition'}, the second-order derivatives fulfill the following system of FBSDEs:
\small
\begin{align}
		D_x^2 Y_{tx\mu}(s)=\ & \int_t^s \left[b_1(r)D_x^2 Y_{tx\mu}(r)+b_2(r)D_x^2 v_{tx\mu}(r)\right]dr +\int_t^s \left[\sigma_1(r)D_x^2 Y_{tx\mu}(r)+\sigma_2(r)D_x^2 v_{tx\mu}(r)\right]dB(r), \notag\\
		D_x^2 p_{tx\mu}(s)=\ & D_x^2 g(Y_{tx\mu}(T),\lr(Y_{t\xi}(T)))^\top D_x^2 Y_{tx\mu}(T) +D_x^3 g(Y_{tx\mu}(T),\lr(Y_{t\xi}(T)))\left(D_x Y_{tx\mu}(T)\right)^{\otimes2} \notag\\
		&+\int_s^T \Bigg[b_1(r)^\top D_x^2 p_{tx\mu}(r)+\sigma_1(r)^\top D_x^2 q_{tx\mu}(r)+\left[D_vD_x f(r,\theta_{tx\mu}(r))\right]^\top D_x^2 v_{tx\mu}(r) \notag \\
		& +\left[D_x^2 f(r,\theta_{tx\mu}(r))\right]^\top D_x^2 Y_{tx\mu}(r)+ \left[\begin{pmatrix}
			D_x^3 f &  D_x D_v D_x f  \\
			D_v D^2_x f & D_x D_v^2 f 
		\end{pmatrix}(s,\Theta_{tx\mu}(s))\right]
		\begin{pmatrix}
			D_x Y_{tx\mu}(s)\\
			D_x v_{tx\mu}(s)
		\end{pmatrix}^{\otimes 2} \Bigg]dr \notag\\
		&-\int_s^T D_x^2 q_{tx\mu}(r)dB(r),\quad t\in[t,T], \label{FB:xx}
\end{align}
\normalsize
where
\small
\begin{align*}
	D_x^2 v_{tx\mu}(s):=\ & \left[D_x\hv (s,\Theta_{tx\mu}(s))\right]^\top D_x^2 Y_{tx\mu}(s)+\left[D_p\hv (s,\Theta_{tx\mu}(s))\right]^\top D_x^2 p_{tx\mu}(s) + \left[D_{q}\hv (s,\Theta_{tx\mu}(s))\right]^\top D_x^2 q_{tx\mu}(s)\\
	&+ \begin{pmatrix}
		D_x^2 \hv &  D_x D_p \hv &  D_xD_q \hv \\
		D_p D_x \hv & D_p^2 \hv & D_p D_q \hv\\
		D_q D_x \hv &  D_q D_p \hv & D_q^2 \hv
	\end{pmatrix}(s,\Theta_{tx\mu}(s))
    \begin{pmatrix}
    	D_x Y_{tx\mu}(s)\\
    	D_x p_{tx\mu}(s)\\
    	D_x q_{tx\mu}(s)
    \end{pmatrix}^{\otimes 2},\quad s\in[t,T].
\end{align*}
\normalsize
The proof of the following result is similar to that leading to statements in Subsection~\ref{subsec:1-order}, and we omit it here. 

\begin{theorem}\label{prop:8}
	Under Assumptions (A1), (A2') and (A3), there exists a constant $c(L,T)>0$ depending only on $(L,T)$, such that when $\lambda_v\geq c(L,T)$, there is a unique adapted solution $(D_x^2 Y_{tx\mu},D_x^2 p_{tx\mu},D_x^2 q_{tx\mu})$ of FBSDEs \eqref{FB:xx}, and $(D_x^2 Y_{tx\mu}(s),D_x^2 p_{tx\mu}(s),D_x^2 q_{tx\mu}(s))$ are the second-order G\^ateaux derivatives of $(Y_{tx\mu}(s), p_{tx\mu}(s), q_{tx\mu}(s))$ in $x$. Moreover, the second-order G\^ateaux derivatives satisfy the boundedness property
	\begin{align}\label{boundedness-property}
		\e\bigg[\sup_{t\le s\le T}\left|\left(D_x^2 Y_{tx\mu}(s),D_x^2 p_{tx\mu}(s)\right)^\top\right|^2+ \int_t^T\left|D_x^2 q_{tx\mu}(s)\right|^2 ds\bigg]\le C(L,T,\lambda_v),
	\end{align}
	and they are continuous in $(x,\mu)$. Furthermore, if Assumption (A3') is satisfied, for any $\lambda_v> 0$, the same assertion holds.  		
\end{theorem}

We next give the derivatives in $z$ of the processes $\bd Y_{t\xi}(s,z)$, $\bd p_{t\xi}(s,z)$, $\bd q_{t\xi}(s,z)$, $\dr Y_{tx\mu}(s,z)$, $\dr p_{tx\mu}(s,z)$ and $\dr q_{tx\mu}(s,z)$, its proof is given in Appendix~\ref{pf:prop:9}.

\begin{theorem}\label{prop:9}
	Under Assumptions (A1'), (A2') and (A3), there exists a constant $c(L,T)>0$ depending only on $(L,T)$, such that when $\lambda_v\geq c(L,T)$, the processes $\bd Y_{t\xi}(s,z)$, $\bd p_{t\xi}(s,z)$ and $\bd q_{t\xi}(s,z)$ are G\^ateaux differentiable in $z$, and the respective G\^ateaux derivatives $D_z\bd Y_{t\xi}(s,z)$, $D_z\bd p_{t\xi}(s,z)$ and $D_y\bd q_{t\xi}(s,z)$ satisfy the boundedness property
    \begin{align*}
    	\e\bigg[\sup_{t\le s\le T}\left|\left(D_z \bd Y_{t\xi}(s,z),D_z\bd p_{t\xi}(s,z)\right)^\top\right|^2+\int_t^T\left|D_z\bd q_{t\xi}(s,z)\right|^2 ds\bigg]\le C(L,T,\lambda_v),
    \end{align*}
    and are continuous in $(\xi,z)$. The processes $\dr Y_{tx\mu}(s,z)$, $\dr p_{tx\mu}(s,z)$ and $\dr q_{tx\mu}(s,z)$ are G\^ateaux differentiable in $z$, and the G\^ateaux derivatives $D_z\dr Y_{tx\mu}(s,z)$, $D_z\dr p_{tx\mu}(s,z)$ and $D_z\dr q_{tx\mu}(s,z)$ satisfy the boundedness property
    \begin{align*}
    	\e\bigg[\sup_{t\le s\le T}\left|\left(D_z \dr Y_{tx\mu}(s,z),D_z\dr p_{tx\mu}(s,z)\right)^\top\right|^2+ \int_t^T\left|D_z\dr q_{tx\mu}(s,z)\right|^2 ds\bigg]\le C(L,T,\lambda_v),
    \end{align*}
    and are continuous in $(x,\mu,z)$. Furthermore, if Assumption (A3') is satisfied, for any $\lambda_v> 0$, the same assertion holds. 
\end{theorem}

\section{Regularity of the value functional}\label{sec:V}

In this section, we apply our results in Sections~\ref{sec:distribution} and \ref{sec:state} to study the spatial, distributional and temporal regularities of the value functional $V$ defined in \eqref{intro_4} as a function in $(t,x,\mu)\in[0,T]\times\brn\times\pr_2(\brn)$. We first give the spatial regularity of $V$, whose proof is given in Appendix~\ref{pf_prop6}. 

\begin{theorem}\label{prop:6}
	Under Assumptions (A1), (A2) and (A3), there exists a constant $c(L,T)>0$ depending only on $(L,T)$, such that when $\lambda_c\geq c(L,T)$, the value functional $V$ is $C^2$ in $x$ with the derivatives
	\begin{align}
		D_x V(t,x,\mu)  =p_{tx\mu}(t), \qquad D_x^2 V(t,x,\mu)  =D_x p_{tx\mu}(t) \label{prop6_03},
	\end{align}
    and they also satisfy the growth conditions	
    \begin{equation}\label{prop6_04}
    	\begin{split}
    		& |V(t,x,\mu)| \le C(L,T,\lambda_v)\left[1+|x|^2+W^2_2(\mu,\delta_0)\right],\\
    		& |D_x V(t,x,\mu)| \le C(L,T,\lambda_v)\left[1+|x|+W_2(\mu,\delta_0)\right],\qquad \left|D_x^2 V(t,x,\mu)\right|\le C(L,T,\lambda_v),
    	\end{split}
    \end{equation}
    and the following continuity property holds:
    \begin{align}
    	\left|D_x V(t,x',\mu')-D_x V(t,x,\mu)\right|\le C(L,T,\lambda_v)\left[|x'-x|+W_2(\mu,\mu')\right], \label{prop6_05}
    \end{align}
    and $D_x^2 V$ is continuous in $(x,\mu)$. Furthermore, if Assumption (A3') is satisfied, for any $\lambda_v> 0$, the same assertion holds. 
\end{theorem}

As a consequence of \eqref{prop6_03}, we have the following characterization of $p_{t\xi}(t)$:	
\begin{align}\label{rk_1}
	p_{t\xi}(t)=p_{tx\mu}\big|_{x=\xi}=D_x V(t,x,\mu)\big|_{x=\xi}.
\end{align}
We next give the regularity of $V$ with respect to the distribution argument $\mu$, whose proof is given in Appendix~\ref{pf_prop7}.

\begin{theorem}\label{prop:7}
	Under Assumptions (A1'), (A2') and (A3), there exists a constant $c(L,T)>0$ depending only on $(L,T)$, such that when $\lambda_v\geq c(L,T)$, the value functional $V$ is linearly functionally differentiable in $\mu$ with the derivatives	
	\begin{align}
			D_y\frac{dV}{d\nu}(t,x,\mu)(y)=\ &\int_t^T \e\bigg\{\left[D_x f (s,\theta_{tx\mu}(s)) \right]^\top \dr Y_{tx\mu}(s,y)+ \left[D_v f  (s,\theta_{tx\mu}(s))  \right]^\top \dr v_{tx\mu}(s,y) \notag \\
			&\quad\qquad +\widetilde{\e}\bigg[\left[D_y \frac{df}{d\nu} (s,\theta_{tx\mu}(s))\left(\widetilde{Y_{ty\mu}}(s)\right) \right]^\top \widetilde{D_y Y_{ty\mu}}(s)  \notag  \\
			&\quad\qquad\qquad +\left[D_y \frac{df}{d\nu} (s,\theta_{tx\mu}(s))\left(\widetilde{Y_{t\xi}}(s)\right) \right]^\top \widetilde{\bd Y_{t\xi}}(s,y) \bigg]  \bigg\} ds  \notag  \\
			& + \e\bigg\{\left[D_x g (Y_{tx\mu}(T),\lr(Y_{t\xi}(T)))  \right]^\top \dr Y_{tx\mu}(T,y) \notag \\
			&\qquad +\widetilde{\e}\bigg[\left[D_y \frac{dg}{d\nu} (Y_{tx\mu}(T),\lr(Y_{t\xi}(T)))\left(\widetilde{Y_{ty\mu}}(T)\right)  \right]^\top \widetilde{D_y Y_{ty\mu}}(T)  \notag  \\
			&\qquad\qquad +\left[D_y \frac{dg}{d\nu} (Y_{tx\mu}(T),\lr(Y_{t\xi}(T)))\left(\widetilde{Y_{t\xi}}(T)\right) \right]^\top \widetilde{\bd Y_{t\xi}}(T,y) \bigg] \bigg\}, \label{prop7_01}
	\end{align}
	and
	\begin{align}
			&D_y^2\frac{dV}{d\nu}(t,x,\mu)(y) \notag \\
			=\ & \int_t^T \e\bigg\{ \left[D_x f (s, \theta_{tx\mu}(s))  \right]^\top D_y\dr Y_{tx\mu}(s,y)+ \left[D_v f (s, \theta_{tx\mu}(s))  \right]^\top D_y \dr v_{tx\mu}(s,y) \notag \\
			&\quad\qquad +\widetilde{\e}\bigg[\left[D_y \frac{df}{d\nu} (s, \theta_{tx\mu}(s))\left(\widetilde{Y_{ty\mu}}(s)\right)  \right]^\top \widetilde{D_y^2Y_{ty\mu}}(s)   \notag \\
			&\quad\qquad\qquad +\left(\widetilde{D_y Y_{ty\mu}}(s) \right)^\top \left[ D_y^2 \frac{df}{d\nu}(s, \theta_{tx\mu}(s))\left(\widetilde{Y_{ty\mu}}(s)\right)\right] \widetilde{D_yY_{ty\mu}}(s)   \notag \\
			&\quad\qquad\qquad +\left[D_y \frac{df}{d\nu} (s, \theta_{tx\mu}(s))\left(\widetilde{Y_{t\xi}}(s)\right) \right]^\top \widetilde{D_y \bd Y_{t\xi}}(s,y) \bigg]  \bigg\} ds  \notag \\
			& + \e\bigg\{\left[D_x g (Y_{tx\mu}(T),\lr(Y_{t\xi}(T)))  \right]^\top D_y \dr Y_{tx\mu}(T,y) \notag \\
			&\qquad +\widetilde{\e}\bigg[\left[D_y \frac{dg}{d\nu} (Y_{tx\mu}(T),\lr(Y_{t\xi}(T)))\left(\widetilde{Y_{ty\mu}}(T)\right) \right]^\top \widetilde{D_y^2Y_{ty\mu}}(T)  \notag \\
			&\qquad\qquad + \left(\widetilde{D_yY_{ty\mu}}(T) \right)^\top \left[ D_y^2 \frac{dg}{d\nu}(Y_{tx\mu}(T),\lr(Y_{t\xi}(T)))\left(\widetilde{Y_{ty\mu}}(T)\right) \right] \widetilde{D_yY_{ty\mu}}(T)  \notag  \\
			&\qquad\qquad + \left[D_y \frac{dg}{d\nu} (Y_{tx\mu}(T),\lr(Y_{t\xi}(T)))\left(\widetilde{Y_{t\xi}}(T)\right) \right]^\top \widetilde{D_y \bd Y_{t\xi}}(T,y) \bigg]\bigg\}, \label{prop7_02}
	\end{align}
	where $\left(D_y Y_{ty\mu},D_y^2 Y_{ty\mu}\right)=\left(D_x Y_{tx\mu},D_x^2 Y_{tx\mu} \right)\!\Big|_{x=y}$, and 
	\begin{align*}
		\left(\widetilde{Y_{ty\mu}}(s),\widetilde{Y_{t\xi}}(s),\widetilde{D_y Y_{ty\mu}}(s),\widetilde{D_y^2 Y_{ty\mu}}(s),\widetilde{\bd Y_{t\xi}}(s,y),\widetilde{D_y\bd Y_{t\xi}}(s,y)\right)
	\end{align*}	
	is an independent copy of 
	\begin{align*}
		\left(Y_{ty\mu}(s),Y_{t\xi}(s),D_y Y_{ty\mu}(s),D_y^2 Y_{ty\mu}(s),\bd Y_{t\xi}(s,y),D_y\bd Y_{t\xi}(s,y)\right).
	\end{align*}
	Moreover, the derivatives satisfy the following growth condition
	\begin{align}
		&\left|D_y\frac{dV}{d\nu}(t,x,\mu)(y)\right|+\left|D_y^2\frac{dV}{d\nu}(t,x,\mu)(y)\right|\le C(L,T,\lambda_v)(1+|x|+W_2(\mu,\delta_0)),\label{prop7_03}
	\end{align}
	and $(D_y\frac{dV}{d\nu},D_y^2\frac{dV}{d\nu})$ are continuous in $(x,\mu,y)$. 	Furthermore, if Assumption (A3') is satisfied, for any $\lambda_v> 0$, the same assertion holds. 
\end{theorem}

To proceed further, we make the following additional assumption on coefficients $(b,\sigma,f)$ in $t$.

\medskip
\textbf{(A4)} The coefficients $b,\sigma$, and $f$ are continuous in $t\in[0,T]$, and $\sigma_2(s)=0$. 

\begin{remark}
	Here, Assumption (A4) is required for the regularity of the value functional in $t$, but it is not necessary for the solvability of the mean field game \eqref{intro_1} and the control problem \eqref{intro_1'}.
\end{remark}

We now give the time regularity of $V$, and the proof is given in Appendix~\ref{pf_prop10}. 

\begin{theorem}\label{prop:10}
	Under Assumptions (A1'), (A2'), (A3) and (A4), there exists a constant $c(L,T)>0$ depending only on $(L,T)$, such that when $\lambda_v\geq c(L,T)>0$, the value functional $V$ is $C^1$ in $t$ with the temporal derivative
	\begin{equation}\label{prop10_01}
		\begin{split}
			\dd_t V(t,x,\mu)=&-H\left(t,x,\mu,D_x V(t,x,\mu),\frac{1}{2}D_x^2 V(t,x,\mu)\sigma(t,x,\mu)\right)\\
			& -\int_\brn \bigg[b \left(t,{y},\mu,\hv(t,y,\mu,D_x V(t,y,\mu))\right)^\top D_y\frac{dV}{d\nu}(t,x,\mu)({y})\\
			&\quad\qquad +\frac{1}{2}\text{Tr}\left(\sigma (t,{y},\mu)^\top D_y^2\frac{dV}{d\nu}(t,x,\mu)({y})\sigma(t,{y},\mu)\right)\bigg]d\mu(y).
		\end{split}
	\end{equation}
    Furthermore, if Assumption (A3') is satisfied, for any $\lambda_v> 0$, the same assertion holds. 
\end{theorem}

\section{Master equation}\label{sec:master}

Last but not least, we also discuss the solvability of the corresponding master equation:
\begin{equation}\label{master}
	\left\{
	\begin{aligned}
		&\dd_t V(t,x,\mu)+H\left(t,x,\mu,D_x V(t,x,\mu),\frac{1}{2}D_x^2 V(t,x,\mu)\sigma(t,x,\mu)\right)\\
		&\quad\qquad\qquad +\int_\brn \bigg[b \left(t,y,\mu,\hv(t,y,\mu,D_x V(t,y,\mu))\right)^\top D_y\frac{dV}{d\nu}(t,x,\mu)({y})\\
		&\qquad\qquad\qquad\qquad +\frac{1}{2}\text{Tr}\left(\left(\sigma\sigma^\top\right) (t,{y},\mu) D_y^2\frac{dV}{d\nu}(t,x,\mu)({y})\right)\bigg] d\mu(y)=0, \quad t\in[0,T),\\
		&V(T,x,\mu)=g(x,\mu),\quad (x,\mu)\in\brn\times\pr_2(\brn).
	\end{aligned}	
	\right.
\end{equation}
We give the well-posedness of the master equation \eqref{master}, and the proof is given in Appendix~\ref{pf_thm1}. 

\begin{theorem}\label{thm:1}
	Under Assumptions (A1'), (A2'), (A3) and (A4), there exists a constant $c(L,T)>0$ depending only on $(L,T)$, such that when $\lambda_v\geq C(L,T)$, the value functional $V$ is the unique solution of the master equation \eqref{master} in the sense that the following derivatives
	\begin{align*}
		\dd_t V,\ D_x V,\ D_x^2 V,\ D_y\frac{dV}{d\nu},\ D_y^2\frac{dV}{d\nu}
	\end{align*}
	exist and  are all continuous, and they satisfy the growth conditions \eqref{prop6_04}, \eqref{prop6_05}, and \eqref{prop7_03}. Furthermore, if Assumption (A3') is satisfied, for any $\lambda_v> 0$, the same assertion holds. 
\end{theorem}

\begin{remark}\label{rk:FP-HJB}
	For $(t,\mu)\in[0,T]\times\pr_2(\brn)$, recall that   $Y_{t\xi}$ is the equilibrium state process for MFG \eqref{intro_1} and $m^{t,\mu}(s):=\lr\left(Y_{t\xi}(s)\right),\ s\in[t,T]$. From Theorem~\ref{prop:6}, we know that $m^{t,\mu}$ satisfies the following FP equation
	\begin{equation}\label{FP}
		\left\{
		\begin{aligned}
			&\dd_s m^{t,\mu}+ {div}\left[ b\left(s,x,m^{t,\mu}(s),\hv \left(s,x,m^{t,\mu}(s),D_x V\left(s,x,m^{t,\mu}(s)\right)\right)\right) m^{t,\mu}\right]\\
			&\:\:\ \qquad -\sum_{i,j=1}^n \dd_{x_i}\dd_{x_j} \left[a_{ij}\left(s,x,m^{t,\mu}(s)\right)m^{t,\mu}\right]=0,\quad s\in(t,T],\\
	        &m^{t,\mu}(t)=\mu,
		\end{aligned}
	    \right.
	\end{equation}
	where $a_{ij}:=\frac{1}{2}\left(\sigma\sigma^\top\right)_{ij}$. We define 
	\begin{align}\label{decouple}
		v^{t,\mu}(s,x):=V\left(s,x,m^{t,\mu}(s)\right),\ (s,x)\in[t,T]\times\brn,
	\end{align}
	where $V$ is the value functional of Problem \eqref{intro_1'}. Then, we know that
	\begin{align*}
		D_x v^{t,\mu}(s,x)=D_x V\left(s,x,m^{t,\mu}(s)\right), \quad D_x^2 v^{t,\mu}(s,x)=D_x^2 V\left(s,x,m^{t,\mu}(s)\right),
	\end{align*}
	and then, from the FP equation \eqref{FP}, we have
	\begin{align*}
		\dd_s v^{t,\mu}(s,x)=\ & \int_\brn \bigg[b\left(s,y,m^{t,\mu}(s),\hv \left(s,y,m^{t,\mu}(s),D_x v^{t,\mu}(s,y)\right)\right) ^\top  D_y\frac{dV}{d\nu}\left(s,x,m^{t,\mu}(s)\right)(y)\\
		&\qquad + a_{ij}\left(s,x,m^{t,\mu}(s)\right) D_y^2\frac{dV}{d\nu}\left(s,x,m^{t,\mu}(s)\right)(y) \bigg] d\mu(y) +\dd_s V\left(s,x,m^{t,\mu}(s)\right).
	\end{align*}
	Therefore, from Theorem~\ref{thm:1} and the mean field game master equation \eqref{master}, we know that $v^{t,\mu}$ satisfies the following backward equation:
	\begin{equation}\label{HJB}
		\left\{
		\begin{aligned}
			&\dd_s v^{t,\mu}(s,x)+H\left(s,x,m^{t,\mu}(s),D_x v^{t,\mu}(s,x),\frac{1}{2}D_x^2 v^{t,\mu}(s,x)\sigma\left(s,x,m^{t,\mu}(s)\right)\right)=0,\\
			&v^{t,\mu}(T,x)=g\left(x,m^{t,\mu}(T)\right).
		\end{aligned}
		\right.
	\end{equation}
	From \eqref{FP}-\eqref{HJB}, we know that the pair $\left(m^{t,\mu},v^{t,\mu}\right)$ is the solution to the system of HJB-FP equations corresponding to MFG \eqref{intro_1}. That is, our mean field game master equation \eqref{master} is a decoupling field of the mean field game HJB-FP system. We also refer to \cite{AB_JMPA,AB_SPA} for more discussions on the relations between the mean field game master equation and the HJB-FP system.
\end{remark}

\begin{example}[Also refer to \cite{AB_JMPA,AB_SPA,MR3489817}]

%\section{The linear-quadratic case}\label{sec:LQ}
	
We consider the following linear-quadratic (LQ) example: 
\begin{equation}\label{example:coefficients}
\begin{aligned}
    &g(x,m):=g_0(m)+ x^\top g_1 \overline{m} + \frac{1}{2}x^\top G x,\\
    &f(s,x,m,v):=f_0(s,m)+ x^\top f_1(s) \overline{m} +  v^\top f_2(s)\overline{m}+ \frac{1}{2}x^\top F_1(s)x+\frac{1}{2}v^\top F_2(s)v,
\end{aligned}
\end{equation}
where $\overline{m}:=\int_\brn y m(dy)$. Here, the maps
\begin{align*}
    &g_0:\pr_2(\brn)\to \br,\quad g_1\in\br^{n\times n},\quad G\in\br^{n\times n},\quad f_0:[0,T]\times\pr_2(\brn)\to \br,\\
    &f_1:[0,T]\to\br^{n\times n},\quad f_2:[0,T]\to\br^{d\times n},\quad F_1:[0,T]\to \br^{n\times n},\quad F_2:[0,T]\to \br^{d\times d};
\end{align*}
and the matrices $F_1(s),G,F_2(s)$ are bounded, symmetric and positive definite, i.e.,
\begin{align*}
    \left| F_1(s) \right|, \left| F_2(s) \right|, \left| G \right|\le L,\qquad F_1(s)\geq 2\lambda_x I_{n},\quad F_2(s)\geq 2\lambda_v I_{d},\quad G\geq 2\lambda_g I_{n}.
\end{align*}
By assuming the following Lipschitz continuous condition for $(f_0,g_0)$:
\begin{align*}
	&\left| D_y\frac{d f_0}{d \nu} (s,m')(y')-D_y\frac{d f_0}{d \nu} (s,m)(y)\right|+\left| D_y\frac{d g_0}{d \nu} (m')(y')-D_y\frac{d g_0}{d \nu} (m)(y)\right|\\
	\le\ & L\left(W_2(m,m')+|y'-y|\right),
\end{align*}
and assuming the following boundedness condition for $(f_1,f_2,g_1)$:
\begin{align*}
	&\left| f_1(s)\right|\le L_x, \quad \left| f_2(s) \right|\le L_v,\quad \left| g_1 \right|\le L_g,
\end{align*}
we know that Assumption (A2') holds true. If we assume that $\lambda_v>0,\ \lambda_x\geq\frac{L_v^2}{8\lambda_v}+\frac{L_x}{2}$ and  $\lambda_g\geq\frac{L_g}{2} $, then, Assumption (A3') holds true. Therefore, from Lemma~\ref{lem:2}, we have the well-posedness of FBSDEs \eqref{intro_2} with coefficients \eqref{example:coefficients}; and from Theorem~\ref{thm:1}, we have the classical solvability of the associated master equation. We can go further to attempt to give an explicit solution. Indeed, with \eqref{example:coefficients}, we know that $\hv (s,x, m,p)=-F_2(s)^{-1} \left(b_2(s)^\top p+f_2(s)\overline{m}\right)$, and then
\small
\begin{align*}
	H(s,x,m,p,q)=\ & p^\top \left[b_0(s, m )+b_1(s)x-b_2(s)F_2(s)^{-1} \left(b_2(s)^\top p+f_2(s)\overline{m}\right)\right]+\sum_{j=1}^n \left(q^j\right)^\top  \left[\sigma^j_0(s, m)+\sigma^j_1(s)x\right]\\
	&+f_0(s,m)+ \overline{m}^\top f_1(s)^\top x- \overline{m}^\top f_2(s)^\top F_2(s)^{-1} \left(b_2(s)^\top p+f_2(s)\overline{m}\right) \\
	&+ \frac{1}{2} F_1(s)x^{\otimes 2}+\frac{1}{2} F_2(s) \left[F_2(s)^{-1} \left(b_2(s)^\top p+f_2(s)\overline{m}\right)\right]^{\otimes 2}.
\end{align*}
\normalsize
The value functional 
$V$ in this LQ case should be the following quadratic form of $x$:
\begin{align*}
	V(t,x,\mu)=V_0(t,\mu)+V_1(t, \mu)^\top x+\frac{1}{2}x^\top V_2(t)x,
\end{align*}
where $V_1(t, \mu)\in\brn$,  and $V_2(t)\in\br^{n\times n}$ is symmetric. We have
\begin{align*}
	&\dd_t V(t,x,\mu)=\dd_t V_0(t,\mu)+ \dd_t{V}_1(t,\mu)^\top x+\frac{1}{2}x^\top \dot{ V}_2(t)x,\\
	&D_xV(t,x,\mu)=V_1(t, \mu)+V_2(t)x,\quad D_x^2V(t,x,\mu)=V_2(t),\\
	&D_y\frac{dV}{d\nu}(t,x,\mu)(y)=D_y\frac{dV_0}{d\nu}(t,\mu)(y)  +\left( D_y\frac{dV_0}{d\nu}(t,\mu)(y)\right)^\top x ,\\
	&D_y^2\frac{dV}{d\nu}(t,x,\mu)(y)=D_y^2\frac{dV_0}{d\nu}(t,\mu)(y)  +\left( D_y^2 \frac{dV_0}{d\nu}(t,\mu)(y)\right)^\top x  .
\end{align*}
%Therefore, we know that
%\small
%\begin{align*}
	%&H\left(t,x,\mu,D_x V(t,x,\mu),\frac{1}{2}D_x^2 V(t,x,\mu)\sigma(t,x,\mu)\right)\\
    %=\ & \frac{1}{2}\left\{ \text{Tr}\left[V_2(t)\sigma_1(t) \sigma_1(t)^\top\right]-   V_2(t) b_2(t) F_2(t)^{-1} b_2(t)^\top V_2(t) +V_2(t) b_1(t) +b_1(t)^\top V_2(t) +F_1(t) \right\} x^{\otimes2}\\
	%&+ x^\top \bigg\{  \left[b_1(t) -b_2(t)F_2(t)^{-1} b_2(t)^\top V_2(t)\right]^\top V_1(t)+ \text{Tr}\left[ V_2(t)\sigma_0(s)\sigma_1(t)^\top \right]\\
	%&\quad\qquad +V_2(t)b_0(t)-V_2(t)b_2(t)F_2(t)^{-1}  f_2(t)+f_1(t)\bigg\}\\
	%&+ \bigg\{V_1(t)^\top \left[b_0(t)-b_2(t)F_2(t)^{-1} \left(b_2(t)^\top V_1(t)+f_2(t)\right)\right]+ \frac{1}{2}\text{Tr}\left[V_2(t)\sigma_0(s)\sigma_0(s)^\top\right]\\
	%&\qquad +f_0(t,\mu)- f_2(t)^\top F_2(t)^{-1} \left[b_2(t)^\top V_1(t)+f_2(t)\right] \\
	%&\qquad +\frac{1}{2} F_2(t) \left[F_2(t)^{-1} \left(b_2(t)^\top V_1(t)+f_2(t)\right)\right]^{\otimes 2}\bigg\},
%\end{align*}
%\normalsize
%is a quadratic form of $x$, and 
%\begin{align*}
	%&\int_\brn \bigg[b \left(t,y,\mu,\hv(t,y,\mu,D_x V(t,y,\mu))\right)^\top D_y\frac{dV}{d\nu}(t,x,\mu)({y})\\
	%& \qquad +\frac{1}{2}\text{Tr}\left(\sigma (t,{y},\mu)^\top D_y^2\frac{dV}{d\nu}(t,x,\mu)({y})\sigma(t,{y},\mu)\right)\bigg] d\mu(y)\\
	%=\ & x^\top  \int_\brn \bigg\{ \left( D_y\frac{dV_0}{d\nu}(t,\mu)(y)\right) %\left[b_0(t)+b_1(t)y-b_2(t)F_2(t)^{-1} \left(b_2(t)^\top \left(V_1(t)+V_2(t)y\right)+f_2(t)\right)\right]    \\
	%&\quad\qquad +\frac{1}{2}\text{Tr}\left[ \left(\sigma\sigma^\top\right)(t,{y},\mu) D_y^2\frac{dV_1}{d\nu}(t,\mu)(y)   \right]  \bigg\}d\mu(y).
%\end{align*}
Applying these formulae in the mean field game master equation \eqref{master}, we see that $V_2$ solves the following Riccati equation:
\begin{equation}\label{Ricca}
	\begin{split}
		&\dot{V}_2(t)+ \sigma_1^\top(t) V_2(t)\sigma_1(t)-   V_2(t) b_2(t) F_2^{-1} (t)b_2^\top(t) V_2(t) \\
		&\: \qquad +V_2(t) b_1(t) +b_1^\top(t) V_2(t) +F_1(t) =0,\quad t\in[0,T), \qquad  V_2(T)=G.
	\end{split}
\end{equation}
The well-posedness for this generalized symmetric matrix Riccati equation \eqref{Ricca} is given in \cite{Riccati_Tang}. %when the coefficients are constants, \cite{Matrix_Riccati,Riccati_1,Riccati_2} gave the global solutions; by extending this method to the time-dependent cases, we can deduce that our Riccati equation \eqref{Ricca} has a global solution on $[0,T]$.
The vector-valued  functional $V_1$ solves the following equation:
\small
\begin{align}
	& \dd_t{V}_1(t,\mu)+\left[b_1(t) -b_2(t)F_2^{-1} (t) b_2^\top(t) V_2(t)\right]^\top V_1(t,\mu) +  \sigma_1^\top(t) V_2(t)\sigma_0(t,\mu) \notag \\
    &\qquad\qquad +V_2(t)b_0(t,\mu)-V_2(t)b_2(t)F_2^{-1} (t)  f_2(t)\overline{\mu}+f_1(t)\overline{\mu} \notag \\
	&\qquad\qquad + \int_\brn \bigg\{ \left( D_y\frac{dV_1}{d\nu}(t,\mu)(y)\right) \left[b_0(t,\mu)+b_1(t)y-b_2(t)F_2^{-1} (t) \left(b_2^\top(t) \left(V_1(t,\mu)+V_2(t)y\right)+f_2(t)\overline{\mu}\right)\right]     \notag \\
	&\qquad\qquad\qquad\qquad  +\frac{1}{2}\sum_{i,j=1}^n \left[ \left(\sigma\sigma^\top\right)_{ij}(t,{y},\mu) D_{y_iy_j}^2\frac{dV_1}{d\nu}(t,\mu)(y)   \right]  \bigg\}d\mu(y) =0,\quad t\in[0,T), \notag \\
	&V_1(T,\mu)=g_1\overline{\mu}. \label{V1_equa}
\end{align}
\normalsize
The well-posedness of \eqref{V1_equa} can be deduced with the standard fixed point argument. 
%for the iterative map $\Phi(V_1^n):=V_1^{n+1}$ where
%\small
%\begin{align*}
	%& \dd_t{V}^{n+1}_1(t,\mu)+\left[b_1(t) -b_2(t)F_2^{-1} (t) b_2^\top(t) V_2(t)\right]^\top V^{n+1}_1(t,\mu) \notag \\
	%&\quad\qquad + \int_\brn \bigg\{ \left( D_y\frac{dV^{n+1}_1}{d\nu}(t,\mu)(y)\right) \left[b_0(t,\mu)+b_1(t)y-b_2(t)F_2^{-1} (t) \left(b_2^\top(t) \left(V^n_1(t,\mu)+V_2(t)y\right)+f_2(t)\overline{\mu}\right)\right]     \notag \\
	%&\quad\qquad\qquad\qquad  +\frac{1}{2}\sum_{i,j=1}^n \left[ \left(\sigma\sigma^\top\right)_{ij}(t,{y},\mu) D_{y_iy_j}^2\frac{dV^{n+1}_1}{d\nu}(t,\mu)(y)   \right]  \bigg\}d\mu(y) +  \sigma_1^\top(t) V_2(t)\sigma_0(t,\mu) \notag \\
	%&\quad\qquad +V_2(t)b_0(t,\mu)-V_2(t)b_2(t)F_2^{-1} (t)  f_2(t)\overline{\mu}+f_1(t)\overline{\mu}=0,\quad t\in[0,T), \notag \\
	%&V^{n+1}_1(T,\mu)=g_1\overline{\mu},
%\end{align*}
\normalsize
%when the symmetric matrix $F_2(t)\geq \lambda I$. 
%and the well-posedness of  $V_1^{n+1}$ can be obtained by the approach for the  following equation \eqref{V0_equa}. %For the very special case, when the coefficients $b_0$, $\sigma_0$, $f_1$, $f_2$ and $g_1$ are independent of $\mu$, then, $V_1$ is  independent of $\mu$ and \eqref{V1_equa} reduces to a standard linear ordinary differential equation.  
Moreover, $V_0$ solves the following equation:
\small
\begin{align}
	&\dd_t V_0(t,\mu) + \frac{1}{2}\sigma_0^\top(t ,\mu)V_2(t)\sigma_0(t ,\mu) +f_0(t,\mu) - \overline{\mu}^\top f_2^\top(t) F_2^{-1}(t) \left[b_2^\top(t) V_1(t ,\mu)+f_2(t)\overline{\mu}\right] \notag \\
    &+\frac{1}{2}  F_2^{-1}(t) \left[b_2^\top(t) V_1(t ,\mu)+f_2(t)\overline{\mu}\right]^{\otimes 2} +V_1^\top (t ,\mu) \left[b_0(t ,\mu)-b_2(t)F_2^{-1}(t) \left(b_2^\top (t) V_1(t ,\mu)+f_2(t)\overline{\mu}\right)\right] \notag \\
    &+\int_\brn \bigg\{ \left[b_0(t ,\mu)+b_1(t)y-b_2(t)F_2^{-1}(t) \left(b_2^\top(t) \left(V_1(t ,\mu)+V_2(t)y\right)+f_2(t)\overline{\mu}\right)\right]^\top  D_y\frac{dV_0}{d\nu}(t,\mu)(y) \notag \\
	&\quad\qquad +\frac{1}{2}\text{Tr}\left[ \left(\sigma\sigma^\top\right)(t,{y},\mu) D_y^2\frac{dV_0}{d\nu}(t,\mu)(y)   \right]  \bigg\}d\mu(y) =0,\quad t\in[0,T), \notag \\
	&V_0(T,\mu)=g_0(\mu). \label{V0_equa}
\end{align}
\normalsize
Note that \eqref{V0_equa} can be viewed as a mean field master equation, whose well-posedness has been established in the literature; see \cite{AB5,AB9'} for instance. We also refer to \cite{AB_SPA,MR3489817} and \cite[Section 5.4]{AB_JMPA} for more discussions for the case of constant diffusion. 
\end{example}

\section*{Acknowledgement}

Alain Bensoussan is supported by the National Science Foundation under grant NSF-DMS-2204795. Ziyu Huang acknowledges the financial supports as a postdoctoral fellow from Department of Statistics of The Chinese University of Hong Kong. Shanjian Tang is supported by the National Natural Science Foundation of China under grant nos. 12031009 and 11631004. Phillip Yam acknowledges the financial supports from HKGRF-14301321 with the project title ``General Theory for Infinite Dimensional Stochastic Control: Mean Field and Some Classical Problems'', and HKGRF-14300123 with the project title ``Well-posedness of Some Poisson-driven Mean Field Learning Models and their Applications''. 

%%%%%%%%%%%%%%%%%%%%%%%%%%%%%%%%%%%%%%%%%%%%%%%%%%%%%%%%%%%%%%%%%%%%

%%%%%%%%%%%%%%%%%%%%%%%%%%%%%%%%%%%%%%%%%%%%%%%%%%%%%%%%%%%%%%%%%%%%%%%%%%%%
\newpage

\normalsize 
\appendix
\appendixpage
\addappheadtotoc

\section{Proof of Statements in Section~\ref{sec:MP}}\label{pf:MP}

\subsection{Proof of Lemma~\ref{lem:MP1}}\label{pf_lem_MP1}
We first prove the case when Assumption (A3) is satisfied. For any control $v\in \lr_{\f}^2(t,T)$, we denote by $X^v$ the corresponding state. From Assumptions (A2) and (A3), we have
\begin{align}
	&J_{t\xi}\left(v;\lr(Y_{t\xi}(s)),0\le s\le T\right)-J\left(v_{t\xi};\lr(Y_{t\xi}(s)),0\le s\le T\right) \notag\\
	=\ &\e\Bigg[ \int_t^T \Big(f(s,Y_{t\xi}(s),\lr(Y_{t\xi}(s)),v(s))-f(s,Y_{t\xi}(s),\lr(Y_{t\xi}(s)),v_{t\xi}(s)) \notag\\
	&\qquad\qquad + f(s,X^v(s),\lr(Y_{t\xi}(s)),v(s))-f(s,Y_{t\xi}(s),\lr(Y_{t\xi}(s)),v(s))\Big)ds \notag\\
	&\quad +g(X^v(T),\lr(Y_{t\xi}(T)))-g(Y_{t\xi}(T),\lr(Y_{t\xi}(T)))\Bigg] \notag\\
	\geq \ & \e\Bigg[\int_t^T \Bigg(\lambda_v |v(s)-v_{t\xi}(s)|^2+\left(D_vf (s,Y_{t\xi}(s),\lr(Y_{t\xi}(s)),v_{t\xi}(s))\right)^\top (v(s)-v_{t\xi}(s)) \notag\\
	&\qquad\qquad+\int_0^1 \left(D_xf (s,Y_{t\xi}(s)+\delta(X^v(s)-Y_{t\xi}(s)),\lr(Y_{t\xi}(s)),v_{t\xi}(s))\right)^\top (X^v(s)-Y_{t\xi}(s))d\delta\Bigg)ds \notag\\
	&\quad + \int_0^1 \left(D_xg (Y_{t\xi}(T)+\delta(X^v(T)-Y_{t\xi}(T)),\lr(Y_{t\xi}(T)))\right)^\top (X^v(T)-Y_{t\xi}(T))d\delta\Bigg] \notag\\
	\geq\ &  \e\Bigg[\int_t^T \Big(\lambda_v |v(s)-v_{t\xi}(s)|^2+\left(D_vf (s,Y_{t\xi}(s),\lr(Y_{t\xi}(s)),v_{t\xi}(s))\right)^\top (v(s)-v_{t\xi}(s)) \notag\\
	&\qquad\qquad+ \left(D_xf (s,Y_{t\xi}(s),\lr(Y_{t\xi}(s)),v_{t\xi}(s))\right)^\top (X^v(s)-Y_{t\xi}(s))- L |X^v(s)-Y_{t\xi}(s)|^2\Big)ds \notag\\
	&\quad +  (D_xg (Y_{t\xi}(T),\lr(Y_{t\xi}(T))))^\top (X^v(T)-Y_{t\xi}(T))-L |X^v(T)-Y_{t\xi}(T)|^2 \bigg]. \label{MP1_1}
\end{align}    
From It\^o's formula and Assumption (A1), we have for $s\in(t,T)$,
\begin{align*}
	\frac{d}{ds}\e\left[p_{t\xi} (s)^\top (X^v(s)-Y_{t\xi}(s))\right]=\e\bigg[&\bigg(p_{t\xi} (s)^\top b_2(s)+\sum_{j=1}^n \left(q^j_{t\xi} (s)\right)^\top \sigma^j_2(s)\bigg)\left(v(s)-v_{t\xi}(s)\right)\\
	& -(D_xf (s,Y_{t\xi}(s),\lr(Y_{t\xi}(s)), v_{t\xi}(s)) )^\top \left(X^v(s)-Y_{t\xi}(s)\right)\bigg].
\end{align*}
From Equation \eqref{hv}, we have
\begin{align*}
	p_{t\xi} (s)^\top b_2(s)+\sum_{j=1}^n \left(q^j_{t\xi} (s)\right)^\top \sigma^j_2(s)+(D_vf (s,Y_{t\xi}(s),\lr(Y_{t\xi}(s)),v_{t\xi}(s)))^\top=0, \quad s\in[t,T].
\end{align*}
Therefore, we deduce that	
\begin{equation}\label{MP1_2}
	\begin{split}
		&\e\left[(D_x g (Y_{t\xi}(T),\lr(Y_{t\xi}(T))))^\top \left(X^v(T)-Y_{t\xi}(T)\right) \right]\\
		=\ & -\e\Bigg[\int_t^T \Big( (D_vf (s,Y_{t\xi}(s),\lr(Y_{t\xi}(s)),v_{t\xi}(s)) )^\top \left(v(s)-v_{t\xi}(s)\right)\\
		&\quad\qquad\qquad +(D_xf (s,Y_{t\xi}(s),\lr(Y_{t\xi}(s)), v_{t\xi}(s)) )^\top \left(X^v(s)-Y_{t\xi}(s)\right) \Big) ds \Bigg].
	\end{split}
\end{equation}
Substituting \eqref{MP1_2} into \eqref{MP1_1}, we have
\begin{align*}
	&J_{t\xi}\left(v;\lr(Y_{t\xi}(s)),0\le s\le T\right)-J\left(v_{t\xi};\lr(Y_{t\xi}(s)),0\le s\le T\right) \\
	\geq\ & \e  \left[\int_t^T \Big(\lambda_v |v(s)-v_{t\xi}(s)|^2 - L |X^v(s)-Y_{t\xi}(s)|^2\Big)ds  -L |X^v(T)-Y_{t\xi}(T)|^2 \right].
\end{align*}
With standard Picard iteration arguments of SDEs, we can deduce that
\begin{equation*}
	\e\bigg[\sup_{t\le s\le T}\left| X^v(s)-Y_{t\xi}(s)\right|^2\bigg]\le C(L,T) \e\left[\int_t^T \left|v(s)-v_{t\xi}(s)\right|^2ds\right],
\end{equation*}
where $C(L,T)$ is a constant depending only on $(L,T)$. Therefore, we have
\begin{align*}
	&J_{t\xi}\left(v;\lr(Y_{t\xi}(s)),0\le s\le T\right)-J\left(v_{t\xi};\lr(Y_{t\xi}(s)),0\le s\le T\right) \\
	\geq\ & [\lambda_v-C(L,T)] \e\left[\int_t^T |v(s)-v_{t\xi}(s)|^2 ds   \right],
\end{align*}
so we know that the stochastic control problem with the fixed distribution flow $\left\{\lr(Y_{t\xi}(s),\ t\le s\le T\right\}$ has a unique optimal control as in \eqref{vxi}. Therefore, \eqref{vxi} gives a solution for MFG \eqref{intro_1}. For the uniqueness result for \eqref{vxi}, suppose that $v'\in\lr_\f^2(t,T)$ is another solution for MFG \eqref{intro_1}, and we denote by $X'$ its related controlled process. For the stochastic control problem in \eqref{intro_1} with the fixed distribution flow $\left\{\lr\left(X'_s\right),\ t\le s\le T \right\}$, we know that $v'$ is an optimal control, and we denote by $\left(P',Q'\right)$ the corresponding adjoint process. Then, from the necessary condition of the maximum principle (see \cite{MR1696772} for instance), we know that $\left(X',P',Q'\right)$ is also a solution for FBSDEs \eqref{intro_2}. From the uniqueness of the solution for FBSDEs \eqref{intro_2}, we know that $\left(X',P',Q'\right)=(X,P,Q)$, and then, $v'=v_{t\xi}$ in $\lr_\f^2(t,T)$, from which we deduce the claim. Further, if  Assumption (A3') is satisfied, we know that
\begin{align*}
	&J_{t\xi}\left(v;\lr(Y_{t\xi}(s)),0\le s\le T\right)-J\left(v_{t\xi};\lr(Y_{t\xi}(s)),0\le s\le T\right) \\
	\geq \ & \e\Bigg[\int_t^T \Big(\lambda_x |X^v(s)-Y_{t\xi}(s)|^2+(D_xf (s,Y_{t\xi}(s),\lr(Y_{t\xi}(s)),v_{t\xi}(s)))^\top (X^v(s)-Y_{t\xi}(s)) \\
	&\qquad\qquad+\lambda_v |v(s)-v_{t\xi}(s)|^2+(D_vf (s,Y_{t\xi}(s),\lr(Y_{t\xi}(s)),v_{t\xi}(s)))^\top (v(s)-v_{t\xi}(s))\Big)ds \\
	&\quad + \lambda_g |X^v(T)-Y_{t\xi}(T)|^2 +(D_xg (Y_{t\xi}(T),\lr(Y_{t\xi}(T))) )^\top (X^v(T)-Y_{t\xi}(T)) \Bigg] .  
\end{align*}    
Then, from Equation \eqref{MP1_2}, we have the second assertion.

\subsection{Proof of Lemma~\ref{lem:2}}\label{pf_lem2}

In view of \eqref{def:BAFG} and Lemma~\ref{lem:1}, we only need to first check Condition~\ref{Condition_generic} (i) and (ii), and then (iii) under our settings. From Assumption (A1), the definition of $\hv(\cdot)$ in \eqref{hv} and Equation
\eqref{def:BAFG}, for $s\in[t,T]$, by denoting $\hv:=\hv(s,X,\lr(X),p,q),\quad \hv':=\hv\left(s,X',\lr(X'),p',q'\right)$, we have
\begin{equation}\label{lem2_3}
	\begin{split}
		&\e\bigg[\left(\mathbf{F}(s,X',p',q')-\mathbf{F}(s,X,p,q)\right)^\top  (X'-X)+\left(\mathbf{B}(s,X',p',q')-\mathbf{B}(s,X,p,q)\right)^\top  (p'-p)\\
		&\quad +\sum_{j=1}^n \left(\mathbf{A}^j(s,X',p',q')-\mathbf{A}^j(s,X,p,q)\right)^\top  \left({q'}^{j}-q^j\right)\bigg]\\
		\ =& -\e\Big[\left(D_vf\left(s,X',\lr(X'),\hv'\right)-D_vf\left(s,X,\lr(X),\hv \right)\right)^\top   \left(\hv'-\hv\right)\\
		&\qquad +\left(D_xf(s,X',\lr(X'),\hv')-D_xf(s,X,\lr(X),\hv)\right)^\top  (X'-X)\Big]\\
		&+\e\bigg[\left(b_0(s,\lr(X'))-b_0(s,\lr(X))\right)^\top  (p'-p)\\
		&\qquad+\sum_{j=1}^n \left(\sigma^j_0(s,\lr(X'))-\sigma^j_0(s,\lr(X))\right)^\top  \left({q'}^{j}-q^j\right)\bigg],
	\end{split}
\end{equation}    
From Assumption (A3), we have    
\begin{equation}\label{lem2_5}
	\begin{split}
		&\e\left[\left(D_vf\left(s,X,\lr(X),\hv'\right)-D_vf\left(s,X,\lr(X),\hv\right)\right)^\top  \left(\hv'-\hv\right)\right]\geq  2\lambda_v \e\left[\left|\hv'-\hv\right|^2\right].
	\end{split}
\end{equation}
From Assumption (A2) and the Young's inequality with a weight of $\frac{\lambda_v}{L}$, we have
\begin{align*}%\label{lem2_6}
		&-\e\big[\left(D_vf\left(s,X',\lr(X'),\hv\right)-D_vf(s,X,\lr(X),\hv)\right)^\top    \left(\hv'-\hv\right)\\
		&\qquad +\left(D_xf\left(s,X',\lr(X'),\hv'\right)-D_xf\left(s,X,\lr(X),\hv\right)\right)^\top \left(X'-X\right)\big]\\
		&+\e\bigg[\left(b_0(s,\lr(X'))-b_0(s,\lr(X))\right)^\top   \left(p'-p\right) +\sum_{j=1}^n \left(\sigma^j_0(s,\lr(X'))-\sigma^j_0 (s,\lr(X))\right)^\top  \left({q'}^{j}-q^j\right)\bigg]\\
		\le\ & L \e\bigg[\left(|X'-X|+\e\left[|X'-X|^2\right]^{\frac{1}{2}}\right) \left|\hv'-\hv\right| +\left(|X'-X|+\e\left[|X'-X|^2\right]^{\frac{1}{2}}+ \left|\hv'-\hv\right| \right)\left|X'-X\right| \\
		&\qquad+\bigg(|p'-p|+\sum_{j=1}^n|q'^{j}-q^j|\bigg) \e\left[|X'-X|^2\right]^{\frac{1}{2}} \bigg]\\
		\le\ & \lambda_v \e\left[\left|\hv'-\hv\right|^2\right]+C(L)\left(1+\frac{1}{\lambda_v}\right)\e\left[|X'-X|^2+|p'-p|^2+ |q'-q|^2\right].
\end{align*}    
Subsitituting \eqref{lem2_5} and the last inequality back into \eqref{lem2_3}, Condition \ref{Condition_generic} (i) is satisfied with the map $\beta(s,X,p,q)(\omega):=\hv(s,X(\omega),\lr(X),p(\omega),q(\omega))$ by setting $\Lambda=\lambda_v$ and $\alpha=C(L)\left(1+\frac{1}{\lambda_v}\right)$. The Lipschitz-continuities Condition \ref{Condition_generic} (ii) are also satisfied with $ K =C(L)$. 

If Assumption (A3') is satisfied, then, from Cauchy-Schwarz inequality and the Young's inequality with a weight of $\frac{\lambda_v}{L}$, we have
\begin{align}
		& -\e\Big[\left(D_vf\left(s,X',\lr(X'),\hv'\right)-D_vf\left(s,X,\lr(X),\hv \right)\right)^\top   \left(\hv'-\hv\right) \notag \\
		&\qquad +\left(D_xf(s,X',\lr(X'),\hv')-D_xf(s,X,\lr(X),\hv)\right)^\top  (X'-X)\Big] \notag \\
		\le \ &-\e\Big[\left(D_vf\left(s,X',\lr(X'),\hv'\right)-D_vf\left(s,X,\lr(X'),\hv \right)\right)^\top   \left(\hv'-\hv\right) \notag \\
		&\qquad +\left(D_xf(s,X',\lr(X'),\hv')-D_xf(s,X,\lr(X'),\hv)\right)^\top  (X'-X)\Big] \notag \\
		&+ \e\Big[\left|D_vf\left(s,X,\lr(X),\hv\right)-D_vf\left(s,X,\lr(X'),\hv \right)\right|  \left|\hv'-\hv\right| \notag \\
		&\qquad +\left|D_xf(s,X,\lr(X),\hv)-D_xf(s,X,\lr(X'),\hv)\right| \left|X'-X\right|\Big] \notag \\
		\le\ & -2 \e\left[\lambda_v \left| \hv'-\hv\right|^2+\lambda_x \left| X'-X \right|^2\right] \notag \\
		&+L_v\e\left[\left|X'-X\right|^2\right]^{\frac{1}{2}}\e\left[\left|\hv'-\hv\right|^2\right]^{\frac{1}{2}}+L_x \e\left[\left|X'-X\right|^2\right] \notag \\
		\le\ & -\e\left[\lambda_v \left| \hv'-\hv\right|^2 +\left(2\lambda_x-\frac{L_v^2}{4\lambda_v}-L_x \right) \left| X'-X \right|^2\right] \notag \\
		\le\ &-\e\left[\lambda_v \left| \hv'-\hv\right|^2 \right]. \label{small_mf_eff_3}
\end{align} 
In the last inequality, we use the small mean field effect condition in Assumption (A3'): $\lambda_x\geq\frac{L_v^2}{8\lambda_v}+\frac{L_x}{2}$. From \eqref{lem2_3}, we know that Condition \ref{Condition_generic} (i) is valid for $\alpha=0$. In a similar way, 
\begin{align*}
	&\e\left[\left(\mathbf{G}(X')-\mathbf{G}(X)\right)^\top  (X'-X)\right] \notag \\
	=\ & \e\left[\left(D_xg\left(X',\lr(X')\right)-D_xg\left(X,\lr(X)\right)\right)^\top   \left(X'-X\right)\right] \notag \\
	\geq\ & \e\left[\left(D_x g(X',\lr(X'))-D_xg(X,\lr(X'))\right)^\top  (X'-X)\right] \notag \\
	&- \e\left[\left|D_xg(X,\lr(X))-D_xg(X,\lr(X'))\right| \left|X'-X\right|\right] \notag \\
	\geq\ & \left(2 \lambda_g-L_g\right)  \e\left[\left| X'-X \right|^2\right] \geq 0, \notag
\end{align*} 
and therefore,  Condition \ref{Condition_generic} (iii) is satisfied. As a consequence of Lemma~\ref{lem:1}, we obtain the well-posedness for FBSDEs \eqref{intro_2}. From Estimates \eqref{lem1_1} and \eqref{lem1_2}, the proof is complete. 

\subsection{Proof of Lemma~\ref{lem:5}}\label{pf_lem5}

The proof for the solvability is similar as Lemma~\ref{lem:2}. Here, we only give a proof of Estimate \eqref{lem5_2}. For any $\xi,\xi'\in L_{\f_t}^2$ such that $\lr(\xi)=\mu,\ \lr(\xi')=\mu'$, we denote by $(Y_{t\xi},p_{t\xi},q_{t\xi})$ and $(Y_{t\xi'},p_{t\xi'},q_{t\xi'})$ the respective solutions of FBSDEs \eqref{intro_2} corresponding to initials $\xi$ and $\xi'$. From Lemma~\ref{lem:2}, we know that
\begin{equation}\label{lem5_3}
	\begin{split}
		&\e\bigg[\sup_{t\le s\le T}|Y_{t\xi'}(s)-Y_{t\xi}(s)|^2\bigg]\le C(L,T,\lambda_v)\e\left[|\xi'- \xi|^2\right].
	\end{split}
\end{equation}	
We denote by $\de x:=x'-x$, $\de\xi:=\xi'-\xi$, and for $s\in[t,T]$,
\begin{align*}
	&\de Y(s):=Y_{tx'\mu'}-Y_{tx\mu}(s),\quad \de p(s):=p_{tx'\mu'}-p_{tx\mu}(s),\quad \de q(s):=q_{tx'\mu'}-q_{tx\mu}(s),\\
	&\de v(s):=v_{tx'\mu'}-v_{tx\mu}(s).
\end{align*}
Then, the processes $(\Delta Y,\Delta p,\Delta q)$ satisfy the following FBSDEs:
\begin{align}
		\de Y(s)=\ & \de x+\int_t^s \left[b_0(r,\lr(Y_{t\xi'}(r)))-b_0(r,\lr(Y_{t\xi}(r)))+b_1(r)\de Y(r)+b_2(r)\de v(r)\right]dr \notag \\
		&+\int_t^s \left[\sigma_0(r,\lr(Y_{t\xi'}(r)))-\sigma_0(r,\lr(Y_{t\xi}(r)))+\sigma_1(r)\de Y(r)+\sigma_2(r)\de v(r)\right]dB(r), \notag \\
		\de p(s)=\ & D_x g\left(Y_{tx'\mu'}(T),\lr\left(Y_{t\xi'}(T)\right)\right)-D_x g\left(Y_{tx\mu}(T),\lr\left(Y_{t\xi}(T)\right)\right) \notag \\
		&+\int_s^T \bigg[b_1(r)^\top \de p(r)+\sum_{j=1}^n \left(\sigma^j_1 (s)\right)^\top \de q^j(s)+D_x f\left(r,Y_{tx'\mu'}(r),\lr(Y_{t\xi'}(r)),v_{tx'\mu'}(r)\right) \notag \\
		&\qquad\qquad -D_x f\left(r,Y_{tx\mu}(r),\lr(Y_{t\xi}(r)),v_{tx\mu}(r)\right)\bigg]dr-\int_s^T \de q(r)dB(r),\quad  s\in[t,T]. \label{lem5_4}
\end{align}	
With standard arguments of SDEs and BSDEs in \cite{MR1696772}, we have
\begin{equation}\label{lem5_9}
	\begin{split}
		& \e\bigg[\sup_{t\le s\le T}\left|\left(\de Y(s),\de p(s)\right)^\top \right|^2+\int_t^T |\de q(s)|^2 ds\bigg]\\
		\le\ & C(L,T)\e\bigg[|\de x|^2+\sup_{t\le s\le T}\left|Y_{t\xi'}(s)-Y_{t\xi}(s)\right|^2+\int_t^T |\de v(s)|^2ds\Big].
	\end{split}
\end{equation}
From \eqref{hv}, we have for $s\in[t,T]$,	
\begin{equation*}\label{lem5_5}
	\begin{split}
		&b_2(s)^\top\de p(s)+\sum_{j=1}^n \left(\sigma^j_2 (s) \right)^\top \de q^j(s)\\
		=\ & -D_vf\left(s,Y_{tx'\mu'},\lr(Y_{t\xi'}(s)),v_{tx'\mu'}(s)\right)+D_vf\left(s,Y_{tx\mu},\lr(Y_{t\xi}(s)),v_{tx\mu}(s)\right)=0,
	\end{split}
\end{equation*}
then, by It\^o's formula, we have
\begin{equation}\label{lem5_6}
	\begin{split}
		&\e\left[\de p(T)^\top \de Y(T)-\de p(t)^\top \de x\right]\\
		=\ & \e\Bigg\{\int_t^T \Bigg[\left(b_0(\lr(Y_{t\xi'}(s)),s)-b_0(\lr(Y_{t\xi}(s)),s)\right)^\top \de p(s)\\
		&\qquad\qquad +\sum_{j=1}^n\left(\sigma^j_0(\lr(Y_{t\xi'}(s)),s)-\sigma^j_0(\lr(Y_{t\xi}(s)),s)\right)^\top \de q^j(s)\\
		&\qquad\qquad -\left[\left(
		\begin{array}{cc}
			D_x f\\
			D_v f
		\end{array}
		\right) \left|^{\left(s,Y_{tx'\mu'}(s),\lr(Y_{t\xi'}(s)),v_{tx'\mu'}(s)\right)}_{\left(s,Y_{tx\mu}(s),\lr(Y_{t\xi}(s)),v_{tx\mu}(s)\right)} \right.\right]^\top \left(
		\begin{array}{cc}
			\de Y(s)\\
			\de v(s)
		\end{array}
		\right) \Bigg]ds\Bigg\}.
	\end{split}
\end{equation}	
From Assumptions (A2) and (A3) and the Young's inequality with a weight of $\frac{\lambda_v}{L}$, we have
\begin{align*}
	&\left[\left(
	\begin{array}{cc}
		D_x f\\
		D_v f
	\end{array}
	\right) \left|^{\left(s,Y_{tx'\mu'}(s),\lr(Y_{t\xi'}(s)),v_{tx'\mu'}(s)\right)}_{\left(s,Y_{tx\mu}(s),\lr(Y_{t\xi}(s)),v_{tx\mu}(s)\right)}\right.  \right]^\top \left(
	\begin{array}{cc}
		\de Y(s)\\
		\de v(s)
	\end{array}
	\right) \notag \\
	=\ & \left(D_v f(s,Y_{tx\mu}(s),\lr(Y_{t\xi}(s)),v_{tx'\mu'}(s))-D_v f(s,Y_{tx\mu}(s),\lr(Y_{t\xi}(s)),v_{tx\mu}(s))\right)^\top  \de v(s) \notag\\
	&+\left(D_x f(s,Y_{tx'\mu'}(s),\lr(Y_{t\xi'}(s)),v_{tx'\mu'}(s))-D_x f(s,Y_{tx\mu}(s),\lr(Y_{t\xi}(s)),v_{tx\mu}(s))\right)^\top  \de Y(s) \notag\\
	&+\left(D_v f(s,Y_{tx'\mu'}(s),\lr(Y_{t\xi'}(s)),v_{tx'\mu'}(s))-D_v f(s,Y_{tx\mu}(s),\lr(Y_{t\xi}(s)),v_{tx'\mu'}(s))\right)^\top  \de v(s) \notag\\
	\geq \ & 2\lambda_v|\de v(s)|^2-2L\left(|\de Y(s)|+\e\left[\left|Y_{t\xi'}(s)-Y_{t\xi}(s)\right|^2\right]^{\frac{1}{2}}\right)\left(|\de v(s)|+|\de Y(s)|\right) \notag \\
	\geq \ &  \lambda_v |\de v(s)|^2-C(L)\left(1+\frac{1}{\lambda_v}\right)\left(|\de Y(s)|^2+\e\left[\left|Y_{t\xi'}(s)-Y_{t\xi}(s)\right|^2\right]\right). 
\end{align*}
Applying the last inequality to \eqref{lem5_6}, from Assumptions (A1) and (A2), we have
\begin{equation*}\label{lem5_8}
	\begin{split}
		&\lambda_v \e\left[\int_t^T |\de v(s)|^2 ds\right]\\
		\le\ & C(L,T)\left(1+\frac{1}{\lambda_v}\right)\e\bigg[\sup_{t\le s\le T}\left|\left(\de Y(s),\de p(s),Y_{t\xi'}(s)-Y_{t\xi}(s)\right)\right|^2+  \int_t^T |\de q(s)|^2 ds\bigg].
	\end{split}
\end{equation*}
Substituting \eqref{lem5_9}	into this last inequality, we have
\begin{align*}
	&\lambda_v \e\left[\int_t^T |\de v(s)|^2 ds\right]\\
	\le\ & C(L,T)\left(1+\frac{1}{\lambda_v}\right)\e\bigg[|\de x|^2+\sup_{t\le s\le T}\left|Y_{t\xi'}(s)-Y_{t\xi}(s)\right|^2+  \int_t^T |\de v(s)|^2 ds\bigg].
\end{align*}
Therefore, there exists a constant $c(L,T)$ depending only on $(L,T)$, such that when $\lambda_v\geq c(L,T)$,
\begin{equation*}\label{lem5_11}
	\begin{split}
		\e\left[\int_t^T |\de v(s)|^2 ds\right]\le &C(L,T,\lambda_v)\e\bigg[|\de x|^2+\sup_{t\le s\le T}\left|Y_{t\xi'}(s)-Y_{t\xi}(s)\right|^2\bigg],
	\end{split}
\end{equation*}
Substituting \eqref{lem5_3} into the last inequality, we have
\begin{align*}
	\e\left[\int_t^T |\de v(s)|^2 ds\right]\le &C(L,T,\lambda_v)\e\left[|\de x|^2+|\de\xi|^2\right].
\end{align*}
As the choices of $\xi$ and $\xi'$ are arbitrary, we therefore obtain \eqref{lem5_2}.

If Assumption (A3') is satisfied, then, from Cauchy-Schwarz inequality and Young's inequality, we have
\begin{align}
	&\left[\left(
	\begin{array}{cc}
		D_x f\\
		D_v f
	\end{array}
	\right) \left|^{\left(s,Y_{tx'\mu'}(s),\lr(Y_{t\xi'}(s)),v_{tx'\mu'}(s)\right)}_{\left(s,Y_{tx\mu}(s),\lr(Y_{t\xi}(s)),v_{tx\mu}(s)\right)}\right. \right]^\top \left(
	\begin{array}{cc}
		\de Y(s)\\
		\de v(s)
	\end{array}
	\right) \notag \\
	=\ & \left[\left(
	\begin{array}{cc}
		D_x f\\
		D_v f
	\end{array}
	\right) \left|^{\left(s,Y_{tx'\mu'}(s),\lr(Y_{t\xi'}(s)),v_{tx'\mu'}(s)\right)}_{\left(s,Y_{tx\mu}(s),\lr(Y_{t\xi'}(s)),v_{tx\mu}(s)\right)} \right. \right]^\top \left(
	\begin{array}{cc}
		\de Y(s)\\
		\de v(s)
	\end{array}
	\right) \notag\\
	&+ \left[\left(
	\begin{array}{cc}
		D_x f\\
		D_v f
	\end{array}
	\right) \left|^{\left(s,Y_{tx\mu}(s),\lr(Y_{t\xi'}(s)),v_{tx\mu}(s)\right)}_{\left(s,Y_{tx\mu}(s),\lr(Y_{t\xi}(s)),v_{tx\mu}(s)\right)} \right. \right]^\top \left(
	\begin{array}{cc}
		\de Y(s)\\
		\de v(s)
	\end{array}
	\right) \notag \\
	\geq\ & 2\lambda_v \left| \de v(s)\right|^2+2\lambda_x \left| \de Y(s)\right|^2 -\left(L_v |\de v(s)| +L_x |\de Y(s)|\right)\e\left[ \left| Y_{t\xi'}(s)-Y_{t\xi}(s) \right|^2\right]^{\frac{1}{2}} \notag\\
	\geq \ & \lambda_v \left| \de v(s)\right|^2+\left(2\lambda_x-L_x \right) \left| \de Y(s)\right|^2 -\left(\frac{L_v^2}{4\lambda_v}+\frac{L_x}{4}\right)\e\left[ \left| Y_{t\xi'}(s)-Y_{t\xi}(s) \right|^2\right]^{\frac{1}{2}} \notag \\
	\geq\ & \lambda_v \left| \de v(s)\right|^2 -C(L) \left(1+\frac{1}{\lambda_v}\right)\e\left[ \left| Y_{t\xi'}(s)-Y_{t\xi}(s) \right|^2\right]. \label{lem5_12}
\end{align}
Similarly, for any $\epsilon>0$, we have
\begin{align*}
	&\e[\de p(T)^\top \de Y(T)-\de p(t)^\top \de x]\notag\\
	=\ & \e\left[ \left[D_x g\left|^{\left(Y_{tx'\mu'}(T),\lr\left(Y_{t\xi'}(T)\right)\right)}_{\left(Y_{tx\mu}(T),\lr\left(Y_{t\xi'}(T)\right)\right)}\right.\right]^\top  \de Y(T)+\left[D_x g\left|^{\left(Y_{tx\mu}(T),\lr\left(Y_{t\xi'}(T)\right)\right)}_{\left(Y_{tx\mu}(T),\lr\left(Y_{t\xi}(T)\right)\right)}\right.\right]^\top  \de Y(T)-\de p(t)^\top \de x \right]\notag\\
	\geq\ & \e\left[ 2 \lambda_g |\de Y(T)|^2-L_g \e\left[ \left| Y_{t\xi'}(T)-Y_{t\xi}(T) \right|^2\right]^{\frac{1}{2}} |\de Y(T)|-|\de p(t)||\de x| \right]\notag\\
	\geq\ & \left(2\lambda_g-L_g \right) \e\left[  |\de Y(T)|^2\right]- \frac{L_g}{4} \e\left[ \left| Y_{t\xi'}(T)-Y_{t\xi}(T) \right|^2\right]- \epsilon\e\left[|\de p(t)|^2\right]-\frac{1}{4\epsilon}|\de x|^2 \notag \\
	\geq\ & - \epsilon\e\left[|\de p(t)|^2\right]-\frac{1}{4\epsilon}|\de x|^2-C(L)\e\left[ \left| Y_{t\xi'}(T)-Y_{t\xi}(T) \right|^2\right].
\end{align*}
Substituting \eqref{lem5_12} and the last inequality into \eqref{lem5_6}, we have
\begin{align*}
	& \lambda_v \e\left[\int_t^T \left| \de v(s)\right|^2  \right]ds\\
	\le \ & \epsilon\e\left[|\de p(t)|^2\right]+\frac{1}{4\epsilon}|\de x|^2+C(L,T) \left(1+\frac{1}{\lambda_v}\right)\e\bigg[ \sup_{t\le s\le T}\left| Y_{t\xi'}(s)-Y_{t\xi}(s) \right|^2\bigg].
\end{align*}
Substituting \eqref{lem5_9}	into the last inequality, we have
\begin{align*}
	\lambda_v \e\left[\int_t^T \left| \de v(s)\right|^2  \right]ds \le \ & \epsilon C(L,T)\e\left[\int_t^T |\de v(s)|^2ds\right]+C(L,T)\left(1+\frac{1}{\epsilon}\right)|\de x|^2\\
	&+C(L,T) \left(1+\frac{1}{\lambda_v}\right)\e\bigg[ \sup_{t\le s\le T}\left| Y_{t\xi'}(s)-Y_{t\xi}(s) \right|^2\bigg].
\end{align*}
By choosing $\epsilon:=\frac{\lambda_v}{2 C(L,T)}$, we have
\begin{align*}
	& \frac{\lambda_v}{2} \e\left[\int_t^T \left| \de v (s)\right|^2  \right]ds\le  C(L,T)\left(1+\frac{1}{\lambda_v}\right)\bigg(|\de x|^2+\e\bigg[ \sup_{t\le s\le T}\left| Y_{t\xi'}(s)-Y_{t\xi}(s) \right|^2\bigg]\bigg).
\end{align*}
Using \eqref{lem5_3}, we further have
\begin{align*}
	\e\left[\int_t^T |\de v(s)|^2 ds\right]\le C(L,T,\lambda_v)\e\left[|\de x|^2+|\de\xi|^2\right].
\end{align*}
Since $\xi$ and $\xi'$ are arbitrary, we obtain \eqref{lem5_2} for any $\lambda_v>0$.

\section{Proof of Statements in Section~\ref{sec:distribution}}\label{pf:distribution}

\subsection{Proof of Lemma~\ref{lem:3}}\label{pf_lem3}

We now apply Lemma~\ref{lem:1} to establish the solvability through the use of  \eqref{chose_coefficients}. We first compute the monotonicity condition for FBSDEs \eqref{FB:dr}. For $s\in[t,T]$, $X,X',p,p' \in L^2(\Omega,\f,\mathbb{P};\brn)$ and $q,q'\in (L^2(\Omega,\f,\mathbb{P};\brn))^n$, for simplicity, we denote by $\mathbf{V}':=	\mathbf{V}(s,X',p',q')$ and $\mathbf{V}:=	\mathbf{V}(s,X,p,q)$. From equations of \eqref{optimal_condition}, we deduce that
\begin{align*}
	&\e\bigg[(X'-X)^\top \left(\mathbf{F}(s,X',p',q')-\mathbf{F}(X,s;p,q)\right) +(p'-p)^\top \left(\mathbf{B}(s,X',p',q')-\mathbf{B}(s,X,p,q)\right) \\
	&\quad +\sum_{j=1}^n \left({q'}^{j}-q^j\right)^\top  \left(\mathbf{A}^j(s,X',p',q')-\mathbf{A}^j(s,X,p,q)\right) \bigg]\\
	=\ & \e\Bigg\{(p'-p)^\top  \widetilde{\e}\left[\left[D \frac{db_0}{d\nu} (s,\lr(Y_{t\xi}(s)))\left(\widetilde{Y_{t\xi}}(s)\right) \right]^\top \left(\widetilde{X'}-\widetilde{X}\right)\right] \\
	&\quad +\sum_{j=1}^n \left({q'}^{j}-q^j\right)^\top  \widetilde{\e}\left[\left[D \frac{d\sigma^j_0}{d\nu}  \left(s,\lr(Y_{t\xi}(s))\right)\left(\widetilde{Y_{t\xi}}(s)\right) \right]^\top \left(\widetilde{X'}-\tx\right)\right]\\
	&\quad  - \left[\left(
	\begin{array}{cc}
		D_v^2 f & D_vD_x f\\
		D_xD_v f & D_x^2 f
	\end{array}
	\right)(s,\theta_{t\xi}(s)) \right]\left(
	\begin{array}{cc}
		\mathbf{V}'-\mathbf{V}\\
		X'-X
	\end{array}
	\right)^{\otimes 2} \\
	&\quad -\left(X'-X\right)^\top  \widetilde{\e}\left[\left[\left(D_y \frac{d}{d\nu}D_x f\right) (s,\theta_{t\xi}(s))\left(\widetilde{Y_{t\xi}}(s)\right) \right]^\top  \left(\widetilde{X'}-\tx\right)\right]\\
	&\quad -\left(\mathbf{V}'-\mathbf{V}\right)^\top  \widetilde{\e}\left[\left[\left(D_y \frac{d}{d\nu}D_v f\right) (s,\theta_{t\xi}(s))\left(\widetilde{Y_{t\xi}}(s)\right) \right]^\top \left(\widetilde{X'}-\tx\right)\right]\Bigg\}.
\end{align*}
Then, from \eqref{convex}, Assumptions (A1) and (A2) and the Young's inequality with a weight of $\frac{\lambda_v}{L}$, we have	
\begin{equation}\label{lem3_5}
	\begin{split}
		&\e\bigg[(X'-X)^\top \left(\mathbf{F}(s,X',p',q')-\mathbf{F}(X,s;p,q)\right) +(p'-p)^\top \left(\mathbf{B}(s,X',p',q')-\mathbf{B}(s,X,p,q)\right) \\
		&\quad +\sum_{j=1}^n \left({q'}^{j}-q^j\right)^\top  \left(\mathbf{A}^j(s,X',p',q')-\mathbf{A}^j(s,X,p,q)\right) \bigg]\\
		\le\ & -\lambda_v \e\left[\left|\mathbf{V}'-\mathbf{V}\right|^2\right]+C(L)\left(1+\frac{1}{\lambda_v}\right)\e\left[|X'-X|^2+|p'-p|^2+|q'-q|^2\right]\\
		& +L\e\left[|X'-X|^2\right]^{\frac{1}{2}} \e\left[\left|\mathbf{V}'-\mathbf{V}\right|^2\right]^{\frac{1}{2}}\\
		\le\ & -\frac{\lambda_v}{2} \e\left[\left|\mathbf{V}'-\mathbf{V}\right|^2\right]+C(L)\left(1+\frac{1}{\lambda_v}\right)\e\left[|X'-X|^2+|p'-p|^2+|q'-q|^2\right].
	\end{split}
\end{equation}
Therefore, Condition \eqref{Condition_generic} (i) is satisfied with the map $\beta(s,X,p,q):=\mathbf{V}(s,X,p,q)$ by setting $\Lambda=\frac{\lambda_v}{2}$ and $\alpha=C(L)\left(1+\frac{1}{\lambda_v}\right)$. Conditions  \eqref{Condition_generic} (ii) is also satisfied with $ K =C(L)$. Next, we aim to show \eqref{lem3_1}. With standard arguments from the context of SDEs, we have
\begin{equation}\label{lem3_2}
	\e\bigg[\sup_{t\le s\le T}|\dr_\eta Y_{t\xi}(s)|^2\bigg]\le C(L,T)\e\left[|\eta|^2+\int_t^T |\dr_\eta v_{t\xi}(s)|^2ds\right].
\end{equation}
From \eqref{lem3_2}, usual BSDE arguments also give
\begin{equation}\label{lem3_3}
	\e\bigg[\sup_{t\le s\le T}|\dr_\eta p_{t\xi}(s)|^2+\int_t^T |\dr_\eta q_{t\xi}(s)|^2 ds\bigg]\le C(L,T)\e\left[|\eta|^2+\int_t^T |\dr_\eta v_{t\xi}(s)|^2ds\right].
\end{equation}    
From It\^o's formula, similar as \eqref{lem3_5}, we have
\begin{equation}\label{lem3_4}
	\begin{split}
		&\e\left[\left(\dr_\eta p_{t\xi}(T)\right)^\top \dr_\eta Y_{t\xi}(T)-\left(\dr_\eta p_{t\xi}(t)\right)^\top  \eta\right]\\
		\le\ & -\frac{\lambda_v}{2} \e\left[\int_t^T |\dr_\eta v_{t\xi}(s)|^2 ds\right]+C(L)\left(1+\frac{1}{\lambda_v}\right)\e\left[\int_t^T |\dr_\eta Y_{t\xi}(s)|^2+|\dr_\eta p_{t\xi}(s)|^2+ |\dr_\eta q_{t\xi}(s)|^2 ds\right].
	\end{split}
\end{equation}
Substituting \eqref{lem3_2} and \eqref{lem3_3} into \eqref{lem3_4}, we know that there exists a constant $c(L,T)$ depending only on $(L,T)$, such that when $\lambda_v\geq c(L,T)$, we obtain \eqref{lem3_1}.

If Assumption (A3') is satisfied, then, from \eqref{convex'} and the Young's inequality with a weight of $\frac{\lambda_v}{L}$, we have
\begin{align}
	&\e\bigg[(X'-X)^\top \left(\mathbf{F}(s,X',p',q')-\mathbf{F}(X,s;p,q)\right) +(p'-p)^\top \left(\mathbf{B}(s,X',p',q')-\mathbf{B}(s,X,p,q)\right) \notag \\
	&\quad +\sum_{j=1}^n \left({q'}^{j}-q^j\right)^\top  \left(\mathbf{A}^j(s,X',p',q')-\mathbf{A}^j(s,X,p,q)\right) \bigg] \notag \\
	=\ & \e\Bigg\{- \left[\left(
	\begin{array}{cc}
		D_v^2 f & D_vD_x f\\
		D_xD_v f & D_x^2 f
	\end{array}
	\right)(s,\theta_{t\xi}(s)) \right] \left(
	\begin{array}{cc}
		\mathbf{V}'-\mathbf{V}\\
		X'-X
	\end{array}
	\right)^{\otimes 2} \notag \\
	&\quad -\left(X'-X\right)^\top  \widetilde{\e}\left[\left[\left(D_y \frac{d}{d\nu}D_x f \right) (s,\theta_{t\xi}(s))\left(\widetilde{Y_{t\xi}}(s)\right) \right]^\top \left(\widetilde{X'}-\tx\right)\right] \notag \\
	&\quad -\left(\mathbf{V}'-\mathbf{V}\right)^\top  \widetilde{\e}\left[\left[\left(D_y \frac{d}{d\nu}D_v f\right) (s,\theta_{t\xi}(s))\left(\widetilde{Y_{t\xi}}(s)\right) \right]^\top \left(\widetilde{X'}-\tx\right)\right]\Bigg\} \notag \\
	\le \ & \e\left[-2\lambda_v  \left|\mathbf{V}'-\mathbf{V}\right|^2-2\lambda_x \left|X'-X\right|^2 +L_x |X'-X|^2\right]+L_v \e\left[|X'-X|^2\right]^{\frac{1}{2}} \e\left[\left|\mathbf{V}'-\mathbf{V}\right|^2\right]^{\frac{1}{2}} \notag \\
	\le \ & -\lambda_v\e\left[  \left|\mathbf{V}'-\mathbf{V}\right|^2 \right]-\left(2\lambda_x-L_x-\frac{L_v^2}{4\lambda_v}\right) \e\left[|X'-X|^2\right] \notag \\
	\le \ & -\lambda_v\e\left[  \left|\mathbf{V}'-\mathbf{V}\right|^2 \right]. \label{lem3_6}
\end{align}
In the last inequality, we use the small mean field effect condition in Assumption (A3'): $\lambda_x\geq\frac{L_v^2}{8\lambda_v}+\frac{L_x}{2}$. (Here, $\frac{L_v^2}{8\lambda_v}+\frac{L_x}{2}$ is not the optimal parameter; actually, we only need  $\lambda_x>\frac{L_v^2}{16\lambda_v}+\frac{L_x}{2}$). Therefore, Condition \eqref{Condition_generic} (i) is valid for $\alpha=0$. Similarly, 
\begin{align*}
	&\e\left[(X'-X)^\top  \left(\mathbf{G}(X')-\mathbf{G}(X)\right)\right]\\
	=\ &\e\Big[ (X'-X)^\top  D_x^2 g (Y_{t\xi}(T),\lr(Y_{t\xi}(T)))(X'-X)\\
	&\quad +(X'-X)^\top \widetilde{\e}\left[\left[\left(D_y \frac{d}{d\nu}D_x g\right) (Y_{t\xi}(T),\lr(Y_{t\xi}(T)))\left(\widetilde{Y_{t\xi}}(T)\right) \right]^\top \left(\widetilde{X'}-\tx\right) \right]\Big]\\
	\geq\ & \left(2 \lambda_g-L_g\right)  \e\left[\left| X'-X \right|^2\right] \geq 0,
\end{align*} 
therefore,  Condition \eqref{Condition_generic} (iii) is satisfied. As a consequence of Lemma~\ref{lem:1}, we obtain the solvability for FBSDEs \eqref{FB:dr}. Then, we proceed to establish \eqref{lem3_1} under Assumption (A3'). From It\^o's formula and \eqref{lem3_6}, we have
\begin{align*}
		&\e\left[\left(\dr_\eta p_{t\xi}(T)\right)^\top \dr_\eta Y_{t\xi}(T)-\left(\dr_\eta p_{t\xi}(t)\right)^\top  \eta\right]\le -\lambda_v \e\left[\int_t^T |\dr_\eta v_{t\xi}(s)|^2 ds\right].
\end{align*}
From \eqref{convex'} and Assumption (A2), for any $\epsilon>0$, we have
\begin{align*}
		&\e\left[\left(\dr_\eta p_{t\xi}(T)\right)^\top \dr_\eta Y_{t\xi}(T)- \left(\dr_\eta p_{t\xi}(t)\right)^\top  \eta\right]\\
		=\ &\e\bigg[-\left(\dr_\eta p_{t\xi}(t)\right)^\top  \eta+ \left(\dr_\eta Y_{t\xi}(T)\right)^\top  D_x^2 g (Y_{t\xi}(T),\lr(Y_{t\xi}(T)))\dr_\eta Y_{t\xi}(T)\\
		&\quad + \left(\dr_\eta Y_{t\xi}(T)\right)^\top  \widetilde{\e}\left[\left[\left(D_y \frac{d}{d\nu}D_x g \right) (Y_{t\xi}(T),\lr(Y_{t\xi}(T)))\left(\widetilde{Y_{t\xi}}(T)\right) \right]^\top \widetilde{\dr_\eta Y_{t\xi}}(T)\right]\bigg]\\
		\geq\ & \left(2 \lambda_g-L_g\right)\e\left[\left|\dr_\eta Y_{t\xi}(T)\right|^2 \right]-\e\left[ \left(\dr_\eta p_{t\xi}(t)\right)^\top  \eta \right]\\
		\geq\ & -\epsilon \e\left[ \left|\dr_\eta p_{t\xi}(t)\right|^2\right]-\frac{1}{4\epsilon}\e\left[|\eta|^2\right].
\end{align*}
Therefore, we have
\begin{equation*}
	\begin{split}
		& \lambda_v \e\left[\int_t^T |\dr_\eta v_{t\xi}(s)|^2 ds\right]\le \epsilon \e\left[ \left|\dr_\eta p_{t\xi}(t)\right|^2\right]+\frac{1}{4\epsilon}\e\left[|\eta|^2\right].
	\end{split}
\end{equation*}
Substituting \eqref{lem3_3} into the last inequality, we have
\begin{equation*}
	\begin{split}
		& \lambda_v \e\left[\int_t^T |\dr_\eta v_{t\xi}(s)|^2 ds\right]\le \epsilon C(L,T) \e\left[\int_t^T |\dr_\eta v_{t\xi}(s)|^2ds\right]+C(L,T)\left(1+\frac{1}{\epsilon}\right)\e\left[|\eta|^2\right].
	\end{split}
\end{equation*}
By choosing $\epsilon:=\frac{\lambda_v}{2 C(L,T)}$, we have
\begin{equation*}
	\begin{split}
		& \frac{\lambda_v}{2} \e\left[\int_t^T |\dr_\eta v_{t\xi}(s)|^2 ds\right]\le C(L,T)\left(1+\frac{1}{\lambda_v}\right)\e\left[|\eta|^2\right].
	\end{split}
\end{equation*}
Substituting the last inequality back into \eqref{lem3_2} and \eqref{lem3_3}, we obtain \eqref{lem3_1}, which is valid for all $\lambda_v>0$.

\subsection{Proof of Theorem~\ref{lem:4}}\label{pf_lem4}

We first establish \eqref{lem4_1}. For $\epsilon\in(0,1)$ and $s\in[t,T]$, we denote by $\de^\epsilon Y(s):=\frac{1}{\epsilon}\left[Y_{t\xi^\epsilon}(s)-Y_{t\xi}(s)\right]$ and $\delta^\epsilon Y(s):=\de Y^\epsilon(s)-\dr_\eta Y_{t\xi}(s)$. Similar definitions apply to $\de^\epsilon p(s),\de^\epsilon q(s)$ and $\delta^\epsilon p(s),\delta^\epsilon q(s)$. For the sake of convenience, denote by $\dr_\eta\Theta_{t\xi}(s):=(\dr_\eta Y_{t\xi}(s),\dr_\eta p_{t\xi}(s),\dr_\eta q_{t\xi}(s))$, $\de^\epsilon\Theta(s):=(\de^\epsilon Y(s),\de^\epsilon p(s),\de^\epsilon q(s))$ and $\delta^\epsilon\Theta(s):=(\delta^\epsilon Y(s),\delta^\epsilon p(s),\delta^\epsilon q(s))$. From \eqref{lem2_2},  we see that 
\begin{equation}\label{lem4_2}
	\begin{split}
		&\e\bigg[ \sup_{t\le s\le T}\left|\left(\de^\epsilon Y(s),\de^\epsilon p(s)\right)^\top\right|^2 + \int_t^T \left|\de^\epsilon q(s)\right|^2 ds \bigg]\le C(L,T,\lambda_v)\e[|\eta|^2].
	\end{split}
\end{equation}	
Then, from \eqref{lem3_1} we know that
\begin{equation}\label{lem4_2'}
	\begin{split}
		&\e\left[\sup_{t\le s\le T}\left|\left(\dr_\eta Y_{t\xi}(s),\dr_\eta p_{t\xi}(s),\delta^\epsilon Y(s),\delta^\epsilon p(s)\right)^\top\right|^2+ \int_t^T\left|\left(\dr_\eta q_{t\xi}(s),\delta^\epsilon q(s)\right)^\top \right|^2 ds\right]\\
		\le\ & C(L,T,\lambda_v)\e\left[|\eta|^2\right].
	\end{split}
\end{equation}
For $Y,p\in L^2(\Omega,\f,\mathbb{P};\brn)$ and $q\in(L^2(\Omega,\f,\mathbb{P};\brn))^n$, we use the notation $\Theta:=(Y,\lr(Y),p,q)$ and define
\begin{align*}
	\mathbf{V}(s,\Theta,\Theta'):=\ &\left[D_x\hv  (s,\Theta)\right]^\top Y'+\widetilde{\e}\left[\left[D_y\frac{d\hv}{d\nu}  (s,\Theta)\left(\ty\right) \right]^\top \widetilde{Y'}\right]\\
	&+\left[D_p\hv (s,\Theta) \right]^\top p'+\sum_{j=1}^n \left[D_{q^j}\hv
	  (s,\Theta) \right]^\top {q'}^j,
\end{align*}
where $\widetilde{Y'}$ is an independent copy of $Y'$. Then, the process $\delta^\epsilon\Theta$ satisfies the following FBSDEs: for $s\in[t,T]$,
\small
\begin{align*}
		\delta^\epsilon Y(s)=\ &\int_t^s \bigg\{  \widetilde{\e}\bigg[\int_0^1 \bigg[\left[D \frac{db_0}{d\nu} \left(r,\lr\left(Y_{t\xi}^{h,\epsilon}(r)\right)\right)\left(\widetilde{Y_{t\xi}^{h,\epsilon}}(r)\right) \right]^\top \widetilde{\de^\epsilon Y}(s)\\
		&\ \qquad\qquad\qquad -\left[D \frac{db_0}{d\nu} (r,\lr(Y_{t\xi}(r)))\left(\widetilde{Y_{t\xi}}(r)\right) \right]^\top \widetilde{\dr_\eta Y_{t\xi}}(s)\bigg]dh \bigg]+b_1(r)\;\delta^\epsilon Y(r)\\
		&\qquad +b_2(s)\int_0^1 \delta^{h,\epsilon}\mathbf{V}(r) dh \bigg\}dr\\
		&+\int_t^s \bigg\{ \widetilde{\e}\bigg[ \int_0^1 \bigg[\left[D \frac{d\sigma_0}{d\nu} \left(r,\lr\left(Y_{t\xi}^{h,\epsilon}(r)\right)\right)\left(\widetilde{Y_{t\xi}^{h,\epsilon}}(r)\right) \right]^\top \widetilde{\de^\epsilon Y}(s)\\
		&\ \quad\qquad\qquad\qquad -\left[D \frac{d\sigma_0}{d\nu} (r,\lr(Y_{t\xi}(r)))\left(\widetilde{Y_{t\xi}}(r)\right) \right]^\top \widetilde{\dr_\eta Y_{t\xi}}(r)\bigg]dh \bigg]+\sigma_1(r)\; \delta^\epsilon Y(r)\\
		&\quad\qquad +\sigma_2(s)\int_0^1 \delta^{h,\epsilon}\mathbf{V}(r) dh \bigg\} dB(r),\\
		\delta^\epsilon p(s)=\ &\int_0^1 \left[ \left[D_x^2 g \left(Y_{t\xi}^{h,\epsilon}(T),\lr\left(Y_{t\xi}^{h,\epsilon}(T)\right)\right) \right]^\top \de^\epsilon Y(T) - \left[D_x^2 g  (Y_{t\xi}(T),\lr(Y_{t\xi}(T))) \right]^\top \dr_\eta Y_{t\xi}(T)\right]dh\\
		&+\widetilde{\e}\bigg[\int_0^1 \bigg[\left[\left(D_y \frac{d}{d\nu}D_x g \right) \left(Y_{t\xi}^{h,\epsilon}(T),\lr\left(Y_{t\xi}^{h,\epsilon}(T)\right)\right)\left(\widetilde{Y_{t\xi}^{h,\epsilon}}(T)\right) \right]^\top \widetilde{\de^\epsilon Y}(T)\\
		&\quad\qquad\qquad -\left[\left(D_y \frac{d}{d\nu}D_x g \right) (Y_{t\xi}(T),\lr(Y_{t\xi}(T)))\left(\widetilde{Y_{t\xi}}(T)\right) \right]^\top \widetilde{\dr_\eta Y_{t\xi}}(T)\bigg]dh\bigg]\\
		& + \int_s^T \bigg\{ b _1(r)^\top \delta^\epsilon p(r)+\sum_{j=1}^n\left(\sigma_1^j (r)\right)^\top \delta^\epsilon q^j(r)\\
		&\quad\qquad +\int_0^1 \left[\left[D_x^2 f  \left(r,\theta_{t\xi}^{h,\epsilon}(r)\right) \right]^\top \de^\epsilon Y(r)- \left[D_x^2 f  (r,\theta_{t\xi}(r)) \right]^\top \dr_\eta Y_{t\xi}(s)\right]dh\\
		&\quad\qquad +\widetilde{\e}\bigg[\int_0^1 \bigg[ \left[\left(D_y \frac{d}{d\nu}D_x f \right) \left(r,\theta_{t\xi}^{h,\epsilon}(r)\right) \left(\widetilde{Y_{t\xi}^{h,\epsilon}}(r) \right) \right]^\top \widetilde{\de^\epsilon Y}(s)\\
		&\qquad\qquad\qquad\qquad -\left[\left(D_y \frac{d}{d\nu}D_x f \right) (r,\theta_{t\xi}(r))\left(\widetilde{Y_{t\xi}}(r)\right) \right]^\top \widetilde{\dr_\eta Y_{t\xi}}(r)\bigg]dh\bigg]\\
		&\quad\qquad +\int_0^1 \Big[\left[D_vD_x f \left(r,\theta_{t\xi}^{h,\epsilon}(r)\right) \right]^\top \mathbf{V}\left(r,\Theta_{t\xi}^{h,\epsilon}(r),\de^\epsilon\Theta(r)\right)\\
		&\qquad\qquad\qquad  -\left[D_vD_x f  (r,\theta_{t\xi}(r))\right]^\top\mathbf{V}(r,\Theta_{t\xi}(r),\dr_\eta\Theta_{t\xi}(r))\Big]dh \bigg\}dr-\int_s^T \delta^\epsilon q(r)\; dB(r),
\end{align*}    
\normalsize
where $Y_{t\xi}^{h,\epsilon}(r):=Y_{t\xi}(r)+h\epsilon\de^\epsilon Y(r)$, $\theta_{t\xi}^{h,\epsilon}(r):=\theta_{t\xi}(r)+h\epsilon\de^\epsilon\theta(r)$, $\Theta_{t\xi}^{h,\epsilon}(r):=\Theta_{t\xi}(r)+h\epsilon\de^\epsilon\Theta(r)$ and  $\delta^{h,\epsilon}\mathbf{V}(r):=\mathbf{V}\left(r,\Theta_{t\xi}^{h,\epsilon}(r),\de^\epsilon\Theta(r)\right)-\mathbf{V}(r,\Theta_{t\xi}(r),\dr_\eta\Theta_{t\xi}(r))$ for $(r,h)\in[t,T]\times[0,1]$; and $\left(\widetilde{Y_{t\xi}^{h,\epsilon}}(r),\widetilde{\de^\epsilon Y}(r),\widetilde{\dr_\eta Y_{t\xi}}(r)\right)$ is an independent copy of $\left(Y_{t\xi}^{h,\epsilon}(r),\de^\epsilon Y(r),\dr_\eta Y_{t\xi}(r)\right)$. From Conditions \eqref{optimal_condition} and It\^o's formula, we have
\begin{align}
	&\e\left[\left(\delta^\epsilon p(T)\right)^\top \delta^\epsilon Y(T)\right] \notag \\
	=\ & \int_t^T\int_0^1 \e\widetilde{\e} \Bigg\{ - \left[\left(
	\begin{array}{cc}
		D_v^2 f & D_vD_x f\\
		D_xD_v f & D_x^2 f
	\end{array}
	\right)(s,\theta_{t\xi}(s)) \right] \left(
	\begin{array}{cc}
		\delta^{h,\epsilon}\mathbf{V}(s)\\
		\delta^\epsilon Y(s)
	\end{array}
	\right)^{\otimes 2}  \notag \\
	&\quad\qquad\qquad +\left(\delta^\epsilon p(s)\right)^\top   \left[D \frac{db_0}{d\nu} \left(s,\lr\left(Y_{t\xi}(s)\right)\right)\left(\widetilde{Y_{t\xi}}(s)\right) \right]^\top \widetilde{\delta^\epsilon Y}(s) \notag \\
	&\quad\qquad\qquad +\sum_{j=1}^n \left(\delta^\epsilon q^j(s)\right)^\top  \left[D \frac{d\sigma^j_0}{d\nu} \left(s,\lr\left(Y_{t\xi}(s)\right)\right)\left(\widetilde{Y_{t\xi}}(s)\right) \right]^\top \widetilde{\delta^\epsilon Y}(s) \notag \\
	&\quad\qquad\qquad -\left(\delta^{h,\epsilon}\mathbf{V}(s)\right)^\top  \left[\left(D_y\frac{d}{d\nu}D_v f \right)  (s,\theta_{t\xi}(s))\left(\widetilde{Y_{t\xi}}(s)\right) \right]^\top \widetilde{\delta^\epsilon Y}(s)   \notag \\
	&\quad\qquad\qquad -\left(\delta^\epsilon Y(s)\right)^\top  \left[\left(D_y \frac{d}{d\nu}D_x f \right) \left(s,\theta_{t\xi}(s)\right)\left(\widetilde{Y_{t\xi}}(s)\right)\right]^\top  \widetilde{\delta^\epsilon Y}(s) \notag \\
	&\quad\qquad\qquad + \left(\delta^\epsilon p(s)\right)^\top  \left[D \frac{db_0}{d\nu} \left|_{(s,\lr(Y_{t\xi}(s)))\left(\widetilde{Y_{t\xi}}(s)\right)}^{\left(s,\lr\left(Y_{t\xi}^{h,\epsilon}(s)\right)\right)\left(\widetilde{Y_{t\xi}^{h,\epsilon}}(s)\right)} \right.  \right]^\top \widetilde{\de^\epsilon Y}(s) \notag \\
	&\quad\qquad\qquad +\sum_{j=1}^n \left(\delta^\epsilon q^j(s)\right)^\top  \left[D \frac{d\sigma^j_0}{d\nu} \left|_{\left(s,\lr\left(Y_{t\xi}(s)\right)\right)\left(\widetilde{Y_{t\xi}}(s)\right)  }^{\left(s,\lr\left(Y_{t\xi}^{h,\epsilon}(s)\right)\right)\left(\widetilde{Y_{t\xi}^{h,\epsilon}}(s)\right) } \right. \right]^\top \widetilde{\de^\epsilon Y}(s) \notag \\
	&\quad\qquad\qquad  -\left(\delta^{h,\epsilon}\mathbf{V}(s)\right)^\top  \left[D_v^2f (s,\theta_{t\xi}(s))\right]^\top \left[\mathbf{V}\left(s,\Theta_{t\xi}(s),\de^\epsilon\Theta(s)\right)-\mathbf{V}\left(s,\Theta_{t\xi}^{h,\epsilon}(s),\de^\epsilon\Theta_{t\xi}(s)\right)\right] \notag \\
	&\quad\qquad\qquad  -\left(\delta^\epsilon Y(s)\right)^\top   \left[D_vD_x f \left|_{\left(s,\theta_{t\xi}(s)\right)}^{ \left(s,\theta_{t\xi}^{h,\epsilon}(s)\right)} \right.  \right]^\top \mathbf{V}\left(s,\Theta_{t\xi}^{h,\epsilon}(s),\de^\epsilon\Theta(s)\right)  \notag \\
	&\quad\qquad\qquad  -\left(\delta^\epsilon Y(s)\right)^\top  \left[D_x^2 f  \left|_{\left(s,\theta_{t\xi}(s)\right)}^{ \left(s,\theta_{t\xi}^{h,\epsilon}(s)\right)} \right. \right]^\top  \de^\epsilon Y(s) \notag \\
	&\quad\qquad\qquad -\left(\delta^\epsilon Y(s)\right)^\top  \left[\left(D_y \frac{d}{d\nu}D_x f \right) \left|_{(s,\theta_{t\xi}(s))\left(\widetilde{Y_{t\xi}}(s)\right) }^{\left(s,\theta_{t\xi}^{h,\epsilon}(s)\right)\left(\widetilde{Y_{t\xi}^{h,\epsilon}}(s)\right) } \right.\right]^\top \widetilde{\de^\epsilon Y}(s)  \Bigg\}dh ds, \label{lem4_3} 
\end{align}
where $\widetilde{\delta^\epsilon Y}(s)$ is an independent copy of $\delta^\epsilon Y(s)$. From \eqref{convex}, we know that 
\begin{align*}
	&\e\left[\left(\delta^\epsilon p(T)\right)^\top \delta^\epsilon Y(T)\right] \notag \\
	\le\ & -\lambda_v \e\widetilde{\e} \left[\int_t^T\int_0^1 \left| \delta^{h,\epsilon}\mathbf{V}(s) \right|^2 dhds\right] +I_1(\epsilon)\\
	&+ C(L)\left(1+\frac{1}{\lambda_v}\right) \e\widetilde{\e} \int_t^T\int_0^1 \left[\left| \delta^{\epsilon}Y(s) \right|^2+L \left| \delta^{\epsilon}Y(s) \right| \left(\left| \delta^{\epsilon}p(s) \right|+\left| \delta^{\epsilon}q(s) \right|+\left| \delta^{h,\epsilon} \mathbf{V}(s) \right| \right) \right] dhds,
\end{align*}
where
\small
\begin{align*}
	I_1(\epsilon):= \ & \e\widetilde{\e} \int_t^T\int_0^1  \bigg\{  L\left|\delta^{h,\epsilon}\mathbf{V}(s)\right| \left|\mathbf{V}\left(s,\Theta_{t\xi}(s),\de^\epsilon\Theta(s)\right)-\mathbf{V}\left(s,\Theta_{t\xi}^{h,\epsilon}(s),\de^\epsilon\Theta_{t\xi}(s)\right)\right| \\
	& +\left|\delta^\epsilon p(s)\right| \left| \widetilde{\de^\epsilon Y}(s)\right|  \left|D \frac{db_0}{d\nu}\left(s,\lr\left(Y_{t\xi}^{h,\epsilon}(s)\right)\right)\left(\widetilde{Y_{t\xi}^{h,\epsilon}}(s)\right) - D \frac{db_0}{d\nu}\left( s,\lr(Y_{t\xi}(s))\right)\left(\widetilde{Y_{t\xi}}(s)\right) \right|  \\
	& +\left|\delta^\epsilon q(s)\right| \left| \widetilde{\de^\epsilon Y}(s)\right|  \left|D \frac{d\sigma_0}{d\nu}\left(s,\lr\left(Y_{t\xi}^{h,\epsilon}(s)\right)\right)\left(\widetilde{Y_{t\xi}^{h,\epsilon}}(s)\right) - D \frac{d\sigma_0}{d\nu}\left( s,\lr(Y_{t\xi}(s))\right)\left(\widetilde{Y_{t\xi}}(s)\right) \right| \\
	& +\left|\delta^\epsilon Y(s)\right|  \left|\mathbf{V}\left(s,\Theta_{t\xi}^{h,\epsilon}(s),\de^\epsilon\Theta(s)\right)\right|  \left|D_vD_x f\left(s,\theta_{t\xi}^{h,\epsilon}(s) \right)-D_vD_x f\left(s,\theta_{t\xi}(s)\right) \right| \\
	& +\left|\delta^\epsilon Y(s)\right|  \left|\widetilde{\de^\epsilon Y}(s)\right| \left|\left(D_y \frac{d}{d\nu}D_x f\right)\left(s,\theta_{t\xi}^{h,\epsilon}(s)\right)\left(\widetilde{Y_{t\xi}^{h,\epsilon}}(s)\right) -\left(D_y \frac{d}{d\nu}D_x f\right) (s,\theta_{t\xi}(s))\left(\widetilde{Y_{t\xi}}(s)\right) \right|  \\
	& +\left|\delta^\epsilon Y(s)\right|  \left| \de^\epsilon Y(s)\right|  \left|D_x^2 f\left(s,\theta_{t\xi}^{h,\epsilon}(s)\right)-D_x^2 f\left(s,\theta_{t\xi}(s)\right)  \right| \bigg\} dh ds.
\end{align*}
\normalsize
Similarly, 
\begin{align}
	& \e\left[\left(\delta^\epsilon p(T)\right)^\top  \delta^\epsilon Y(T)\right]\notag \\
	=\  & \int_0^1 \e\bigg\{ \left(\delta^\epsilon Y(T)\right)^\top  \left[D_x^2 g \left(Y_{t\xi}(T),\lr\left(Y_{t\xi}(T)\right)\right) \right]^\top \delta^\epsilon Y(T) \notag\\
	&\quad\qquad +\left(\delta^\epsilon Y(T)\right)^\top \left[D_x^2 g  \left(Y_{t\xi}^{h,\epsilon}(T),\lr\left(Y_{t\xi}^{h,\epsilon}(T)\right)\right)- D_x^2 g  (Y_{t\xi}(T),\lr(Y_{t\xi}(T))) \right]^\top \de^\epsilon Y(T) \notag\\
	&\quad\qquad +\left(\delta^\epsilon Y(T)\right)^\top  \widetilde{\e}\bigg[\left[\left(D_y \frac{d}{d\nu}D_x g \right) \left(Y_{t\xi}(T),\lr\left(Y_{t\xi}(T)\right)\right)\left(\widetilde{Y_{t\xi}}(T)\right) \right]^\top \widetilde{\delta^\epsilon Y}(T) \notag\\
	&\quad\qquad\qquad\qquad\qquad +\bigg[\left(D_y \frac{d}{d\nu}D_x g \right) \left(Y_{t\xi}^{h,\epsilon}(T),\lr\left(Y_{t\xi}^{h,\epsilon}(T)\right)\right)\left(\widetilde{Y_{t\xi}^{h,\epsilon}}(T)\right)  \notag\\
	&\ \qquad\qquad\qquad\qquad\qquad -\left(D_y \frac{d}{d\nu}D_x g\right)  (Y_{t\xi}(T),\lr(Y_{t\xi}(T)))\left(\widetilde{Y_{t\xi}}(T)\right) \bigg]^\top \widetilde{\de^\epsilon Y}(T)\bigg]\bigg\}dh \notag\\
	\geq\ & -2L \e\left[\left|\delta^\epsilon Y(T)\right|^2\right]-I_2(\epsilon), \label{lem4_3'}
\end{align}
where
\small
\begin{align*}
	I_2(\epsilon):=\ &\e\bigg[\left|\delta^\epsilon Y(T)\right| \left|\widetilde{\de^\epsilon Y}(T)\right| \bigg(  \left|D_x^2 g \left(Y_{t\xi}^{h,\epsilon}(T),\lr\left(Y_{t\xi}^{h,\epsilon}(T)\right)\right) - D_x^2 g (Y_{t\xi}(T),\lr(Y_{t\xi}(T))) \right|  \\
	&+ \left| D_y \frac{d}{d\nu}D_x g \left(Y_{t\xi}^{h,\epsilon}(T),\lr\left(Y_{t\xi}^{h,\epsilon}(T)\right)\right)\left(\widetilde{Y_{t\xi}^{h,\epsilon}}(T)\right) - D_y \frac{d}{d\nu}D_x g (Y_{t\xi}(T),\lr(Y_{t\xi}(T)))\left(\widetilde{Y_{t\xi}}(T)\right)\right| \bigg) \bigg].
\end{align*}
\normalsize
Therefore, we have
\begin{align}
		&\lambda_v \e\widetilde{\e}\left[ \int_t^T\int_0^1 \left| \delta^{h,\epsilon}\mathbf{V}(s) \right|^2 dhds \right]  \notag \\
		\le\ & C(L)\left(1+\frac{1}{\lambda_v}\right) \e\widetilde{\e} \int_t^T\int_0^1 \left[\left| \delta^{\epsilon}Y(s) \right|^2+L \left| \delta^{\epsilon}Y(s) \right| \left(\left| \delta^{\epsilon}p(s) \right|+\left| \delta^{\epsilon}q(s) \right|+\left| \delta^{h,\epsilon} \mathbf{V}(s) \right| \right) \right] dhds \notag \\
		&+ I_1(\epsilon)+I_2(\epsilon) +2L \e\left[\left|\delta^\epsilon Y(T)\right|^2\right], \label{lem4_9}
\end{align}
With the usual standard arguments for SDEs and BSDEs, we have
\begin{equation}\label{lem4_4}
	\begin{split}
		&\e\bigg[\sup_{t\le s\le T}\left|\left(\delta^\epsilon Y(s),\delta^\epsilon p(s)\right)^\top\right|^2 + \int_t^T \left|\delta^\epsilon q(s) \right|^2ds \bigg]\\
		\le\ & C(L,T)\e\widetilde{\e} \int_t^T\int_0^1 \left| \delta^{h,\epsilon}\mathbf{V}(s) \right|^2 dhds +C(L,T)I_3(\epsilon),
	\end{split}    	
\end{equation}    
where
\small
\begin{align*}
	I_3(\epsilon):=\e\widetilde{\e} \int_t^T\int_0^1  & \bigg\{ \left| \widetilde{\de^\epsilon Y}(s)\right|^2  \left|D \frac{db_0}{d\nu}\left(s,\lr\left(Y_{t\xi}^{h,\epsilon}(s)\right)\right)\left(\widetilde{Y_{t\xi}^{h,\epsilon}}(s)\right) - D \frac{db_0}{d\nu}\left( s,\lr(Y_{t\xi}(s))\right)\left(\widetilde{Y_{t\xi}}(s)\right) \right| ^2 \\
	& + \left| \widetilde{\de^\epsilon Y}(s)\right|^2  \left|D \frac{d\sigma_0}{d\nu}\left(s,\lr\left(Y_{t\xi}^{h,\epsilon}(s)\right)\right)\left(\widetilde{Y_{t\xi}^{h,\epsilon}}(s)\right) - D \frac{d\sigma_0}{d\nu}\left( s,\lr(Y_{t\xi}(s))\right)\left(\widetilde{Y_{t\xi}}(s)\right) \right|^2 \\
	& +  \left|\mathbf{V}\left(s,\Theta_{t\xi}^{h,\epsilon}(s),\de^\epsilon\Theta(s)\right)\right|^2  \left|D_vD_x f\left(s,\theta_{t\xi}^{h,\epsilon}(s) \right)-D_vD_x f\left(s,\theta_{t\xi}(s)\right) \right| \\
	& + \left|\widetilde{\de^\epsilon Y}(s)\right|^2 \left|\left(D_y \frac{d}{d\nu}D_x f\right) \left(s,\theta_{t\xi}^{h,\epsilon}(s)\right)\left(\widetilde{Y_{t\xi}^{h,\epsilon}}(s)\right) -\left(D_y \frac{d}{d\nu}D_x f\right) (s,\theta_{t\xi}(s))\left(\widetilde{Y_{t\xi}}(s)\right) \right|^2  \\
	& + \left| \de^\epsilon Y(s)\right|^2  \left|D_x^2 f\left(s,\theta_{t\xi}^{h,\epsilon}(s)\right)-D_x^2 f\left(s,\theta_{t\xi}(s)\right)  \right| \bigg\} dh ds.
\end{align*}
\normalsize
Applying \eqref{lem4_4} to \eqref{lem4_9}, one can find a constant $c(L,T)$ depending only on $(L,T)$, such that when $\lambda_v\geq c(L,T)$, we have
\begin{equation*}\label{lem4_6}
	\begin{split}
		& \lambda_v \e\widetilde{\e}\left[ \int_t^T\int_0^1 \left| \delta^{h,\epsilon}\mathbf{V}(s) \right|^2 dhds \right]  \le C(L,T,\lambda_v)\left[I_1(\epsilon)+I_2(\epsilon)+I_3(\epsilon)\right].
	\end{split}	
\end{equation*}
Using the last inequality with \eqref{lem4_4}, we have
\begin{equation}\label{lem4_7}
	\begin{split}
		&\e\bigg[\sup_{t\le s\le T}\left|\left(\delta^\epsilon Y(s),\delta^\epsilon p(s)\right)^\top \right|^2 + \int_t^T \left|\delta^\epsilon q(s) \right|^2ds \bigg]\le C(L,T,\lambda_v)\left[I_1(\epsilon)+I_2(\epsilon)+I_3(\epsilon)\right].
	\end{split}
\end{equation}
From Estimates \eqref{lem4_2} and \eqref{lem4_2'}, Assumptions (A1) and (A2), and the dominated convergence theorem, we therefore have
\small
\begin{align*}
	\lim_{\epsilon\to0}\e\widetilde{\e} \int_t^T\int_0^1   & \left| \widetilde{\de^\epsilon Y}(s)\right|^2  \left|D \frac{db_0}{d\nu}\left(s,\lr\left(Y_{t\xi}^{h,\epsilon}(s)\right)\right)\left(\widetilde{Y_{t\xi}^{h,\epsilon}}(s)\right) - D \frac{db_0}{d\nu}\left( s,\lr(Y_{t\xi}(s))\right)\left(\widetilde{Y_{t\xi}}(s)\right) \right| ^2  dh ds=0.
\end{align*}
\normalsize
Using similar approach for all other $I_\cdot$'s, we can deduce that altogether 
\begin{align}\label{lem4_10}
	\lim_{\epsilon\to0}\left[I_1(\epsilon)+I_2(\epsilon)+I_3(\epsilon)\right]=0.
\end{align}
Therefore, from \eqref{lem4_7}, we deduce \eqref{lem4_1}.  

If Assumption (A3') is satisfied, then, from \eqref{lem4_3}, \eqref{convex'}, Assumption (A2) and the Young's inequality with a weight of $\frac{\lambda_v}{L}$, we have
\begin{align*}
	&\e\left[\left(\delta^\epsilon p(T)\right)^\top  \delta^\epsilon Y(T)\right] \notag \\
	\le\ & -2\int_t^T\int_0^1 \e\left[  \lambda_v\left| \delta^{h,\epsilon}\mathbf{V}(s)\right|^2 +\lambda_x \left| \delta^\epsilon Y(s) \right|^2 \right] dh ds\\
	&+\int_t^T\int_0^1 \e\widetilde{\e} \Bigg\{  L_v\left|\delta^{h,\epsilon}\mathbf{V}(s)\right| \left|\widetilde{\delta^\epsilon Y}(s)\right| +L_x \left| \delta^\epsilon Y(s) \right| \left|\widetilde{\delta^\epsilon Y}(s)\right|   \Bigg\} dh ds+I_1(\epsilon)\\
	\le\ &  \int_t^T\int_0^1 \e \left[ -\lambda_v\left| \delta^{h,\epsilon}\mathbf{V}(s)\right|^2 - \left(2\lambda_x-\frac{L_v^2}{4\lambda_v}-L_x \right)  \left| \delta^\epsilon Y(s) \right|^2 \right] dh ds+I_1(\epsilon)\\
	\le\ &  -\lambda_v \int_t^T\int_0^1 \e \left[\left| \delta^{h,\epsilon}\mathbf{V}(s)\right|^2  \right] dh ds+I_1(\epsilon),
\end{align*}
where the last inequality follows from the small mean field effect condition in Assumption (A3'): $\lambda_x\geq\frac{L_v^2}{8\lambda_v}+\frac{L_x}{2}$. Similarly, from \eqref{lem4_3'}, we have
\begin{align*}
	& \e\left[\left(\delta^\epsilon p(T)\right)^\top  \delta^\epsilon Y(T)\right] \\
	\geq\  & \int_0^1 \e \widetilde{\e}\left[ 2\lambda_g  \left|\delta^\epsilon Y(T)\right|^2  -L_g \left|\delta^\epsilon Y(T)\right| \left| \widetilde{\delta^\epsilon Y}(T)\right| \right] dh-I_2(\epsilon)\\
	\geq\ & \int_0^1 \e \left[ (2\lambda_g-L_g)  \left|\delta^\epsilon Y(T)\right|^2   \right] dh-I_2(\epsilon)\geq -I_2(\epsilon).
\end{align*}
Therefore, we have
\begin{align*}
	\lambda_v \int_t^T\int_0^1 \e \left[\left| \delta^{h,\epsilon}\mathbf{V}(s)\right|^2  \right] dh ds \le I_1(\epsilon)+I_2(\epsilon).
\end{align*}
Then, from \eqref{lem4_4} and \eqref{lem4_10}, we achieve \eqref{lem4_1} for any $\lambda_v>0$.  The claimed G\^ateaux differentiability is then a consequence of \eqref{lem4_1} and Lemma~\ref{lem:3}. The proof of continuity of the processes $(\dr_\eta Y_{t\xi}(s),\dr_\eta p_{t\xi}(s),\dr_\eta q_{t\xi}(s))$ in $\xi$ is similar to that of \eqref{lem4_1}, and is thus omitted.

\section{Proof of Statements in Section~\ref{sec:state}}\label{pf:state}

\subsection{Proof of Lemma~\ref{prop:4}}\label{pf_prop4}

The solvability of FBSDEs \eqref{FB:dr'} is shown in a similar way to  that of Lemma~\ref{lem:3}, and is thus omitted. We only prove \eqref{prop4_1} here. From the uniqueness of the solution of FBSDEs \eqref{intro_2}, we see 
\begin{align*}
	\Theta_{tx\mu}(s)|_{x=\xi}=\Theta_{t\xi}(s),\quad s\in[t,T];
\end{align*}
see \cite{BR} for details. Then, from FBSDEs \eqref{FB:x} and \eqref{FB:dr'}, we know that the processes
\begin{align*}
	D_xY_{tx\mu}(s)\big|_{x=\xi}\ \eta+\bd_\eta Y_{t\xi}(s), \ D_xp_{tx\mu}(s)\big|_{x=\xi}\ \eta+\bd_\eta p_{t\xi}(s), \ D_xq_{tx\mu}(s)\big|_{x=\xi}\ \eta+\bd_\eta q_{t\xi}(s)
\end{align*}
satisfy the following FBSDEs	
\begin{align*}
		&D_xY_{tx\mu}(s)\big|_{x=\xi}\ \eta+\bd_\eta Y_{t\xi}(s)\\
		=\ & \eta+\int_t^s \bigg\{ b_1(r)\left[D_xY_{tx\mu}(r)\big|_{x=\xi}\ \eta+\bd_\eta Y_{t\xi}(r)\right]+b_2(r)\left[D_x v_{tx\mu}(r)\big|_{x=\xi}\ \eta+\bd_\eta v_{t\xi}(r)\right]\\
		&\qquad\qquad +\widetilde{\e}\left[\left[D \frac{db_0}{d\nu} (r,\lr(Y_{t\xi}(r)))\left(\widetilde{Y_{t\xi}}(r)\right)\right]^\top  \left(\widetilde{D_xY_{tx\mu}}(r)\left|_{x=\widetilde{\xi}}\ \right. \widetilde{\eta}+\widetilde{\bd_\eta Y_{t\xi}}(r)\right)\right] \bigg\}dr\\
		&\;\; +\int_t^s \bigg\{ \sigma_1(r)\left[D_xY_{tx\mu}(r)\big|_{x=\xi}\ \eta+\bd_\eta Y_{t\xi}(r)\right]+\sigma_2(r)\left[D_x v_{tx\mu}(r)\big|_{x=\xi}\ \eta+\bd_\eta v_{t\xi}(r)\right]\\
		&\qquad\qquad +\widetilde{\e}\left[\left[D \frac{d\sigma_0}{d\nu} (r,\lr(Y_{t\xi}(r)))\left(\widetilde{Y_{t\xi}}(r)\right)\right]^\top  \left(\widetilde{D_xY_{tx\mu}}(r)\left|_{x=\widetilde{\xi}}\ \right. \widetilde{\eta}+\widetilde{\bd_\eta Y_{t\xi}}(r)\right)\right] \bigg\}dB(r),\\
		&D_xp_{tx\mu}(s)\big|_{x=\xi}\ \eta+\bd_\eta p_{t\xi}(s)\\
		=\ & \left[D_x^2 g \left(Y_{t\xi}(T),\lr(Y_{t\xi}(T))\right)\right]^\top \left[D_x Y_{tx\mu}(T)\big|_{x=\xi}\ \eta+ \bd_\eta Y_{t\xi}(s)\right]\\
		&+\widetilde{\e}\left[\left[\left(D_y \frac{d}{d\nu}D_x g\right)  (Y_{t\xi}(T),\lr(Y_{t\xi}(T)))\left(\widetilde{Y_{t\xi}}(T)\right)\right]^\top  \left(\widetilde{D_xY_{tx\mu}}(T)\left|_{x=\widetilde{\xi}}\ \right. \widetilde{\eta}+\widetilde{\bd_\eta Y_{t\xi}}(T)\right)\right]\\
		& +\int_s^T \bigg\{b _1(r)^\top \left[D_xp_{tx\mu}(r)\big|_{x=\xi}\ \eta+ \bd_\eta p_{t\xi}(r)\right]+\sum_{j=1}^n \left(\sigma_1^j (r)\right)^\top \left[D_x q^j_{tx\mu}(r)\big|_{x=\xi}\ \eta+\bd_\eta q^j_{t\xi}(r)\right]\\
		&\quad\qquad+\left[D_x^2 f (r,\theta_{t\xi}(r))\right]^\top  \left[D_x Y_{tx\mu}(r)\big|_{x=\xi}\ \eta+\bd_\eta Y_{t\xi}(r)\right] \\
		&\quad\qquad +\widetilde{\e}\left[\left[\left(D_y \frac{d}{d\nu}D_x f\right) (r,\theta_{t\xi}(r))\left(\widetilde{Y_{t\xi}}(r)\right)\right]^\top  \left(\widetilde{\bd_\eta Y_{t\xi}}(r)+\widetilde{D_xY_{tx\mu}}(r)\left|_{x=\widetilde{\xi}}\ \right. \widetilde{\eta}\right)\right]\\
		&\quad\qquad+ \left[D_vD_x f (r,\theta_{t\xi}(r))\right]^\top  \left[D_x v_{tx\mu}(r)\big|_{x=\xi}\ \eta+\bd_\eta v_{t\xi}(r)\right]\bigg\}dr\\
		&  -\int_s^T \left[D_x q_{tx\mu}(r)\big|_{x=\xi}\ \eta+\bd_\eta q_{t\xi}(r)\right] dB(r),\quad s\in[t,T],
\end{align*}
and	
\begin{align*}
	&D_x v_{tx\mu}(s)\big|_{x=\xi}\ \eta+\bd_\eta v_{t\xi}(s)\\
	=\ &\left[D_x\hv (s,\Theta_{t\xi}(s)) \right]^\top \left[D_x Y_{tx\mu}(s)\big|_{x=\xi}\ \eta+\bd_\eta Y_{t\xi}(s)\right]\\
	&+\widetilde{\e}\left[\left[D\frac{d\hv}{d\nu} (s,\Theta_{t\xi}(s))\left(\widetilde{Y_{t\xi}}(s)\right)\right]^\top \left(\widetilde{D_xY_{tx\mu}}(s)\left|_{x=\widetilde{\xi}}\ \right. \widetilde{\eta}+\widetilde{\bd_\eta Y_{t\xi}}(s)\right)\right]\\
	&+\left[D_p\hv (s,\Theta_{t\xi}(s))\right]^\top \left[D_x p_{tx\mu}(s)\big|_{x=\xi}\ \eta+\bd_\eta p_{t\xi}(s)\right]\\
	&+\sum_{j=1}^n \left[D_{q^j}\hv (s,\Theta_{t\xi}(s))\right]^\top \left[D_x q^j_{tx\mu}(s)\big|_{x=\xi}\ \eta+\bd_\eta q^j_{t\xi}(s)\right],\quad s\in[t,T],
\end{align*}
which altogether form exactly the FBSDEs system \eqref{FB:dr} for $(\dr_\eta Y_{t\xi},\dr_\eta p_{t\xi},\dr_\eta q_{t\xi})$. From the uniqueness of the solution of FBSDEs \eqref{FB:dr}, we also deduce \eqref{prop4_1}.

\subsection{Proof of Lemma~\ref{prop:3}}\label{pf_prop3} 

In view of Lemma~\ref{prop:4}, the system of FBSDEs \eqref{FB:mu'} also reads
\begin{align}
		\dr_\eta Y_{tx\xi}(s)=\ & \int_t^s \bigg\{\widetilde{\e}\left[\left[D \frac{db_0}{d\nu} (r,\lr(Y_{t\xi}(r)))\left(\widetilde{Y_{t\xi}}(r)\right)\right]^\top \widetilde{\dr_\eta Y_{t\xi}}(r)\right] \notag\\
		&\qquad +b_1(r)\dr_\eta Y_{tx\xi}(r)+b_2(r)\dr_\eta v_{tx\xi}(r)\bigg\}dr \notag\\
		&+\int_t^s \bigg\{\widetilde{\e}\left[\left[D \frac{d\sigma_0}{d\nu} \left(r,\lr(Y_{t\xi}(r))\right)\left(\widetilde{Y_{t\xi}}(r)\right)\right]^\top \widetilde{\dr_\eta Y_{t\xi}}(r)\right] \notag\\
		&\quad\qquad +\sigma_1(r)\dr_\eta Y_{tx\xi}(r)+\sigma_2(r)\dr_\eta v_{tx\xi}(r)\bigg\}dB(r), \notag\\
		\dr_\eta p_{tx\xi}(s)=\ & \left[D_x^2 g  (Y_{tx\mu}(T),\lr(Y_{t\xi}(T)))\right]^\top\dr_\eta Y_{tx\xi}(T) \notag\\
		&+\widetilde{\e}\left\{\left[\left(D_y \frac{d}{d\nu}D_x g\right) (Y_{tx\mu}(T),\lr(Y_{t\xi}(T)))\left(\widetilde{Y_{t\xi}}(T)\right)\right]^\top \widetilde{\dr_\eta Y_{t\xi}}(T)\right\} \notag\\
		&+\int_s^T \bigg\{b _1(r)^\top \dr_\eta p_{tx\xi}(r)+\sum_{j=1}^n\left(\sigma_1^j(r)\right)^\top  \dr_\eta q^j_{tx\xi}(r) +\left[D_x^2 f (r,\theta_{tx\mu}(r))\right]^\top \dr_\eta Y_{tx\xi}(r) \notag\\
		&\quad\qquad +\widetilde{\e}\left[\left[\left(D_y \frac{d}{d\nu}D_x f\right) (r,\theta_{tx\mu}(r))\left(\widetilde{Y_{t\xi}}(r)\right)\right]^\top \widetilde{\dr_\eta Y_{t\xi}}(r)\right]  \notag\\
		&\quad\qquad +\left[D_vD_x f (r,\theta_{tx\mu}(r))\right]^\top\dr_\eta v_{tx\xi}(r)\bigg\}dr \notag \\
		&-\int_s^T \dr_\eta q_{tx\xi}(r)dB(r),\quad s\in[t,T], \label{FB:mu}
\end{align}
where
\begin{align*}
	\dr_\eta v_{tx\xi}(s):=\ &\left[D_x\hv (s,\theta_{tx\mu}(s)) \right]^\top \dr_\eta Y_{tx\xi}(s)+\widetilde{\e}\left[\left[D_y\frac{d\hv}{d\nu}  (s,\Theta_{tx\mu}(s))\left(\widetilde{Y_{t\xi}}(s)\right) \right]^\top\widetilde{\dr_\eta Y_{t\xi}}(s) \right]\\
	&+\left[D_p\hv(s,\Theta_{tx\mu}(s)) \right]^\top \dr_\eta p_{tx\xi}(s)+\sum_{j=1}^n \left[D_{q^j}\hv (s,\Theta_{tx\mu}(s)) \right]^\top\dr_\eta q^j_{tx\xi}(s),\quad s\in[t,T],
\end{align*}
and $\left(\widetilde{Y_{t\xi}}(s),\widetilde{\dr_\eta Y_{t\xi}}(s)\right)$ is an independent copy of $(Y_{t\xi}(s),\dr_\eta Y_{t\xi}(s))$. Then, the proof of the well-posedness of the above FBSDEs is similar as that of Lemma~\ref{lem:3}, and the proof for the G\^ateaux differentiability is similar as that of Theorem~\ref{lem:4}, which is omitted here. And then we aim to prove \eqref{prop3_1}. Following standard arguments of SDEs, we have
\begin{equation}\label{prop3_2}
	\e\bigg[\sup_{t\le s\le T}|\dr_\eta Y_{tx\xi}(s)|^2\bigg]\le C(L,T)\e\bigg[\sup_{t\le s\le T}|\dr_\eta Y_{t\xi}(s)|^2+\int_t^T |\dr_\eta v_{tx\xi}(s)|^2ds\bigg].
\end{equation}
Using standard arguments of BSDEs, from \eqref{prop3_2}, we have
\begin{equation}\label{prop3_3}
	\begin{split}
		&\e\bigg[\sup_{t\le s\le T}|\dr_\eta p_{tx\xi}(s)|^2+ \int_t^T |\dr_\eta q_{tx\xi}(s)|^2 ds\bigg]\\
		\le\ & C(L,T)\e\bigg[\sup_{t\le s\le T}|\dr_\eta Y_{t\xi}(s)|^2+\int_t^T |\dr_\eta v_{tx\xi}(s)|^2ds\bigg].
	\end{split}
\end{equation}    
In accordance with It\^o's formula, Conditions \eqref{optimal_condition} and \eqref{convex}, and Assumptions (A2), we have
\begin{align}
		&\e\left[\left(\dr_\eta p_{tx\xi}(T)\right)^\top  \dr_\eta Y_{tx\xi}(T)\right] \notag \\
		=\ & \int_t^T \e\widetilde{\e}\bigg\{ \left(\dr_\eta p_{tx\xi}(s)\right)^\top  \left[D \frac{db_0}{d\nu}  (s,\lr(Y_{t\xi}(s)))\left(\widetilde{Y_{t\xi}}(s)\right) \right]^\top \widetilde{\dr_\eta Y_{t\xi}}(s)  \notag \\
		&\qquad\qquad +\sum_{j=1}^n \left(\dr_\eta q^j_{tx\xi}(s)\right)^\top  \left[D \frac{d\sigma^j_0}{d\nu} (s,\lr(Y_{t\xi}(s)))\left(\widetilde{Y_{t\xi}}(s)\right) \right]^\top \widetilde{\dr_\eta Y_{t\xi}}(s)  \notag \\
		&\qquad\qquad - \left[\left(
		\begin{array}{cc}
			D_v^2 f & D_vD_x f\\
			D_xD_v f & D_x^2 f
		\end{array}
		\right)(s,\theta_{tx\mu}(s))\right] \left(
		\begin{array}{cc}
			\dr_\eta v_{tx\xi}(s)\\
			\dr_\eta Y_{tx\xi}(s)
		\end{array}
		\right)^{\otimes 2}  \notag \\
		&\qquad\qquad -\left(\dr_\eta v_{tx\xi}(s)\right)^\top  \left[\left(D_y\frac{d}{d\nu}D_v f\right) (s,\theta_{tx\mu}(s))\left(\widetilde{Y_{t\xi}}(s)\right) \right]^\top \widetilde{\dr_\eta Y_{t\xi}}(s)  \notag  \\
		&\qquad\qquad -\left(\dr_\eta Y_{tx\xi}(s)\right)^\top  \left[\left(D_y \frac{d}{d\nu}D_x f\right) (s,\theta_{tx\mu}(s))\left(\widetilde{Y_{t\xi}}(s)\right) \right]^\top \widetilde{\dr_\eta Y_{t\xi}}(s) \bigg\}ds \label{prop3_4} \\
		\le\ & -\lambda_v \e\left[\int_t^T |\dr_\eta v_{tx\xi}(s)|^2 ds\right] \notag  \\
		&+C(L)\left(1+\frac{1}{\lambda_v}\right)\e\bigg[\int_t^T \left|\left(\dr_\eta Y_{t\xi}(s),\dr_\eta Y_{tx\xi}(s),\dr_\eta p_{tx\xi}(s),\dr_\eta q_{tx\xi}(s)\right)^\top\right|^2 ds\bigg].  \notag 
\end{align}	
Therefore, we have
\begin{align*}
	&\lambda_v \e\left[\int_t^T |\dr_\eta v_{tx\xi}(s)|^2 ds\right]\\
	\le\ & C(L,T)\left(1+\frac{1}{\lambda_v}\right)\e\bigg[\sup_{t\le s\le T}\left|\left(\dr_\eta Y_{t\xi}(s),\dr_\eta Y_{tx\xi}(s),\dr_\eta p_{tx\xi}(s)\right)^\top\right|^2 +\int_t^T |\dr_\eta q_{tx\xi}(s)|^2 ds\bigg].
\end{align*}
Substituting \eqref{prop3_2} and \eqref{prop3_3} into the last inequality, we have
\begin{align*}
	\lambda_v \e\left[\int_t^T |\dr_\eta v_{tx\xi}(s)|^2 ds\right]\le C(L,T)\left(1+\frac{1}{\lambda_v}\right)\e\bigg[\sup_{t\le s\le T}\left|\dr_\eta Y_{t\xi}(s)\right|^2+\int_t^T |\dr_\eta v_{tx\xi}(s)|^2 ds\bigg].
\end{align*}
Therefore, there exists a constant $c(L,T)$ depending only on $(L,T)$, such that when $\lambda_v\geq c(L,T)$, we have
\begin{equation*}\label{prop3_7}
	\e\left[\int_t^T |\dr_\eta v_{tx\xi}(s)|^2 ds\right]\le C(L,T,\lambda_v)\e\bigg[\sup_{t\le s\le T}\left|\dr_\eta Y_{t\xi}(s)\right|^2\bigg].
\end{equation*}
Then, from Estimate \eqref{lem3_1}, we know that
\begin{equation}\label{prop3_8}
	\begin{split}
		\e\left[\int_t^T |\dr_\eta v_{t\xi}(s)|^2 ds\right]\le C(L,T,\lambda_v)\e\left[|\eta|^2\right].
	\end{split}
\end{equation}
Substituting Estimate \eqref{lem3_1} and using the last inequality for \eqref{prop3_2} and \eqref{prop3_3}, we obtain \eqref{prop3_1}.

If Assumption (A3') is satisfied, then, from \eqref{prop3_4}, \eqref{convex'}, Assumption (A2) and the Young's inequality with a weight of $\frac{\lambda_v}{L}$, we have
\begin{align}
	&\e\left[\left(\dr_\eta p_{tx\xi}(T)\right)^\top  \dr_\eta Y_{tx\xi}(T)\right]  \notag \\
	\le\ & -2\lambda_v \e\left[\int_t^T |\dr_\eta v_{tx\xi}(s)|^2 ds\right]-2\lambda_x \e\left[\int_t^T |\dr_\eta Y_{tx\xi}(s)|^2 ds\right]   \notag \\
	&+\int_t^T \e\widetilde{\e}\bigg[ L_v \left|\dr_\eta v_{tx\xi}(s)\right|  \left|\widetilde{\dr_\eta Y_{t\xi}}(s) \right|  +L_x \left|\dr_\eta Y_{tx\xi}(s)\right| \left|\widetilde{\dr_\eta Y_{t\xi}}(s)\right| \bigg]ds \notag \\
	\le\ & -\lambda_v \e\left[\int_t^T |\dr_\eta v_{tx\xi}(s)|^2 ds\right]-\left(2\lambda_x-\frac{L_v^2}{4\lambda_v}-L_x \right)\e\left[\int_t^T |\dr_\eta Y_{tx\xi}(s)|^2 ds\right]   \notag \\
	&+ \frac{L_x}{4} \e\left[\int_t^T |\dr_\eta Y_{t\xi}(s)|^2 ds\right] \notag \\
	\le\ &  -\lambda_v \e\left[\int_t^T |\dr_\eta v_{tx\xi}(s)|^2 ds\right]+ \frac{L}{4} \e\left[\int_t^T |\dr_\eta Y_{t\xi}(s)|^2 ds\right], \label{small_mf_eff_1}
\end{align}
where the last inequality follows from the small mean field effect condition in Assumption (A3'): $\lambda_x\geq\frac{L_v^2}{8\lambda_v}+\frac{L_x}{2}$. On the other hand, from \eqref{convex'}, we know that
\begin{align*}
	\e\left[\left(\dr_\eta p_{tx\xi}(T)\right)^\top  \dr_\eta Y_{tx\xi}(T)\right] =\e\left[ \left(\dr_\eta Y_{tx\xi}(T)\right)^\top   \left[D_x^2 g (Y_{tx\mu}(T),\lr(Y_{t\xi}(T))) \right]^\top\dr_\eta Y_{tx\xi}(T)\right]\geq 0.
\end{align*}
Therefore, 
\begin{align*}
	\lambda_v \e\left[\int_t^T |\dr_\eta v_{tx\xi}(s)|^2 ds\right]\le  \frac{L}{4} \e\left[\int_t^T |\dr_\eta Y_{t\xi}(s)|^2 ds\right].
\end{align*}
Then, from Lemma~\ref{lem:3}, we have Estimate \eqref{prop3_8},  which  yields~\eqref{prop3_1}.

\subsection{Proof of Theorem~\ref{prop:5}}\label{pf_prop5}
The proof of the well-posedness of FBSDEs \eqref{FB:xi_y} and \eqref{FB:mu_y}, and the derivation of Estimates \eqref{prop5_01} and \eqref{prop5_02} are similar to those leading to Lemmas~\ref{lem:3} and \ref{prop:3}, and we omit here. We only prove \eqref{prop5_03} and \eqref{prop5_04} here. From the SDE in \eqref{FB:xi_y}, for $s\in[t,T]$, we have
\begin{align*}
	\widehat{\e}\left[\bd Y_{t\xi}\left(s,\widehat{\xi}\right) \widehat{\eta}\right]=\ & \int_t^s \bigg\{\widetilde{\e}\widehat{\e}\bigg[\left[D \frac{db_0}{d\nu} (r,\lr(Y_{t\xi}(r)))\left(\left.\widetilde{Y_{tz\mu}}(r)\right|_{z=\widehat{\xi}}\right) \right]^\top \left.\widetilde{D_zY_{tz\mu}}(r)\right|_{z=\widehat{\xi}} \widehat{\eta}\bigg]\\
	&\ \qquad +\widetilde{\e}\bigg[\left[D \frac{db_0}{d\nu} (r,\lr(Y_{t\xi}(r)))\left(\widetilde{Y_{t\xi}}(r)\right) \right]^\top \widehat{\e}\left[\widetilde{\bd Y_{t\xi}}\left(r,\widehat{\xi}\right) \widehat{\eta}\right]\bigg]\\
	&\ \qquad +b_1(r)\widehat{\e}\left[\bd Y_{t\xi}\left(r,\widehat{\xi}\right) \widehat{\eta}\right]+b_2(r)\widehat{\e}\left[\bd v_{t\xi}\left(r,\widehat{\xi}\right) \widehat{\eta}\right]\bigg\}dr\\
	& +\int_t^s \bigg\{\widetilde{\e}\widehat{\e}\bigg[\left[D \frac{d\sigma_0}{d\nu} (r,\lr(Y_{t\xi}(r)))\left(\left.\widetilde{Y_{tz\mu}}(r)\right|_{z=\widehat{\xi}}\right)  \right]^\top \left.\widetilde{D_zY_{tz\mu}}(r)\right|_{z=\widehat{\xi}} \widehat{\eta}\bigg]\\
	&\ \quad\qquad +\widetilde{\e}\bigg[\left[D \frac{d\sigma_0}{d\nu} (r,\lr(Y_{t\xi}(r)))\left(\widetilde{Y_{t\xi}}(r)\right) \right]^\top \widehat{\e}\left[\widetilde{\bd Y_{t\xi}}\left(r,\widehat{\xi}\right) \widehat{\eta}\right]\bigg]\\
	&\ \quad\qquad +\sigma_1(r)\widehat{\e}\left[\bd Y_{t\xi}\left(r,\widehat{\xi}\right) \widehat{\eta}\right]+\sigma_2(r)\widehat{\e}\left[\bd v_{t\xi}\left(r,\widehat{\xi}\right) \widehat{\eta}\right] \bigg\}dB(r).
\end{align*}	
Taking into account that $(\widehat{\xi},\widehat{\eta})$ are independent of $(\xi,\eta)$ and $(\widetilde{\xi},\widetilde{\eta})$, and of the same law as $(\widetilde{\xi},\widetilde{\eta})$, we see that, for instance,
\begin{equation}\label{prop5_5}
	\begin{split}
		&\widetilde{\e}\widehat{\e}\bigg[\left[D \frac{db_0}{d\nu} (r,\lr(Y_{t\xi}(r)))\left(\left.\widetilde{Y_{tz\mu}}(r)\right|_{z=\widehat{\xi}}\right)  \right]^\top \left.\widetilde{D_zY_{tz\mu}}(r)\right|_{z=\widehat{\xi}} \widehat{\eta}\bigg]\\
		=\ &\widetilde{\e} \bigg[\left[D \frac{db_0}{d\nu} (r,\lr(Y_{t\xi}(r)))\left(\widetilde{Y_{t\xi}}(r)\right)  \right]^\top \left.\widetilde{D_zY_{tz\mu}}(r)\right|_{z=\widetilde{\xi}} \widetilde{\eta}\bigg].
	\end{split}
\end{equation}
Therefore, we have
\begin{align}
		&\widehat{\e}\left[\bd Y_{t\xi}\left(s,\widehat{\xi}\right) \widehat{\eta}\right] \notag \\
		=\ & \int_t^s \bigg\{\widetilde{\e} \bigg[\left[D \frac{db_0}{d\nu} (r,\lr(Y_{t\xi}(r)))\left(\widetilde{Y_{t\xi}}(r)\right)  \right]^\top \left(\left.\widetilde{D_zY_{tz\mu}}(r)\right|_{z=\widetilde{\xi}} \widetilde{\eta}+ \widehat{\e}\left[\widetilde{\bd Y_{t\xi}}\left(r,\widehat{\xi}\right) \widehat{\eta}\right]\right)\bigg] \notag\\
		&\quad\qquad +b_1(r)\widehat{\e}\left[\bd Y_{t\xi}\left(r,\widehat{\xi}\right) \widehat{\eta}\right]+b_2(r)\widehat{\e}\left[\bd v_{t\xi}\left(r,\widehat{\xi}\right) \widehat{\eta}\right]\bigg\}dr \notag\\
		& +\int_t^s \bigg\{\widetilde{\e} \bigg[\left[D \frac{d\sigma_0}{d\nu} (r,\lr(Y_{t\xi}(r)))\left(\widetilde{Y_{t\xi}}(r)\right) \right]^\top \left(\left.\widetilde{D_zY_{tz\mu}}(r)\right|_{z=\widetilde{\xi}} \widetilde{\eta}+ \widehat{\e}\left[\widetilde{\bd Y_{t\xi}}\left(r,\widehat{\xi}\right) \widehat{\eta}\right]\right)\bigg] \notag \\
		&\qquad\qquad +\sigma_1(r)\widehat{\e}\left[\bd Y_{t\xi}\left(r,\widehat{\xi}\right) \widehat{\eta}\right]+\sigma_2(r)\widehat{\e}\left[\bd v_{t\xi}\left(r,\widehat{\xi}\right) \widehat{\eta}\right] \bigg\}dB(r). \label{prop5_1}
\end{align}
Using similar approach for all other terms, from \eqref{FB:xi_y_v}, we can deduce that 
\begin{align}
	&\widehat{\e}\left[\bd v_{t\xi}\left(s,\widehat{\xi}\right)\widehat{\eta}\right] \notag \\
	=\ & \left[D_x\hv (s,\Theta_{t\xi}(s)) \right]^\top \widehat{\e}\left[\bd Y_{t\xi}\left(s,\widehat{\xi}\right)\widehat{\eta}\right] \notag \\
	&+\widetilde{\e}\bigg[\left[D_y\frac{d\hv}{d\nu} (s,\Theta_{t\xi}(s))\left(\widetilde{Y_{t\xi}}(s)\right) \right]^\top\left( \left.\widetilde{D_z Y_{tz\mu}}(s)\right|_{z=\widetilde{\xi}}\widetilde{\eta} +\widehat{\e}\left[\widetilde{\bd Y_{t\xi}}\left(s,\widehat{\xi}
	\right)\widehat{\eta}\right]\right)\bigg] \notag \\
	&+\left[D_p\hv (s,\Theta_{t\xi}(s)) \right]^\top\bd p_{t\xi}\left(s,\widehat{\xi}\right)\widehat{\eta}+\sum_{j=1}^n \left[D_{q^j}\hv(s,\Theta_{t\xi}(s)) \right]^\top\bd q^j_{t\xi}\left(s,\widehat{\xi}\right)\widehat{\eta}; \label{prop5_2}
\end{align}
from the BSDE in \eqref{FB:xi_y}, we also have	
\small
\begin{align}
	& \widehat{\e}\left[\bd p_{t\xi}\left(s,\widehat{\xi} \right)\widehat{\eta}\right] \notag \\
	=\ & \left[D_x^2 g (Y_{t\xi}(T),\lr(Y_{t\xi}(T))) \right]^\top\widehat{\e}\left[\bd Y_{t\xi}\left(T,\widehat{\xi}\right)\widehat{\eta}\right] \notag \\
	&+\widetilde{\e}\bigg\{\left[\left(D_y \frac{d}{d\nu}D_x g\right)  (Y_{t\xi}(T),\lr(Y_{t\xi}(T)))\left(\widetilde{Y_{t\xi}}(T)\right)  \right]^\top \left( \left.\widetilde{D_zY_{tz\mu}}(T)\right|_{z=\widetilde{\xi}} \widetilde{\eta}+\widehat{\e}\left[\widetilde{\bd Y_{t\xi}}\left(T,\widehat{\xi}\right) \widehat{\eta}\right]\right) \bigg\} \notag \\
	&+\int_s^T \bigg\{b_1 (r)^\top  \widehat{\e}\left[\bd p_{t\xi}\left(r,\widehat{\xi}\right)\widehat{\eta}\right]+\sum_{j=1}^n\left(\sigma_1^j (r)\right)^\top  \widehat{\e}\left[ \bd q^j_{t\xi}\left(r,\widehat{\xi}\right)\widehat{\eta} \right] \notag \\
	&\ \quad\qquad +\left[D_x^2 f (r,\theta_{t\xi}(r)) \right]^\top \widehat{\e}\left[\bd Y_{t\xi}\left(r,\widehat{\xi}\right)\widehat{\eta}\right] +\left[D_vD_x f (r,\theta_{t\xi}(r)) \right]^\top \widehat{\e}\left[\bd v_{t\xi}\left(r,\widehat{\xi}\right) \widehat{\eta}\right] \notag \\
	&\ \quad\qquad +\widetilde{\e}\bigg[\left[\left(D_y \frac{d}{d\nu}D_x f\right)(r, \theta_{t\xi}(r))\left(\widetilde{Y_{t\xi}}(r)\right) \right]^\top\left(\left.\widetilde{D_zY_{tz\mu}}(r)\right|_{z=\widetilde{\xi}} \widetilde{\eta} +\widehat{\e}\left[ \widetilde{\bd Y_{t\xi}}\left(r,\widehat{\xi}\right)\widehat{\eta}\right]\right)\bigg]\bigg\}dr \notag \\
	&-\int_s^T \widehat{\e}\left[ \bd q_{t\xi}\left(r,\widehat{\xi}\right)\widehat{\eta} \right] dB(r), \label{prop5_3}
\end{align}
\normalsize
From \eqref{prop5_1}, \eqref{prop5_2} and \eqref{prop5_3}, we know that the component processes
\begin{align*}
	\left(\widehat{\e}\left[\bd Y_{t\xi}\left(s,\widehat{\xi}\right) \widehat{\eta}\right],\ \widehat{\e}\left[\bd p_{t\xi}\left(s,\widehat{\xi}\right) \widehat{\eta}\right],\ \widehat{\e}\left[\bd q_{t\xi}\left(s,\widehat{\xi}\right) \widehat{\eta}\right] \right)
\end{align*}
satisfy FBSDEs \eqref{FB:dr'}. From the uniqueness result of FBSDEs \eqref{FB:dr'}, we obtain \eqref{prop5_03}. Next, we shall prove \eqref{prop5_04}. From the SDE in \eqref{FB:mu_y}, \eqref{prop5_03} and \eqref{prop5_5}, we have for $s\in[t,T]$,
\begin{align*}
	&\widehat{\e}\left[\dr Y_{tx\mu}\left(s,\widehat{\xi}\right)\widehat{\eta}\right]\\
	=\ & \int_t^s \bigg\{\widetilde{\e}\bigg[\left[D \frac{db_0}{d\nu} (r,\lr(Y_{t\xi}(r)))\left(\widetilde{Y_{t\xi}}(r)\right) \right]^\top \left(\left.\widetilde{D_z Y_{tz\mu}}(r)\right|_{z=\widetilde{\xi}} \widetilde{\eta}+\widetilde{\bd_\eta Y_{t\xi}}\left(r\right)\right)\bigg]\\
	&\qquad +b_1(r)\widehat{\e}\left[\dr Y_{tx\mu}\left(r,\widehat{\xi}\right)\widehat{\eta}\right]+\widehat{\e}\left[b_2(r)\dr v_{tx\mu}\left(r,\widehat{\xi}\right)\widehat{\eta}\right] \bigg\}dr\\
	& +\int_t^s \bigg\{\widetilde{\e}\bigg[\left[D \frac{d\sigma_0}{d\nu} (r,\lr(Y_{t\xi}(r)))\left(\widetilde{Y_{t\xi}}(r)\right)  \right]^\top \left( \left.\widetilde{D_z Y_{tz\mu}}(r)\right|_{z=\widetilde{\xi}}\widetilde{\eta} + \widetilde{\bd_\eta Y_{t\xi}}\left(r\right)\right)\bigg]\\
	&\quad\qquad +\sigma_1(r)\widehat{\e}\left[\dr Y_{tx\mu}\left(r,\widehat{\xi}\right)\widehat{\eta}\right]+\sigma_2(r)\widehat{\e}\left[\dr v_{tx\mu}\left(r,\widehat{\xi}\right)\widehat{\eta}\right]\bigg\}dB(r).
\end{align*}
In a similar way, we deduce
\begin{align*}
	&\widehat{\e}\left[\dr v_{tx\mu}\left(s,\widehat{\xi}\right)\widehat{\eta}\right]\\
	=&\left[D_x\hv (s,\Theta_{tx\mu}(s)) \right]^\top\widehat{\e}\left[\dr Y_{tx\mu}\left(s,\widehat{\xi}\right)\widehat{\eta}\right]\\
	&+\widetilde{\e}\bigg[\left[D_y\frac{d\hv}{d\nu} (s,\Theta_{tx\mu}(s))\left(\widetilde{Y_{t\xi}}(s)\right) \right]^\top \left(\left.\widetilde{D_z Y_{tz\mu}}(s)\right|_{z=\widetilde{\xi}}\widetilde{\eta} + \widetilde{\bd_\eta Y_{t\xi}}\left(s\right) \right) \bigg]\\
	&+\left[D_p\hv (s,\Theta_{tx\mu}(s)) \right]^\top\widehat{\e}\left[\dr p_{tx\mu}\left(s,\widehat{\xi}\right)\widehat{\eta}\right]+\sum_{j=1}^n \left[D_{q^j}\hv (s,\Theta_{tx\mu}(s)) \right]^\top \widehat{\e}\left[\dr q^j_{tx\mu}\left(s,\widehat{\xi}\right)\widehat{\eta}\right],
\end{align*}
and
\begin{align*}
	&\widehat{\e}\left[\dr p_{tx\mu}\left(s,\widehat{\xi}\right) \widehat{\eta} \right] \\
	=\ &\left[D_x^2 g (Y_{tx\mu}(T),\lr(Y_{t\xi}(T))) \right]^\top \widehat{\e}\left[\dr Y_{tx\mu}\left(T,\widehat{\xi}\right) \widehat{\eta} \right] \\
	&+\widetilde{\e}\bigg[\left[\left(D_y \frac{d}{d\nu}D_x g\right) (Y_{tx\mu}(T),\lr(Y_{t\xi}(T)))\left(\widetilde{Y_{t\xi}}(T)\right) \right]^\top \left(\left.\widetilde{D_z Y_{tz\mu}}(T)\right|_{z=\widetilde{\xi}} \widetilde{\eta} +  \widetilde{\bd_\eta Y_{t\xi}}\left(T\right) \right) \bigg]\\
	&+\int_s^T \bigg\{b_1 (r)^\top \widehat{\e}\left[\dr p_{tx\mu}\left(r,\widehat{\xi}\right) \widehat{\eta} \right] +\sum_{j=1}^n\left(\sigma_1^j (r)\right)^\top  \widehat{\e}\left[\dr q^j_{tx\mu}\left(r,\widehat{\xi}\right) \widehat{\eta}\right] \\
	&
	\ \quad\qquad +\left[D_x^2 f (r,\theta_{tx\mu}(r)) \right]^\top \widehat{\e}\left[\dr Y_{tx\mu}\left(r,\widehat{\xi}\right) \widehat{\eta}\right] +\left[D_vD_x f (r,\theta_{tx\mu}(r))\right]^\top \widehat{\e}\left[\dr v_{tx\mu}\left(r,\widehat{\xi}\right) \widehat{\eta}\right] \\
	&\ \quad\qquad +\widetilde{\e}\bigg[\left[\left(D_y \frac{d}{d\nu}D_x f\right)(r,\theta_{tx\mu}(r))\left(\widetilde{Y_{t\xi}}(r)\right)  \right]^\top \left( \left.\widetilde{D_z Y_{tz\mu}}(r)\right|_{z=\widetilde{\xi}} \widetilde{\eta} + \widetilde{\bd_\eta Y_{t\xi}}\left(r\right)\right) \bigg]\bigg\}dr\\
	&-\int_s^T \widehat{\e}\left[\dr q_{tx\mu}\left(r,\widehat{\xi}\right)\widehat{\eta}\right] dB(r).
\end{align*}
Therefore, the component processes
\begin{align*}
	\left(\widehat{\e}\left[\dr Y_{tx\mu}\left(s,\widehat{\xi}\right) \widehat{\eta}\right],\ \widehat{\e}\left[\dr p_{tx\mu}\left(s,\widehat{\xi}\right) \widehat{\eta}\right],\ \widehat{\e}\left[\dr q_{tx\mu}\left(s,\widehat{\xi}\right) \widehat{\eta}\right] \right)
\end{align*}
satisfy FBSDEs \eqref{FB:mu'}. From the uniqueness result of FBSDEs \eqref{FB:mu'}, we obtain \eqref{prop5_04}. 

\subsection{Proof of Theorem~\ref{prop:9}}\label{pf:prop:9}

To shorten the formulae and simplify the notations, we only establish the case $n=d=1$ for simplicity. We emphasize that the general case with dimension $n,d\geq 1$ can be obtained by using a straight-forward manner. We here give the respective systems of FBSDEs for the derivatives vector $(D_z \bd Y_{t\xi}(s,z),D_z \bd p_{t\xi}(s,z),D_z \bd q_{t\xi}(s,z))$ and $(D_z \dr Y_{tx\mu}(s,z),D_z \dr p_{tx\mu}(s,z),D_z \dr q_{tx\mu}(s,z))$ for $n=d=1$ here. From FBSDEs \eqref{FB:xi_y}, the G\^ateaux derivatives of $(\bd Y_{t\xi}(s,z),\bd p_{t\xi}(s,z),\bd q_{t\xi}(s,z))$ in $z$ can be characterized as the solution of the following FBSDEs: 
\begin{align}
	D_z\bd Y_{t\xi}(s,z)=\ & \int_t^s \bigg\{\widetilde{\e}\bigg[D \frac{db_0}{d\nu}(r,\lr(Y_{t\xi}(r)))\left(\widetilde{Y_{t\xi}}(r)\right) \widetilde{D_z \bd Y_{t\xi}}(r,z) \notag\\
	&\qquad\qquad +D \frac{db_0}{d\nu}(r,\lr(Y_{t\xi}(r)))\left(\widetilde{Y_{tz\mu}}(r)\right) \widetilde{D_z^2 Y_{tz\mu}}(r) \notag\\
	&\qquad\qquad + D^2 \frac{db_0}{d\nu}(r,\lr(Y_{t\xi}(r)))\left(\widetilde{Y_{tz\mu}}(r)\right) \left|\widetilde{D_zY_{tz\mu}}(r)\right|^2  \bigg] \notag\\
	&\qquad +b_1(r)D_z\bd Y_{t\xi}(r,z)+b_2(r)D_z\bd v_{t\xi}(r,z)\bigg\}ds  \notag\\
	&+\int_t^s \bigg\{\widetilde{\e}\bigg[D \frac{d\sigma_0}{d\nu}(r,\lr(Y_{t\xi}(r)))\left(\widetilde{Y_{t\xi}}(r)\right) \widetilde{D_z \bd Y_{t\xi}}(r,z) \notag\\
	&\quad\qquad\qquad +D \frac{d\sigma_0}{d\nu}(r,\lr(Y_{t\xi}(r)))\left(\widetilde{Y_{tz\mu}}(r)\right) \widetilde{D_z^2 Y_{tz\mu}}(r) \notag\\
	&\quad\qquad\qquad + D^2 \frac{d\sigma_0}{d\nu}(r,\lr(Y_{t\xi}(r)))\left(\widetilde{Y_{tz\mu}}(r)\right) \left|\widetilde{D_zY_{tz\mu}}(r)\right|^2  \bigg] \notag\\
	&\quad\qquad +\sigma_1(r)D_z\bd Y_{t\xi}(r,z)+\sigma_2(r)D_z\bd v_{t\xi}(r,z)\bigg\}dB(r), \notag\\
	D_z \bd p_{t\xi}(s,z)=\ &D_x^2 g(Y_{t\xi}(T),\lr(Y_{t\xi}(T))) D_z \bd Y_{t\xi}(T,z)  \notag\\
	& +\widetilde{\e}\bigg[D_y \frac{d}{d\nu}D_x g(Y_{t\xi}(T),\lr(Y_{t\xi}(T)))\left(\widetilde{Y_{t\xi}}(T)\right) \widetilde{D_z \bd Y_{t\xi}}(T,z)\notag\\
	&\qquad+D_y \frac{d}{d\nu}D_x g(Y_{t\xi}(T),\lr(Y_{t\xi}(T)))\left(\widetilde{Y_{tz\mu}}(T)\right) \widetilde{D_z^2 Y_{tz\mu}}(T)  \notag\\
	&\qquad + D_y^2 \frac{d}{d\nu}D_x g(Y_{t\xi}(T),\lr(Y_{t\xi}(T)))\left(\widetilde{Y_{ty\mu}}(T)\right) \left|\widetilde{D_zY_{tz\mu}}(T)\right|^2  \bigg] \notag\\
	& +\int_s^T \bigg\{ b_1(r) D_z \bd p_{t\xi}(r,z) + \sigma_1(r) D_z \bd q_{t\xi}(r,z)  \notag\\
	&\ \quad\qquad +D_x^2 f(r,\theta_{t\xi}(r)) D_z \bd Y_{t\xi}(r,z)+D_vD_x f(r,\theta_{t\xi}(r)) D_z\bd v_{t\xi}(r,z)  \notag\\
	&\ \quad\qquad +\widetilde{\e}\bigg[D_y \frac{d}{d\nu}D_x f(r,\theta_{t\xi}(r))\left(\widetilde{Y_{t\xi}}(r)\right) \widetilde{D_z \bd Y_{t\xi}}(r,z)\notag\\
	&\ \quad\qquad\qquad +D_y \frac{d}{d\nu}D_x f(r,\theta_{t\xi}(r))\left(\widetilde{Y_{tz\mu}}(r)\right)\widetilde{D_z^2 Y_{tz\mu}}(r)  \notag\\
	&\ \quad\qquad\qquad + D_y^2 \frac{d}{d\nu}D_x f(r,Y_{t\xi}(r))\left(\widetilde{Y_{tz\mu}}(r)\right)\left|\widetilde{D_z Y_{tz\mu}}(r)\right|^2  \bigg]  \bigg\}ds  \notag\\
	& -\int_s^T D_z\bd q_{t\xi}(r,z)dB(r),\quad (s,z)\in[t,T]\times\br, \label{FB:xi_yy}
\end{align}
where
\begin{align*}
	D_z \bd v_{t\xi}(s,z):=\ & D_x\hv (s,\Theta_{t\xi}(s)) D_z \bd Y_{t\xi}(s,z)\\
	&+\widetilde{\e}\bigg[D_y\frac{d\hv}{d\nu} (s,\Theta_{t\xi}(s)) \left(\widetilde{Y_{t\xi}}(s)\right) \widetilde{D_z \bd Y_{t\xi}}(s,z)\\
	&\qquad + D_y\frac{d\hv}{d\nu} (s,\Theta_{t\xi}(s))\left(\widetilde{Y_{tz\mu}}(s)\right) \widetilde{D_z^2 Y_{tz\mu}}(s)\\
	&\qquad + D_y^2\frac{d\hv}{d\nu} (s,\Theta_{t\xi}(s))\left(\widetilde{Y_{tz\mu}}(s)\right)\left| \widetilde{D_z Y_{tz\mu}}(s)\right|^2 \bigg]\\
	&+D_p\hv  (s,\Theta_{t\xi}(s)) D_z \bd p_{t\xi}(s,z)+ D_{q}\hv (s,\Theta_{t\xi}(s)) D_z \bd q_{t\xi}(s,z),
\end{align*}
and $\left(D_z Y_{tz\mu},D_z^2 Y_{tz\mu}\right)=\left(D_x Y_{tx\mu},D_x^2 Y_{tx\mu} \right)\!\Big|_{x=z}$, and $\widetilde{Y_{tz\mu}}(s)$, $\widetilde{Y_{t\xi}}(s)$, $\widetilde{D_z Y_{tz\mu}}(s)$, $\widetilde{D_z^2 Y_{tz\mu}}(s)$ and $\widetilde{D_z\bd Y_{t\xi}}(s,z)$ are the respective independent copy of $Y_{tz\mu}(s)$, $Y_{t\xi}(s)$, $D_z Y_{tz\mu}(s)$, $D_z^2 Y_{tz\mu}(s)$ and $D_z\bd Y_{t\xi}(s,z)$. From FBSDEs \eqref{FB:mu_y} and \eqref{FB:xi_yy}, the component G\^ateaux derivatives of $\dr Y_{tx\mu}(s,z)$, $\dr p_{tx\mu}(s,z)$ and $\dr q_{tx\mu}(s,z))$ in $z$ can be characterized as the solution of the following FBSDEs: 
\begin{align}
	D_z\dr Y_{tx\mu}(s,z)=\ & \int_t^s \bigg\{\widetilde{\e}\bigg[D \frac{db_0}{d\nu}(r,\lr(Y_{t\xi}(r)))\left(\widetilde{Y_{t\xi}}(r)\right) \widetilde{D_z\bd Y_{t\xi}}(r,z) \notag \\
	&\qquad\qquad +D \frac{db_0}{d\nu}(r,\lr(Y_{t\xi}(r)))\left(\widetilde{Y_{tz\mu}}(r)\right) \widetilde{D_z^2 Y_{tz\mu}}(r) \notag \\
	&\qquad\qquad + D^2 \frac{db_0}{d\nu}(r,\lr(Y_{t\xi}(r)))\left(\widetilde{Y_{tz\mu}}(r)\right) \left|\widetilde{D_z Y_{tz\mu}}(r)\right|^2 \bigg] \notag\\
	&\qquad +b_1(r)D_z\dr Y_{tx\mu}(r,z)+b_2(r)D_y\dr v_{tx\mu}(r,z)\bigg\}dr \notag \\
	&+\int_t^s \bigg\{\widetilde{\e}\bigg[D \frac{d\sigma_0}{d\nu}(r,\lr(Y_{t\xi}(r)))\left(\widetilde{Y_{t\xi}}(r)\right) \widetilde{D_z\bd Y_{t\xi}}(r,z) \notag \\
	&\quad\qquad\qquad +D \frac{d \sigma_0}{d\nu}(r,\lr(Y_{t\xi}(r)))\left(\widetilde{Y_{tz\mu}}(r)\right) \widetilde{D_z^2 Y_{tz\mu}}(r) \notag \\
	&\quad\qquad\qquad + D^2 \frac{d \sigma_0}{d\nu}(r,\lr(Y_{t\xi}(r)))\left(\widetilde{Y_{tz\mu}}(r)\right) \left|\widetilde{D_z Y_{tz\mu}}(r)\right|^2 \bigg] \notag\\
	&\quad\qquad +\sigma_1(r)D_z\dr Y_{tx\mu}(r,z)+\sigma_2(r)D_y\dr v_{tx\mu}(r,z)\bigg\} dB(r),\notag \\
	D_z\dr p_{tx\mu}(s,z)=\ &D_x^2 g(Y_{tx\mu}(T),\lr(Y_{t\xi}(T)))D_z \dr Y_{tx\mu}(T,z) \notag \\
	& +\widetilde{\e}\bigg[D_y \frac{d}{d\nu}D_x g(Y_{tx\mu}(T),\lr(Y_{t\xi}(T)))\left(\widetilde{Y_{t\xi}}(T)\right) \widetilde{D_z \bd Y_{t\xi}}(T,z) \notag \\
	&\qquad +D_y \frac{d}{d\nu}D_x g(Y_{tx\mu}(T),\lr(Y_{t\xi}(T)))\left(\widetilde{Y_{tz\mu}}(T)\right)\widetilde{D_z^2 Y_{tz\mu}}(T) \notag\\
	&\qquad +D_y^2 \frac{d}{d\nu}D_x g(Y_{tx\mu}(T),\lr(Y_{t\xi}(T)))\left(\widetilde{Y_{tz\mu}}(T)\right)\left|\widetilde{D_z Y_{tz\mu}}(T)\right|^2 \bigg] \notag \\
	& +\int_s^T \bigg\{ b_1(r) D_z\dr p_{tx\mu}(r,z)+\sigma_1(r) D_z\dr q_{tz\mu}(r,z)\notag \\
	&\ \quad\qquad +D_x^2 f(r,\theta_{tx\mu}(r)) D_z\dr Y_{tx\mu}(r,z)+D_vD_x f(r,\theta_{tx\mu}(r)) D_z \dr v_{tx\mu}(r,z) \notag \\
	&\ \quad\qquad +\widetilde{\e}\bigg[D_y \frac{d}{d\nu}D_x f(r,\theta_{tx\mu}(r))\left(\widetilde{Y_{t\xi}}(s)\right)\widetilde{D_z \bd Y_{t\xi}}(r,z) \notag \\
	&\ \quad\qquad\qquad + D_y \frac{d}{d\nu}D_x f(r,\theta_{tx\mu}(r))\left(\widetilde{Y_{tz\mu}}(r)\right) \widetilde{D_z^2 Y_{tz\mu}}(r)\notag\\
	&\ \quad\qquad\qquad + D_y^2 \frac{d}{d\nu}D_x f(r,\theta_{tx\mu}(r))\left(\widetilde{Y_{tz\mu}}(r)\right) \left|\widetilde{D_z Y_{tz\mu}}(r)\right|^2 \bigg]  \bigg\}dr \notag \\
	&-\int_s^T D_z \dr q_{tx\mu}(r,z)dB(r),\quad (s,z)\in[t,T]\times\br, \label{FB:mu_yy}
\end{align}
where
\begin{align*}
	D_z \dr v_{tx\mu}(s,z):=&D_x\hv (s,\Theta_{tx\mu}(s))D_z \dr Y_{tx\mu}(s,z)\\
	&+\widetilde{\e}\bigg[D_y\frac{d\hv}{d\nu} (s,\Theta_{tx\mu}(s))\left(\widetilde{Y_{t\xi}}(s)\right)\widetilde{D_z \bd Y_{t\xi}}(s,z)\\
	&\qquad +D_y\frac{d\hv}{d\nu} (s,\Theta_{tx\mu}(s))\left(\widetilde{Y_{tz\mu}}(s)\right) \widetilde{D_z^2 Y_{tz\mu}}(s)\\
	&\qquad +D_y^2 \frac{d\hv}{d\nu} (s,\Theta_{tx\mu}(s))\left(\widetilde{Y_{tz\mu}}(s)\right) \left|\widetilde{D_z Y_{tz\mu}}(s)\right|^2 \bigg]\\
	&+D_p\hv (s,\Theta_{tx\mu}(s))D_z \dr p_{tx\mu}(s,z)+D_q\hv (s,\Theta_{tx\mu}(s))D_z \dr q_{tx\mu}(s,z).
\end{align*}
From Theorems~\ref{prop:2} and \ref{prop:8}, Conditions \eqref{optimal_condition}, \eqref{optimal_condition'} and \eqref{optimal_condition''}, the proof of the well-posedness of FBSDEs \eqref{FB:xi_yy} and \eqref{FB:mu_yy}, and the corresponding convergence results for finite differences are similar to those of statements in Subsection~\ref{subsec:1-order}, and we simply omit. From the continuity conditions specified in Assumption (A2') and Theorem~\ref{prop:5}, following a similar approach as in the proof of Theorem~\ref{lem:4}, we obtain the continuity of the derivatives. 

\section{Proof of Statements in Section~\ref{sec:V}}\label{pf:V}

\subsection{Proof of Theorem~\ref{prop:6}}\label{pf_prop6}
We first show the growth condition of $V$. From Assumption (A2), \eqref{intro_4'}, Lemmas~\ref{lem:2} and \ref{lem:5}, we know that
\begin{align*}
	|V(t,x,\mu)| &\le C(L)\e\bigg[ \int_t^T\left( 1+|Y_{tx\mu}(s)|^2+W^2_2(\lr(Y_{t\xi}(s)),\delta_0)+|v_{tx\mu}(s)|^2 \right)ds\\
	&\quad\qquad\qquad+ |Y_{tx\mu}(T)|^2+W^2_2(\lr(Y_{t\xi}(T)),\delta_0)\bigg]\\
	&\le C(L,T)\e\bigg[\sup_{t\le s\le T}\left|\left(Y_{tx\mu}(s),Y_{t\xi}(s),p_{tx\mu}(s)\right)^\top\right|^2 + \int_t^T \left|q_{tx\mu}(s)\right|^2 ds\bigg] \\
	&\le C(L,T,\lambda_v)\left(1+|x|^2+W^2_2(\mu,\delta_0)\right).
\end{align*}
We next prove \eqref{prop6_03}. From the definition of $V$ in \eqref{intro_4}, we have
\begin{equation}\label{prop6_1}
	\begin{split}
		&J_{tx'}\left(v_{tx'\mu};\lr(Y_{t\xi}(s)),t\le s\le T\right)-J_{tx}\left(v_{tx'\mu};\lr(Y_{t\xi}(s)),t\le s\le T\right) \\
		\le\ & V(t,x',\mu)-V(t,x,\mu)\\
		\le\ &  J_{tx'}\left(v_{tx\mu};\lr(Y_{t\xi}(s)),t\le s\le T\right)-J_{tx}\left(v_{tx\mu};\lr(Y_{t\xi}(s)),t\le s\le T\right).
	\end{split}
\end{equation}
From Assumption (A2), we deduce that
\begin{align}
		&J_{tx'}\left(v_{tx\mu};\lr(Y_{t\xi}(s)),t\le s\le T\right)-J_{tx}\left(v_{tx\mu};\lr(Y_{t\xi}(s)),t\le s\le T\right) \notag \\
		=\ & \e\bigg[\int_t^T \left[f\left(s,X_{tx'\mu}^{v_{tx\mu}}(s),\lr(Y_{t\xi}(s)),v_{tx\mu}(s)\right)-f\left(s,Y_{tx\mu}(s),\lr(Y_{t\xi}(s)),v_{tx\mu}(s)\right)\right]ds \notag \\
		&\quad +g\left(X^{v_{tx\mu}}_{tx'\mu}(T),\lr(Y_{t\xi}(T))\right)-g\left(Y_{tx\mu}(T),\lr(Y_{t\xi}(T))\right)\bigg] \notag \\
		\le\ & \e\bigg[\int_t^T \left[D_x f (s,Y_{tx\mu}(s),\lr(Y_{t\xi}(s)),v_{tx\mu}(s))\right]^\top \left(X_{tx'\mu}^{v_{tx\mu}}(s)-Y_{tx\mu}(s)\right) ds \notag \\
		&\qquad +\left[D_x g (Y_{tx\mu}(T),\lr(Y_{t\xi}(T)))\right]^\top \left(X_{tx'\mu}^{v_{tx\mu}}(T)-Y_{tx\mu}(T)\right) \bigg] \notag \\
		& +C(L,T)\e\bigg[\sup_{t\le s\le T}\left|X_{tx'\mu}^{v_{tx\mu}}(s)-Y_{tx\mu}(s)\right|^2\bigg], \label{prop6_2}
\end{align}
where $X_{tx'\mu}^{v_{tx\mu}}$ is the controlled state process with the initial $x'$ and control $v_{tx\mu}$:
\begin{align*}
	X_{tx'\mu}^{v_{tx\mu}}(s)=x'&+\int_t^sb\left(r,X_{tx'\mu}^{v_{tx\mu}}(r),\lr(Y_{t\xi}(r)),v_{tx\mu}(r)\right)dr\\
	&+\int_t^s\sigma\left(r,X_{tx'\mu}^{v_{tx\mu}}(r),\lr(Y_{t\xi}(r)),v_{tx\mu}(r)\right)dB(r),\quad s\in[t,T].
\end{align*}
Note the fact that $Y_{tx\mu}=X_{tx\mu}^{v_{tx\mu}}$, then by using standard arguments in the context of SDEs similar as Lemma~\ref{lem:1}, we have
\begin{equation}\label{prop6_3}
	\e\bigg[\sup_{t\le s\le T}\left|X_{tx'\mu}^{v_{tx\mu}}(s)-Y_{tx\mu}(s)\right|^2\bigg]=\e\bigg[\sup_{t\le s\le T}\left|X_{tx'\mu}^{v_{tx\mu}}(s)-X_{tx\mu}^{v_{tx\mu}}(s)\right|^2\bigg]\le C(L,T)|x'-x|^2.
\end{equation}
By applying It\^o's formula to $(p_{tx\mu}(s))^\top \left(X_{tx'\mu}^{v_{tx\mu}}(s)-Y_{tx\mu}(s)\right)$, we can deduce that
\begin{equation*}\label{prop6_4}
	\begin{split}
		&\e\bigg[\int_t^T \left[D_x f  \left(s,Y_{tx\mu}(s),\lr(Y_{t\xi}(s)),v_{tx\mu}(s)\right)\right]^\top \left(X_{tx'\mu}^{v_{tx\mu}}(s)-Y_{tx\mu}(s)\right)  ds\\
		&\quad +\left[D_x g (Y_{tx\mu}(T),\lr(Y_{t\xi}(T))) \right]^\top \left(X_{tx'\mu}^{v_{tx\mu}}(T)-Y_{tx\mu}(T)\right) \bigg]\\
		=\ & (p_{tx\mu}(t))^\top (x'-x).
	\end{split}
\end{equation*}
Substituting \eqref{prop6_3} and the last equation back to \eqref{prop6_2}, we obtain that
\begin{equation}\label{prop6_5}
	\begin{split}
		&J_{tx'}\left(v_{tx\mu};\lr(Y_{t\xi}(s)),t\le s\le T\right)-J_{tx}\left(v_{tx\mu};\lr(Y_{t\xi}(s)),t\le s\le T\right) \\
		\le\ &( p_{tx\mu}(t))^\top  (x'-x) +C(L,T)|x'-x|^2.
	\end{split}
\end{equation}
Similar as \eqref{prop6_5} and \cite[Lemma 5.1]{AB10}, we can also obtain the lower bound
\begin{equation}\label{prop6_6}
	\begin{split}
		&J_{tx'}\left(v_{tx'\mu};\lr(Y_{t\xi}(s)),t\le s\le T\right)-J_{tx}\left(v_{tx'\mu};\lr(Y_{t\xi}(s)),t\le s\le T\right) \\
		\geq\ & (p_{tx'\mu}(t))^\top  (x'-x) -C(L,T)|x'-x|^2.
	\end{split}
\end{equation}
From \eqref{lem5_2}, we have
\begin{equation}\label{prop6_7}
	|p_{tx'\mu}(t)-p_{tx\mu}(t)|\le C(L,T,\lambda_v)|x'-x|.
\end{equation}
Combining \eqref{prop6_1}, \eqref{prop6_5}, \eqref{prop6_6} and \eqref{prop6_7}, we deduce that
\begin{equation}
	\left|V(t,x',\mu)-V(t,x,\mu)-(p_{tx\mu}(t))^\top (x'-x)\right|\le C(L,T,\lambda_v)|x'-x|^2,
\end{equation} 
from which we obtain $D_x V(t,x,\mu)  =p_{tx\mu}(t)$. Then, \eqref{prop6_03} is a consequence of Theorem~\ref{prop:2}. Estimate \eqref{prop6_04} is another immediate sequel of \eqref{lem5_1} and \eqref{thm2_1}, and Estimate \eqref{prop6_05} is a result after of \eqref{lem5_2}. The continuity of $D_x^2 V$ in $(x,\mu)$ is a consequence of Theorem~\ref{prop:2}.    

\subsection{Proof of Theorem~\ref{prop:7}}\label{pf_prop7}
For any $\xi,\eta\in L_{\f_t}^2$ both independent of the Brownian motion $B_\cdot$ such that $\lr(\xi)=\mu$, from \eqref{intro_4'}, Theorem~\ref{lem:4} and Lemma~\ref{prop:3}, we have
\begin{align*}
	&\lim_{\epsilon\to0} \frac{1}{\epsilon}[V(t,x,\lr(\xi+\epsilon\eta))-V(t,x,\mu)]\\
	=\ & \int_t^T \e\bigg\{ \left[D_x f (s,\theta_{tx\mu}(s))\right]^\top \dr_\eta Y_{tx\xi}(s)+\left[D_v f (s,\theta_{tx\mu}(s)) \right]^\top \dr_\eta v_{tx\xi}(s)\\
	&\quad\qquad +\widetilde{\e}\bigg[\left[D_y \frac{df}{d\nu}  (s,\theta_{tx\mu}(s))\left(\widetilde{Y_{t\xi}}(s)\right)\right]^\top \widetilde{\dr_\eta Y_{t\xi}}(s)  \bigg]   \bigg\} ds\\
	& + \e\bigg\{ \left[D_x g (Y_{tx\mu}(T),\lr(Y_{t\xi}(T))) \right]^\top \dr_\eta Y_{tx\xi}(T)\\
	&\qquad +\widetilde{\e}\left[\left[D_y \frac{dg}{d\nu} (Y_{tx\mu}(T),\lr(Y_{t\xi}(T)))\left(\widetilde{Y_{t\xi}}(T)\right)\right]^\top \widetilde{\dr_\eta Y_{t\xi}}(T) \right] \bigg\}.
\end{align*}
where $\left(\widetilde{Y_{t\xi}}(s),\widetilde{\dr_\eta Y_{t\xi}}(s)\right)$ is an independent copy of $\left({Y_{t\xi}}(s),\dr_\eta Y_{t\xi}(s)\right)$ for $s\in[t,T]$. Then, from Lemma~\ref{prop:4} and Theorem~\ref{prop:5}, we deduce that
\begin{align*}
	&\lim_{\epsilon\to0} \frac{1}{\epsilon}[V(t,x,\lr(\xi+\epsilon\eta))-V(t,x,\mu)]\\
	=\ & \int_t^T \e\bigg\{ \left[D_x f (s,\theta_{tx\mu}(s))\right]^\top \widehat{\e}\left[\dr Y_{tx\mu}\left(s,\widehat{\xi}\right) \widehat{\eta}\right]+ (D_v f)^\top  (s,\theta_{tx\mu}(s)) \widehat{\e}\left[\dr v_{tx\mu}\left(s,\widehat{\xi}\right)\widehat{\eta}\right] \\
	&\quad\qquad +\widetilde{\e}\left[\left[D_y \frac{df}{d\nu}  (s,\theta_{tx\mu}(s))\left(\widetilde{Y_{t\xi}}(s)\right)\right]^\top\left(\left.\widetilde{D_yY_{ty\mu}}(s)\right|_{y=\widetilde{\xi}}\ \widetilde{\eta}+\widehat{\e}\left[\widetilde{\bd Y_{t\xi}}\left(s,\widehat{\xi}\right) \widehat{\eta}\right]\right) \right]  \bigg\}ds\\
	& + \e\bigg\{\left[D_x g (Y_{tx\mu}(T),\lr(Y_{t\xi}(T)))\right]^\top \widehat{\e}\left[\bd Y_{t\xi}\left(T,\widehat{\xi}\right) \widehat{\eta}\right]\\
	&\qquad +\widetilde{\e}\left[\left[D_y \frac{dg}{d\nu} (Y_{tx\mu}(T),\lr(Y_{t\xi}(T)))\left(\widetilde{Y_{t\xi}}(T)\right)\right]^\top \left(\left.\widetilde{D_yY_{ty\mu}}(T)\right|_{y=\widetilde{\xi}}\ \widetilde{\eta}+\widehat{\e}[\widetilde{\bd Y_{t\xi}}\left(T,\widehat{\xi}\right) \widehat{\eta}]\right) \right]\bigg\},
\end{align*} 
where $\left(\widetilde{\xi},\widetilde{\eta},\widetilde{D_y Y_{ty\mu}}(s),\widetilde{\bd Y_{t\xi}}(s,y)\right)$ is an independent copy of $\left(\xi,{\eta},{D_y Y_{ty\mu}}(s),{\bd Y_{t\xi}}(s,y)\right)$, and $(\widehat{\xi},\widehat{\eta})$ is another independent copy of $(\xi,\eta)$. Taking into account that $(\widehat{\xi},\widehat{\eta})$ are independent of $(\xi,\eta)$ and $(\widetilde{\xi},\widetilde{\eta})$, and of the same law as $(\widetilde{\xi},\widetilde{\eta})$, we see that, for instance,
\begin{align*}
	&\e\bigg\{ \widetilde{\e}\left[\left[D_y \frac{dg}{d\nu} (Y_{tx\mu}(T),\lr(Y_{t\xi}(T)))\left(\widetilde{Y_{t\xi}}(T)\right)\right]^\top\left(\left.\widetilde{D_yY_{ty\mu}}(T)\right|_{y=\widetilde{\xi}}\ \widetilde{\eta}+\widehat{\e}[\widetilde{\bd Y_{t\xi}}\left(T,\widehat{\xi}\right) \widehat{\eta}]\right) \right]\bigg\}\\
	=\ &\widehat{\e}\bigg\{  \e\widetilde{\e}\bigg[\left[D_y \frac{dg}{d\nu} (Y_{tx\mu}(T),\lr(Y_{t\xi}(T)))\left(\left.\widetilde{Y_{ty\mu}}(T)\right|_{y=\widehat{\xi}}\right)\right]^\top \left.\widetilde{D_yY_{ty\mu}}(T)\right|_{y=\widehat{\xi}}  \\
	&\ \quad\qquad +\left[ D_y \frac{dg}{d\nu} (Y_{tx\mu}(T),\lr(Y_{t\xi}(T)))\left(\widetilde{Y_{t\xi}}(T)\right)\right]^\top \widetilde{\bd Y_{t\xi}}\left(T,\widehat{\xi}\right) \bigg]  \widehat{\eta}\bigg\}.
\end{align*}
Therefore, we have
\begin{align*}
	&\lim_{\epsilon\to0} \frac{1}{\epsilon}[V(t,x,\lr(\xi+\epsilon\eta))-V(t,x,\mu)]\\
	=\ & \widehat{\e}\bigg\{ \bigg[ \int_t^T  \e \bigg( \left[D_x f (s,\theta_{tx\mu}(s))\right]^\top \dr Y_{tx\mu}\left(s,\widehat{\xi}\right)+ (D_v f)^\top  (s,\theta_{tx\mu}(s)) \dr v_{tx\mu}\left(s,\widehat{\xi}\right)\\
	&\quad\qquad\qquad +\widetilde{\e}\bigg[\left[D_y \frac{df}{d\nu} (s,\theta_{tx\mu}(s))\left(\left.\widetilde{Y_{ty\mu}}(s)\right|_{y=\widehat{\xi}}\right)\right]^\top \left.\widetilde{D_yY_{ty\mu}}(s)\right|_{y=\widehat{\xi}} \bigg] \\
	&\quad\qquad\qquad\qquad +\left[D_y \frac{df}{d\nu} (s,\theta_{tx\mu}(s))\left(\widetilde{Y_{t\xi}}(s)\right)\right]^\top \widetilde{\bd Y_{t\xi}}\left(s,\widehat{\xi}\right) \bigg) ds  \\
	&\qquad + \e\bigg(\left[D_x g  (Y_{tx\mu}(T),\lr(Y_{t\xi}(T))) \right]^\top \dr Y_{tx\mu}\left(T,\widehat{\xi}\right)\\
	&\qquad\qquad +\widetilde{\e}\bigg[\left[D_y \frac{dg}{d\nu} (Y_{tx\mu}(T),\lr(Y_{t\xi}(T)))\left(\left.\widetilde{Y_{ty\mu}}(T)\right|_{y=\widehat{\xi}}\right)\right]^\top \left.\widetilde{D_yY_{ty\mu}}(T)\right|_{y=\widehat{\xi}}  \\
	&\qquad\qquad\qquad +\left[ D_y \frac{dg}{d\nu} (Y_{tx\mu}(T),\lr(Y_{t\xi}(T)))\left(\widetilde{Y_{t\xi}}(T)\right)\right]^\top \widetilde{\bd Y_{t\xi}}\left(T,\widehat{\xi}\right) \bigg] \bigg)  \bigg] \widehat{\eta}\bigg\},
\end{align*}
where $\widetilde{Y_{ty\mu}}(s)$ is an independent copy of $Y_{ty\mu}(s)$. Since the choices of $\xi$ and $\eta$ are arbitrary, we can then deduce \eqref{prop7_01}. Then, \eqref{prop7_02} is a direct consequence of \eqref{prop7_01} and Theorem~\ref{prop:9}. And Estimate \eqref{prop7_03} is a direct consequence of Estimates \eqref{thm2_1}, \eqref{prop5_01} and \eqref{prop5_02}, and Theorems~\ref{prop:8} and \ref{prop:9}.    

\subsection{Proof of Theorem~\ref{prop:10}}\label{pf_prop10}

Note that when $\sigma_2(s)=0$, so the map $\sigma$ do not depend on $v$ and the map $\hat{v}$ do not depend on $q$. By Cauchy-Schwarz inequality and \eqref{lem2_1}, we have
\begin{align}
		&\e\left[|Y_{t\xi}(t')-\xi|^2\right] \notag \\
		=\ & \e\Bigg[\left| \int_t^{t'}b(s,\theta_{t\xi}(s)) ds+ \int_t^{t'}\sigma(s,Y_{t\xi}(s),\lr(Y_{t\xi}(s))) dB(s) \right|^2\Bigg] \notag \\
		\le\ & C(L,T) \int_t^{t'}\e\left[1+|Y_{t\xi}(s)|^2+|\hv(s,Y_{t\xi}(s),\lr(Y_{t\xi}(s)),p_{t\xi}(s))|^2\right]ds \notag \\
		\le\ & C(L,T) \int_t^{t'}\e\bigg[1+\sup_{t\le s\le T}|Y_{t\xi}(s)|^2+\sup_{t\le s\le T}|p_{t\xi}(s)|^2\bigg]ds \notag \\
		\le\ & C(L,T,\lambda_v)\left(1+W_2^2(\mu,\delta_0)\right)|t'-t|.  \label{prop10_11}
\end{align}
In a similar manner, from \eqref{lem2_1} and \eqref{lem5_1}, we can deduce that 
\begin{equation}\label{prop10_12}
	\begin{split}
		&\e\left[|Y_{tx\mu}(t')-x|^2\right] \le C(L,T,\lambda_v)\left(1+|x|^2+W_2^2(\mu,\delta_0)\right) |t'-t|.
	\end{split}		
\end{equation}
Note that
\begin{align}\label{prop10_13}
	p_{t\xi}(t')-p_{t\xi}(t)=p_{t\xi}(t')-\e\left[p_{t\xi}(t')|\xi\right]-\int_t^{t'}\e\left[D_x H(s,\Theta_{t\xi}(s))|\xi\right]ds.
\end{align}  
From Cauchy-Schwarz inequality and \eqref{lem2_1}, we deduce that
\begin{align*}
	&\e\Bigg[\left|\int_t^{t'}\e\left[D_x H(s,\Theta_{t\xi}(s))|\xi\right]ds\right|^2\Bigg]\\
	\le\ & |t'-t| \e\bigg[\int_t^{t'} \left| D_x H(s,\Theta_{t\xi}(s))\right|^2 ds\bigg]\\
	\le\ &  |t'-t| C(L,T) \e\bigg[\int_t^{t'} \left(|Y_{t\xi}(r)|^2 + |v_{t\xi}(r)|^2+|p_{t\xi}(r)|^2+|q_{t\xi}(r)|^2 \right) ds\bigg]\\
	\le\ & C(L,T,\lambda_v)\left(1+W_2^2(\mu,\delta_0)\right) |t'-t| .
\end{align*}
Then, from \eqref{prop10_13}, we have
\begin{equation}\label{prop10_14}
	\begin{split}
		&\e\left[|p_{t\xi}(t')-p_{t\xi}(t)|^2\right]\le C(L,T,\lambda_v)\left(1+W_2^2(\mu,\delta_0)\right)|t'-t|.
	\end{split}		
\end{equation}
Similarly, we can deduce that
\begin{equation}\label{prop10_15}
	\begin{split}
		&\e\left[|p_{tx\mu}(t')-p_{tx\mu}(t)|^2\right]\le C(L,T,\lambda_v)\left(1+|x|^2+W_2^2(\mu,\delta_0)\right) |t'-t|.
	\end{split}		
\end{equation}
By dynamic programming principle, for any $\epsilon\in[0,T-t]$,
\begin{equation*}
	V(t,x,\mu)=\e\left[\int_t^{t+\epsilon}f(s,\theta_{tx\mu}(s))ds\right]+V\left(t+\epsilon,Y_{tx\mu}(t+\epsilon),\lr(Y_{t\xi}(t+\epsilon))\right),
\end{equation*}
so we have
\begin{align}
		\frac{1}{\epsilon}\left[V(t+\epsilon,x,\mu)-V(t,x,\mu)\right] \notag =\ & \frac{1}{\epsilon}\e[V(t+\epsilon,x,\mu)-V(t,x,\mu)] \notag \\
		=\ & \frac{1}{\epsilon} \e\left[V(t+\epsilon,x,\mu)-V(t+\epsilon,Y_{tx\mu}(t+\epsilon),\lr(Y_{t\xi}(t+\epsilon)))\right] \notag \\
		& -\frac{1}{\epsilon} \e\left[\int_t^{t+\epsilon}f(s,\theta_{tx\mu}(s))ds\right]. \label{prop10_1}
\end{align}	
From Theorems~\ref{prop:6} and \ref{prop:7}, Estimates \eqref{prop10_11} and \eqref{prop10_12} and It\^o's formula for measure-dependent functionals (see \cite[Theorem 7.1]{BR} and also \cite{AB5,AB9'}), we deduce that
\begin{equation}\label{prop10_2}
	\begin{split}
		&\lim_{\epsilon\to0}\frac{1}{\epsilon}\e\left[V(t+\epsilon,x,\mu)-V(t+\epsilon,Y_{tx\mu}(t+\epsilon),\lr(Y_{t\xi}(t+\epsilon)))\right]\\
		=\ & - b (t,x,\mu,v_{tx\mu}(t))^\top D_x V(t,x,\mu)-\frac{1}{2}\text{Tr}\left(\sigma (t,x,\mu)^\top D_x^2 V(t,x,\mu)\sigma(t,x,\mu)\right)\\
		& -{\e}\left[b (t,{\xi},\mu,{v_{t\xi}}(t))^\top D_y\frac{dV}{d\nu}(t,x,\mu)({\xi})+\frac{1}{2}\text{Tr}\left(\sigma (t,{\xi},\mu)^\top D_y^2\frac{dV}{d\nu}(t,x,\mu)({\xi})\sigma(t,{\xi},\mu)\right) \right].
	\end{split}
\end{equation}
From \eqref{prop10_11}, \eqref{prop10_12}, \eqref{prop10_14} and \eqref{prop10_15}, we have
\begin{equation*}\label{prop10_3}
	\lim_{\epsilon\to0}\frac{1}{\epsilon}\e\left[\int_t^{t+\epsilon}f(s,Y_{tx\mu}(s),\lr(Y_{t\xi}(s)),v_{tx\mu}(s))ds\right]=f(t,x,\mu,v_{tx\mu}(t)).
\end{equation*}
Substituting \eqref{prop10_2} and the last equation back to \eqref{prop10_1}, from the definition of $v_{tx\mu}(t)$ and $v_{t\xi}(t)$, \eqref{prop6_03} and \eqref{rk_1}, we have
\begin{align*}%\label{prop10_4}
		&\lim_{\epsilon\to0}\frac{1}{\epsilon}\left[V(t+\epsilon,x,\mu)-V(t,x,\mu)\right]\\
		=\ & -b(t,x,\mu,v_{tx\mu}(t))^\top D_x V(t,x,\mu) -\frac{1}{2}\text{Tr}\left[\sigma (t,x,\mu)^\top D_x^2 V(t,x,\mu)\sigma(t,x,\mu)\right]-f(t,x,\mu,v_{tx\mu}(t))\\
		& -{\e}\left[b(t,{\xi},\mu,{v_{t\xi}}(t))^\top D_y\frac{dV}{d\nu}(t,x,\mu)({\xi})+\frac{1}{2}\text{Tr}\left(\sigma (t,{\xi},\mu)^\top D_y^2\frac{dV}{d\nu}(t,x,\mu)({\xi})\sigma(t,{\xi},\mu)\right) \right]\\
		=\ & -b \left(t,x,\mu,\hv(t,x,\mu,D_x V(t,x,\mu))\right)^\top D_x V(t,x,\mu) -\frac{1}{2}\text{Tr}\left(\sigma (t,x,\mu)^\top D_x^2 V(t,x,\mu)\sigma(t,x,\mu)\right)\\
		&-f\left(t,x,\mu,\hv(t,x,\mu,D_x V(t,x,\mu))\right)\\
		& -{\e}\bigg[b \left(t,{\xi},\mu,\hv(t,\xi,\mu,D_x V(t,\xi,\mu))\right)^\top D_y\frac{dV}{d\nu}(t,x,\mu)({\xi})\\
		&\qquad +\frac{1}{2}\text{Tr}\left(\sigma (t,{\xi},\mu)^\top D_y^2\frac{dV}{d\nu}(t,x,\mu)({\xi})\sigma(t,{\xi},\mu)\right) \bigg]\\
		=\ & -H\left(t,x,\mu,D_x V(t,x,\mu),\frac{1}{2}D_x^2 V(t,x,\mu)\sigma(t,x,\mu)\right)\\
		& -\int_\brn \bigg[ b \left(t,y,\mu,\hv(t,y,\mu,D_x V(t,y,\mu))\right)^\top D_y\frac{dV}{d\nu}(t,x,\mu)({y})\\
		&\quad\qquad +\frac{1}{2}\text{Tr}\left( \sigma (t,y,\mu)^\top D_y^2\frac{dV}{d\nu}(t,x,\mu)({y})\sigma(t,y,\mu)\right)\bigg] d\mu(y),
\end{align*}
from which we deduce \eqref{prop10_01}.	

\section{Proof of Theorem~\ref{thm:1}}\label{pf_thm1}

The existence result is a direct consequence of Theorems~\ref{prop:6}, \ref{prop:7} and \ref{prop:10}. We next aim to prove the uniqueness. Let ${U}$ be another classical solution of the master equation \eqref{master}. For any initial $(t,\mu)$, we choose any $\xi\sim\mu$, and define the process $\bar{Y}_{t\xi}$ as
	\begin{equation}\label{SDE_bar_Y}
		\begin{aligned}
			\bar{Y}_{t\xi}(s)=\xi&+\int_t^s b\left(r, \bar{Y}_{t\xi}(r),\lr(\bar{Y}_{t\xi}(r)), \hv(r, \bar{Y}_{t\xi}(r),\lr(\bar{Y}_{t\xi}(r)),D_x U(r, \bar{Y}_{t\xi}(r),\lr(\bar{Y}_{t\xi}(r)))) \right)dr\\
			&+\int_t^s \sigma \left(r, \bar{Y}_{t\xi}(r),\lr(\bar{Y}_{t\xi}(r))\right)dB(r),\quad s\in[t,T].
		\end{aligned}
	\end{equation}
	The well-posedness of SDE \eqref{SDE_bar_Y} follows from the regularity of the functional $U$ and standard argument for SDEs, and we denote by $\bar{v}_{t\xi}(s):= \hv(s, \bar{Y}_{t\xi}(s),\lr(\bar{Y}_{t\xi}(s)),D_x U(s, \bar{Y}_{t\xi}(s),\lr(\bar{Y}_{t\xi}(s))))$. Then, for any initial $x$, we consider the control problem $J_{tx}(\cdot;\lr(\bar{Y}_{t\xi}(s)),t\le s\le T)$: for an admissible control $v(\cdot)\in\lr_{\f}^2(t,T)$, the controlled state process is 
	\begin{align*}
		&X_{tx\mu}^v(s)=x+\int_t^s b\left(r,X^v_{tx\mu}(r), \lr(\bar{Y}_{t\xi}(r)),v(r)\right)dr+\int_t^s \sigma\left(r,X^v_{tx\mu}(r), \lr(\bar{Y}_{t\xi}(r))\right)dB(r),\quad s\in[t,T],
	\end{align*}	 
	and the cost functional is 
	\begin{align*}
		J_{tx}(v;\lr(\bar{Y}_{t\xi}(s)),t\le s\le T)=\e\left[\int_t^T f\left(s,X^v_{tx\mu}(s), \lr(\bar{Y}_{t\xi}(s)),v(s)\right) ds +g\left(X^v_{tx\mu}(T), \lr(\bar{Y}_{t\xi}(T))\right) \right].
	\end{align*}
	From It\^o's formula for measure-dependent functionals (see \cite[Theorem 7.1]{BR} and also \cite{AB5}), we deduce that
	\begin{align*}
		&\e\left[U\left(s,X^v_{tx\mu}(s),\lr(\bar{Y}_{t\xi}(s))\right)-U(t,x,\mu)\right]\\
		=\ & \int_t^s \e \bigg\{ \dd_t U\left(r,X^v_{tx\mu}(r),\lr(\bar{Y}_{t\xi}(r))\right)+b \left(r,X^v_{tx\mu}(r),\lr(\bar{Y}_{t\xi}(r)),v(r)\right)^\top D_x U\left(r,X^v_{tx\mu}(r),\lr(\bar{Y}_{t\xi}(r))\right)\\
		&\quad\qquad +\frac{1}{2}\text{Tr}\left[ D_x^2 U\left(r,X^v_{tx\mu}(r),\lr(\bar{Y}_{t\xi}(r))\right)\left(\sigma\sigma^\top\right) \left(r,X^v_{tx\mu}(r),\lr(\bar{Y}_{t\xi}(r))\right)\right]\\
		&\quad\qquad+\widetilde{\e}\bigg[b \left(r,\widetilde{\bar{Y}_{t\xi}(r)},\lr(\bar{Y}_{t\xi}(r)),\widetilde{\bar{v}_{t\xi}}(r)\right)^\top D_y\frac{dU}{d\nu}\left(r,X^v_{tx\mu}(r),\lr(\bar{Y}_{t\xi}(r))\right)\left(\widetilde{\bar{Y}_{t\xi}(r)}\right) \\
		&\quad\qquad\qquad +\frac{1}{2}\text{Tr}\left[ D_y^2\frac{dU}{d\nu}\left(r,X^v_{tx\mu}(r),\lr(\bar{Y}_{t\xi}(r))\right)\left(\widetilde{\bar{Y}_{t\xi}(r)}\right)\left(\sigma\sigma^\top\right) \left(r,\widetilde{\bar{Y}_{t\xi}(r)},\lr(\bar{Y}_{t\xi}(r))\right)\right]\bigg]\bigg\}dr\\
		=\ & \int_t^s \e\bigg\{\dd_t U\left(r,X^v_{tx\mu}(r),\lr(\bar{Y}_{t\xi}(r))\right)-f\left(r,X^v_{tx\mu}(r),\lr(\bar{Y}_{t\xi}(r)),v(r)\right)\\
		&\quad\qquad +L\Big(r,X^v_{tx\mu}(r),\lr(\bar{Y}_{t\xi}(r)),v(r),D_x U\left(r,X^v_{tx\mu}(r),\lr(\bar{Y}_{t\xi}(r))\right),\\
		&\qquad\qquad\qquad \frac{1}{2}D_x^2 U\left(r,X^v_{tx\mu}(r),\lr(\bar{Y}_{t\xi}(r))\right) \sigma\left(r,X^v_{tx\mu}(r),\lr(\bar{Y}_{t\xi}(r))\right)\Big)\\
		&\quad\qquad +\widetilde{\e}\bigg[b \left(r,\widetilde{\bar{Y}_{t\xi}(r)},\lr(\bar{Y}_{t\xi}(r)),\widetilde{\bar{v}_{t\xi}}(r)\right)^\top D_y\frac{dU}{d\nu}\left(r,X^v_{tx\mu}(r),\lr(\bar{Y}_{t\xi}(r))\right)\left(\widetilde{\bar{Y}_{t\xi}(r)}\right)\\
		&\quad\qquad\qquad +\frac{1}{2}\text{Tr}\left[ D_y^2\frac{dU}{d\nu}\left(r,X^v_{tx\mu}(r),\lr(\bar{Y}_{t\xi}(r))\right)\left(\widetilde{\bar{Y}_{t\xi}(r)}\right)\left(\sigma\sigma^\top\right) \left(r,\widetilde{\bar{Y}_{t\xi}(r)},\lr(\bar{Y}_{t\xi}(r))\right)\right]\bigg] \bigg\}dr,
	\end{align*}	
	where $\left(\widetilde{\bar{Y}_{t\xi}}(r),\widetilde{\bar{v}_{t\xi}}(r)\right)$ is an independent copy of $({\bar{Y}_{t\xi}}(r),{\bar{v}_{t\xi}}(r))$. Since $U$ satisfies Equation \eqref{master}, from the last equation, we know that
	\begin{equation*}\label{thm1_1}
		\begin{split}
			&\e\left[g\left(X^v_{tx\mu}(T),\lr(\bar{Y}_{t\xi}(T))\right)-U(t,x,\mu)\right]\\
			=\ & \int_t^T \e\Big\{-f\left(r,X^v_{tx\mu}(r),\lr(\bar{Y}_{t\xi}(r)),v(r)\right)\\
			&\quad\qquad +L\Big(r,X^v_{tx\mu}(r),\lr(\bar{Y}_{t\xi}(r)),v(r),D_x U\left(r,X^v_{tx\mu}(r),\lr(\bar{Y}_{t\xi}(r))\right),\\
			&\qquad\qquad\qquad \frac{1}{2}D_x^2 U\left(r,X^v_{tx\mu}(r),\lr(\bar{Y}_{t\xi}(r))\right) \sigma\left(r,X^v_{tx\mu}(r),\lr(\bar{Y}_{t\xi}(r))\right)\Big)\\
			&\quad\qquad -H\Big(r,X^v_{tx\mu}(r),\lr(\bar{Y}_{t\xi}(r)),D_x U\left(r,X^v_{tx\mu}(r),\lr(\bar{Y}_{t\xi}(r))\right),\\
			&\qquad\qquad\qquad \frac{1}{2}D_x^2 U\left(r,X^v_{tx\mu}(r),\lr(\bar{Y}_{t\xi}(r))\right)\sigma\left(r,X^v_{tx\mu}(r),\lr(\bar{Y}_{t\xi}(r))\right)\Big)\Big\}dr,
		\end{split}
	\end{equation*}
	therefore, from the very definition of the Hamiltonian $H$, we have
	\begin{equation}\label{pf:uni_2}
		\begin{split}
			J_{tx}(v(\cdot);\lr(\bar{Y}_{t\xi}(s),t\le s\le T))&\geq U(t,x,\mu).
		\end{split}
	\end{equation}
	If we set $v(s)=\bar{v}_{tx\mu}(s):=\hv\left(s, \bar{Y}_{tx\mu}(s),\lr(\bar{Y}_{t\xi}(s)),D_x U\left(s,\bar{Y}_{tx\mu}(s),\lr(\bar{Y}_{t\xi}(s))\right) \right)$, where $\bar{Y}_{tx\mu}$ is the solution of the following SDE
	\begin{equation*}
		\begin{aligned}
			\bar{Y}_{tx\mu}(s)=x&+\int_t^s b\left(r, \bar{Y}_{tx\mu}(r),\lr(\bar{Y}_{t\xi}(r)), \hv(r, \bar{Y}_{tx\mu}(r),\lr(\bar{Y}_{t\xi}(r)),D_x U(r, \bar{Y}_{tx\mu}(r),\lr(\bar{Y}_{t\xi}(r)))) \right)dr\\
			&+ \int_t^s \sigma \left(r, \bar{Y}_{tx\mu}(r),\lr(\bar{Y}_{t\xi}(r))\right)dB(r),\quad s\in[t,T],
		\end{aligned}
	\end{equation*}    
	then we see that 
	\begin{align}\label{pf:uni_3}
		{U}(t,x,\mu)= J_{tx}(\bar{v}_{tx\mu};\lr(\bar{Y}_{t\xi}(s)),t\le s\le T).
	\end{align}
	We now define the processes 
	\begin{equation}\label{pf:uni_1}
		\begin{aligned}
			\bar{p}_{t\xi}(s):=\ & D_x U\left(s,\bar{Y}_{t\xi}(s),\lr(\bar{Y}_{t\xi}(s))\right),\\
			\bar{q}_{t\xi}(s):=\ & D_x^2 U\left(s,\bar{Y}_{t\xi}(s),\lr(\bar{Y}_{t\xi}(s))\right)\sigma \left(s,\bar{Y}_{t\xi}(s),\lr(\bar{Y}_{t\xi}(s))\right),
		\end{aligned}
	\end{equation}
	then, by using  It\^o's lemma and the master equation  \eqref{master} for $U$, we know that the processes $(\bar{p}_{t\xi},\bar{q}_{t\xi})$ satisfy the following BSDE:
	\begin{equation}\label{BSDE_bar_pq}
		\begin{aligned}
			\bar{p}_{t\xi}(T)=\ & D_x g\left(\bar{Y}_{t\xi}(T),\lr(\bar{Y}_{t\xi}(T))\right)-\int_s^T \bar{q}_{t\xi}(r)dB(r)\\
			&+\int_s^T \left(D_x H\left(r, \bar{Y}_{t\xi}(r),\lr(\bar{Y}_{t\xi}(r)), \bar{p}_{t\xi}(r), \bar{q}_{t\xi}(r) \right)\right) dr,\quad s\in[t,T].
		\end{aligned}
	\end{equation}
	From \eqref{SDE_bar_Y}, \eqref{pf:uni_1} and \eqref{BSDE_bar_pq}, we know that the processes $(\bar{Y}_{t\xi}, \bar{p}_{t\xi},\bar{q}_{t\xi})$ satisfy the FBSDEs \eqref{intro_2}. Then, from the uniqueness result of FBSDEs \eqref{intro_2}, we know that $(\bar{Y}_{t\xi}, \bar{p}_{t\xi},\bar{q}_{t\xi})=({Y}_{t\xi}, {p}_{t\xi},{q}_{t\xi})$, and therefore, $\lr(\bar{Y}_{t\xi}(s))=\lr({Y}_{t\xi}(s))$. Similar as above, if we define 		
	\begin{align*}
		\bar{p}_{tx\mu}(s):=\ & D_x U\left(s, \bar{Y}_{tx\mu}(s),\lr(Y_{t\xi}(s))\right), \\
		\bar{q}_{tx\mu}(s):=\ & D_x^2 U\left(s,\bar{Y}_{tx\mu}(s),\lr(Y_{t\xi}(s))\right)\sigma \left(s,\bar{Y}_{tx\mu}(s),\lr(Y_{t\xi}(s))\right),
	\end{align*}
	then we can show that $(\bar{Y}_{tx\mu}, \bar{p}_{tx\mu},\bar{q}_{tx\mu})=({Y}_{tx\mu}, {p}_{tx\mu},{q}_{tx\mu})$, and therefore $\bar{v}_{tx\mu}=v_{tx\mu}$. From \eqref{pf:uni_2} and \eqref{pf:uni_3} and the definition of $V$ in  \eqref{intro_4}, we know that $U$ coincides with the value functional $V$.

\end{document}